%% file: gpdps.tex
\pdfoutput=1
\documentclass[english]{jnsao} 
\usepackage[utf8]{inputenc}
\usepackage{csquotes}
\usepackage{enumitem}
\usepackage{mathtools}

\usepackage{tikz,pgfplots}
\usepgfplotslibrary{colorbrewer}
\pgfplotsset{compat=newest}
\pgfplotsset{plot coordinates/math parser=false}
\usetikzlibrary{plotmarks}
\pgfplotsset{every tick label/.append style={font=\scriptsize}}

\theoremstyle{definition}
\newtheorem{algorithm}{Algorithm}
\numberwithin{algorithm}{section}
\newtheorem{assumption}[definition]{Assumption}

\usepackage[nameinlink,capitalise]{cleveref}
\crefname{assumption}{Assumption}{Assumptions}

\usepackage{mdframed}
\mdfdefinestyle{examplebg}{
    backgroundcolor=black!7,%
    hidealllines=true,%
    innertopmargin=.2em,%
    innerbottommargin=.7em,%
    innerleftmargin=.7em,%
    innerrightmargin=.7em,%
}

\surroundwithmdframed[style=examplebg]{example}
\surroundwithmdframed[style=examplebg]{algorithm}

\def\proofstep#1{\vspace{-1ex}\paragraph{#1}}

\usepackage{booktabs}
\usepackage[scientific-notation=true,round-mode=places,exponent-product=\cdot]{siunitx}
\input{symbols}

\def\thetitle{Primal–dual proximal splitting and generalized conjugation in non-smooth non-convex optimization}
\shortauthor{Clason, Mazurenko, Valkonen}

\manuscripteprinttype{arxiv}
\manuscripteprint{1901.02746v4}
\date{2020-03-19}

\title{\thetitle}
\author{%
    Christian Clason\thanks{Faculty of Mathematics, University Duisburg-Essen, 45117 Essen, Germany (\email{christian.clason@uni-due.de}, \orcid{0000-0002-9948-8426})}
    \and
    Stanislav Mazurenko\thanks{Loschmidt Laboratories, Masaryk University, Brno, Czechia; \emph{previously} Department of Mathematical Sciences, University of Liverpool, United Kingdom (\email{stan.mazurenko@gmail.com}, \orcid{0000-0003-3659-4819})}
    \and
    Tuomo Valkonen\thanks{ModeMat, Escuela Politécnica Nacional, Quito, Ecuador \emph{and} Department of Mathematics and Statistics, University of Helsinki, Finland; \emph{previously} Department of Mathematical Sciences, University of Liverpool, United Kingdom (\email{tuomo.valkonen@iki.fi}, \orcid{0000-0001-6683-3572})}
}

\AtBeginDocument{
    \hypersetup{
        pdftitle = {\thetitle},
        pdfauthor = {C.~Clason and S.~Mazurenko and T.~Valkonen}
    }
}

\begin{document}

\maketitle

\begin{abstract}
    We demonstrate that difficult non-convex non-smooth optimization problems, such as Nash equilibrium problems and anisotropic as well as isotropic Potts segmentation model, can be written in terms of generalized conjugates of convex functionals. These, in turn, can be formulated as saddle-point problems involving convex non-smooth functionals and a general smooth but non-bilinear coupling term. We then show through detailed convergence analysis that a conceptually straightforward extension of the primal--dual proximal splitting method of Chambolle and Pock is applicable to the solution of such problems. Under sufficient local strong convexity assumptions of the functionals -- but still with a non-bilinear coupling term -- we even demonstrate local linear convergence of the method. We illustrate these theoretical results numerically on the aforementioned example problems.
\end{abstract}

\section{Introduction}

This work is concerned with the numerical solution of non-smooth non-convex saddle-point problems of the form
\begin{equation}
    \label{eq:our-minmax-problem}
    \min\limits_{x\in X}\max\limits_{y\in Y}G(x)+K(x,y)-F^*(y),
\end{equation}
where $G:X\to \Rinf$ and $F^*:Y\to \Rinf$ are (possibly non-smooth) proper, convex and lower semicontinuous functionals on Hilbert spaces $X$ and $Y$, and $K:X\times Y\to \R$ is smooth but may be non-convex-concave. Such problems arise in many areas of optimal control, inverse problems, and imaging; we will treat two specific examples below.
To find a critical point for \eqref{eq:our-minmax-problem}, we propose the \emph{generalized primal--dual proximal splitting} (GPDPS) method:
\begin{algorithm}[GPDPS]
    \label{alg:gpdps}
    Given a starting point $(x^0,y^0)$ and step lengths $\tau_i,\omega_i,\sigma_i> 0$, iterate:
    \begin{align*}
        x^{i+1} & := \prox_{\tau_i G}(x^i - \tau_i K_x(x^i,\thisy)),\\
        \overnextx & :=  x^{i+1}+\omega_i(x^{i+1}-x^{i}),\\
        y^{i+1} & :=  \prox_{\sigma_{i+1} F^*}(\thisy + \sigma_{i+1} K_y(\overnextx,\thisy)),
    \end{align*}
\end{algorithm}
where $\prox_{\tau_i G}(v) = (I + \tau_i \subdiff G)^{-1}(v)$ is the proximal mapping for $G$; and $K_x,K_y$ are the partial Fréchet derivatives of $K$ with respect to $x$ and $y$. A main result of this work is that under suitable conditions on the step length parameters $\tau_i$, $\sigma_i$, and $\omega_i$, this algorithm converges weakly to a critical point of \eqref{eq:our-minmax-problem}; see \cref{thm:weak-convergence}. Furthermore, if $\partial G$ and/or $\partial F^*$ is strongly metrically subregular at the saddle point (in particular, if $G$ and/or $F^*$ are strongly convex), we show optimal convergence rates for the standard acceleration strategies; see \cref{thm:acceleration-nlpdhgm,thm:linear-convergence}.

In addition, we demonstrate in this work how through a suitable reformulation this method can be applied to the following two non-trivial applications:
\begin{enumerate}[label=(\roman*)]
    \item \emph{elliptic Nash equilibrium problems}, where $K(x,y)$ is the so-called \emph{Nikaido--Isoda} function encoding the Nash equilibrium \cite{nikaidoisoda,Krawczyk2000,heusinger2009gnep}; see \cref{sec:applications:nash} for details.
    \item \emph{(Huber-regularized) $\ell^0$-$TV$ denoising} (also referred to as the \emph{Potts model}) \cite{geman1984stochastic,storath2014jump-sparse,storath2015joint}, where $K(x,y)$ is used to express the non-convex Potts functional as the \emph{generalized $K$-conjugate} of a convex indicator function; see \cref{sec:applications:potts} for details.
\end{enumerate}
In particular, the second example demonstrates how the proposed method can be used to solve (some) non-convex non-smooth problems by reformulating in them in terms of a convex but non-smooth functional and a smooth but non-convex coupling term.
(We stress, however, that we do not claim that this approach is superior to state-of-the-art problem-specific approaches such as the ones mentioned in the cited works for the specific problems; such an investigation is left for the future.) 

\paragraph{Related literature.}
Our approach is obviously motivated by the well-known primal--dual proximal splitting (PDPS) method of Chambolle and Pock \cite{chambolle2010first} for convex optimization problems of the form $\min_{x\in X} F(Ax)+G(x)$ for $F:Y\to \Rinf$ proper, convex, and lower semicontinuous and $A:X\to Y$ linear. The method is based on the equivalent reformulation as the saddle-point problem
\begin{equation}
    \label{eq:linear-problem}
    \min_{x\in X} \max_{y\in Y} G(x) + \iprod{Ax}{y} - F^*(y)
\end{equation}
where $F^*$ is the Fenchel conjugate of $F$. 
Several other alternative techniques for such optimization problems have also been developed, e.g., using smoothing schemes \cite{nesterov2005smooth} or a proximal alternating predictor corrector \cite{drori2015-nonsmooth}. 
This approach was extended to allow for nonlinear but Fréchet differentiable $A$ in \cite{tuomov-nlpdhgm}. Later work \cite{tuomov-pdex2nlpdhgm,tuomov-nlpdhgm-redo} applied this to non-convex PDE-constrained optimization problems and derived accelerated variants.

In a broader context, generalized convex conjugation has been studied for many decades with applications in economics, see, e.g., \cite{martinez2005generalized,singer2007duality,elster1988-gen-conjugate} and the references therein. 
Algorithms for the solution of general saddle-point problems $\min_x\max_y f(x,y)$ have been considered in several seminal papers. 
In particular, a prox-type method was suggested in \cite{nemirovski2004-prox-method} for $C^{1,1}$ convex--concave functions yielding a $O(1/N)$ rate of convergence for an ergodic version of the gap $\max_{y'\in Y}f(x,y')-\min_{x'\in X}f(x',y)$. 
These results were further extended to allow non-smooth functions in the Mirror Descent method \cite{juditsky2011-mirror-decent}, demonstrating a $O(1/\sqrt{N})$ rate of convergence for the ergodic gap although with a vanishing step size for large $N$. 
The authors also considered an acceleration of the Mirror Proximal method for the case when the gradient map of $f$ can be split into a Lipschitz-continuous part and a monotone operator \cite{juditsky2011-mirror-prox}. 
The latter was assumed ``simple'' in the sense that a solution to a specific variational inequality could be found relatively efficiently. As a result, the authors obtained an $O(1/N)$ rate of convergence with a possibility for improvement to $O(1/N^2)$ for a strongly concave $f$.
Finally, the reformulation of \eqref{eq:our-minmax-problem} with a bilinear $K$ as a monotone inclusion problem was considered in \cite{he2016-accelerated-hpe}.
Algorithms applicable to \eqref{eq:our-minmax-problem} with a genuinely nonlinear $K$ have only started to appear in literature relatively recently.
An abstract convergence result was obtained for an inexact regularized Gauss--Seidel method in \cite{attouch2013convergence}.
In \cite{he2015-mirror-prox}, the authors considered saddle-point representable functions and arrived at a very similar structure to \eqref{eq:our-minmax-problem}; specifically, they reformulated this problem as a smooth linearly-constrained saddle point problem by moving the non-smooth terms into the problem domain and applied the Mirror Proximal algorithm mentioned earlier, with a smooth cost function and the $O(1/N)$ convergence rate \cite{nemirovski2004-prox-method}. 
Following \cite{he2016-accelerated-hpe}, Kolossoski and Monteiro \cite{kolossoski2017-accelerated-nehpe} developed a non-Euclidean hybrid proximal extragradient for $G$ and $F^*$ Bregman distances, and $K$ general convex--concave.
The case of a general convex--concave $K$ in \eqref{eq:our-minmax-problem} (which therefore becomes an overall convex--concave problem) has been recently studied in \cite{hamedani2018primal}. Besides being restricted to convex--concave problems, their algorithm differs from \cref{alg:gpdps} in applying the overrelaxation to $K_y(x^{i+1},y^i)$ instead of to $x^{i+1}$ in the third step.
Finally, problems for general sufficiently smooth $K(x,y)$ were considered in \cite{benning2015preconditioned} in conjunction with a variant of ADMM; however, no proofs of convergence were given in the general case.

\bigskip

\paragraph{Organization.} To motivate our approach, we start with a more detailed description of the above-mentioned example problems and their reformulation as a saddle-point problem of the form \eqref{eq:our-minmax-problem} in the next \cref{sec:applications}. (This section can be skipped by readers only interested in the convergence analysis for the general \cref{alg:gpdps}.)
The following \cref{sec:preliminaries} then collects basic notation and definitions as well as the fundamental assumptions that will be used throughout the following.
We then study the convergence and convergence rates of \cref{alg:gpdps} in \cref{sec:testing,sec:locality,sec:convergence}. More precisely, in \cref{sec:testing} we derive a basic convergence estimate using the ``testing'' framework introduced in \cite{tuomov-proxtest,tuomov-cpaccel} for the study of preconditioned proximal point methods.
The results and assumptions depend on the iterates staying in a local neighborhood of a solution. In \cref{sec:locality} we therefore derive conditions on the step length parameters and initial iterate that ensure that the iterates do not escape from a local neighborhood.
Afterwards, we provide in \cref{sec:convergence} exact step length rules for \cref{alg:gpdps} together with respective weak convergence or convergence rate results: linear under sufficient strong convexity of $G$ and $F^*$, and ``accelerated'' $O(1/N)$ or $O(1/N^2)$ rates with somewhat lesser assumptions. Finally, we illustrate the applicability and performance of the proposed approach applied to our two example problems in \cref{sec:numerical}. \Crefrange{app:three-point-reduction}{app:step-potts} contain further technical results on the assumptions required for convergence, in particular verifying them for the Huber-regularized $\ell^0$-TV denoising example.

\section{Applications}
\label{sec:applications}

Before we begin our analysis of the convergence of \cref{alg:gpdps}, we motivate its generality by discussing two examples of practically relevant problems that can be cast in the form \eqref{eq:our-minmax-problem} and which will be used to numerically illustrate the behavior of the algorithm in \cref{sec:numerical}. The idea in each case is to write a \emph{non-convex} functional $F$ as the generalized $K$-conjugate of a \emph{convex} functional $F^*$, i.e.,
\begin{equation*}
    F(x) = \sup_{y\in Y} K(x,y) - F^*(y)
\end{equation*}
for a suitable $K$ (depending on $F$).

\subsection{Elliptic Nash equilibrium problems}
\label{sec:applications:nash}

Our first example is the reformulation of Nash equilibrium problems using the Nikaido--Isoda function following \cite{heusinger2009gnep}.
Consider a non-cooperative game of $n\in \N$ players, each of which has a strategy $x_k\in X_k\subset \R$ and a payout function $\phi_k:\R^n\to \R$.
For convenience, we introduce the vector $x\in \R^n$ of strategies and the notation
\begin{equation*}
    (x_{-k}|z) := (x_1,\dots,x_{k-1},z,x_{k+1},\dots x_n) \qquad (1\leq k\leq n,\ z\in\R)
\end{equation*}
for the vector where player $k$ changes their strategy $x_k$ to $z$. We also set $X:=X_1 \times \dots \times X_n$. A vector $x^*\in X$ of strategies is then a Nash equilibrium if 
\begin{equation}\label{eq:nash}
    \phi_k (x^*) = \phi_k(x^*_{-k}|x^*_k) = \min_{z\in \R} \phi_k(x^*_{-k}|z) \qquad (1\leq k \leq n).
\end{equation}
We now introduce the Nikaido--Isoda function \cite{nikaidoisoda} (also called the Ky Fan function \cite{flam1997eppa}) 
\begin{equation*}
    \Psi(x,y) = \sum_{k=1}^n \left(\phi_k(x_{-k}|x_k) - \phi_k(x_{-k}|y_k)\right) \qquad(x,y\in X)
\end{equation*}
as well as the optimum response function
\begin{equation}\label{eq:enep_value}
    V(x) = \max_{y\in X} \Psi(x,y) \qquad (x\in X).
\end{equation}
It follows from \cite[Thm.~2.2]{heusinger2009gnep} that $x^*\in X$ is a Nash equilibrium if and only if it is a minimizer of $V$.
Using the indicator function of the set $X\subset\R^n$ defined by
\begin{equation*}
    \delta_X(x) = 
    \begin{cases} 
        0 & \text{if } x\in X,\\ 
        \infty & \text{if } x\notin X,
    \end{cases}
\end{equation*}
we see that the generally non-convex response function $V$ is the $\Psi$-preconjugate of the convex functional $\delta_X$ and can characterize a Nash equilibrium $x^*\in X$ as the solution to the saddle-point problem
\begin{equation*}
    \min_{x\in\R^n} \max_{y\in \R^n} \delta_X(x) + \Psi(x,y) - \delta_X(y).
\end{equation*}

We can therefore solve the Nash equilibrium problem \eqref{eq:nash} by applying \cref{alg:gpdps} to
\begin{equation*}
    K(x,y) = \Psi(x,y), \qquad F^* = G = \delta_X.
\end{equation*}
In \cref{sec:numerical:nash}, we illustrate this exemplarily for the two-player elliptic Nash equilibrium problem from \cite{borzikanzow2013}.

\begin{remark}
    If the set $X_k$ of feasible strategies for each player depends on the strategies of the other players (i.e., $X_k = X_k(x_{-k})$),  \eqref{eq:nash} becomes a \emph{generalized Nash equilibrium problem (GNEP)}; see the survey \cite{facchinei2010generalized} and the literature cited therein. If for all $k$
    \begin{equation*}
        X_k(x_{-k}) = \set{x_k\in \R^n:(x_{-k}|x_k)\in Z}\qquad (1\leq k\leq n)
    \end{equation*}
    for some closed and convex set $Z\subset \R^n$, the GNEP is called \emph{jointly convex}.
    In this case, minimization of \eqref{eq:enep_value} is no longer an equivalent characterization but defines a \emph{variational equilibria} \cite{rosen1964existence};  every variational equilibrium is a generalized Nash equilibrium but not vice versa, see, e.g., \cite[Thm.~3.9]{facchinei2010generalized}. Hence \cref{alg:gpdps} can also be applied to compute (some if not all) solutions to jointly convex GNEPs.
\end{remark}

\subsection{Huber--Potts denoising}
\label{sec:applications:potts}

Our next example is concerned with (Huber-regularized) $\ell^0$-TV denoising or segmentation, also referred to as \emph{Potts model}. Let $f \in \R^{N_1\times N_2}$, $N_1,N_2\in\N$, be a given noisy or to be segmented image. We then search for the denoised or segmented image as the solution to
\begin{equation}\label{eq:potts}
    \min_{x\in \R^{N_1\times N_2}} \frac1{2\alpha} \norm{x-f}^2 + \norm{D_h x}_{p,0},
\end{equation}
for a regularization parameter $\alpha\geq 0$ (which we write in front of the discrepancy term to simplify the computations), the discrete gradient $D_h:\R^{N_1\times N_2}\to \R^{N_1 \times N_2\times 2}$, and the vectorial $\ell^0$-seminorm
\begin{equation}\label{eq:l0seminorm}
    \norm{z}_{p,0} := \sum_{i=1}^{N_1}\sum_{j=1}^{N_2} \left|\left(|z_{ij1}|_0,|z_{ij2}|_0 \right)\right|_p, \qquad \text{where } |t|_0 = 
    \begin{cases}
        0 & \text{if }t = 0,\\
        1 & \text{if }t \neq 0,
    \end{cases}
\end{equation}
and $|\cdot|_p$ for $p\in [1,\infty]$ is the usual $p$-norm on $\R^2$; we will discuss the choice of $p$ in detail below.
Clearly, $\norm{\freevar}_{p,0}$ is a non-convex functional for any $p\in [1,\infty]$.
Let us briefly comment on the use of $\ell^0$-TV as a regularizer in imaging. Intuitively, the functional in \eqref{eq:l0seminorm} applied to the discrete gradient counts the number of jumps of the image value between neighboring pixels; it can therefore be expected that minimizers are piecewise constant, and that jumps are penalized even more strongly than by the (convex) total variation model.

To motivate our approach, we first consider a simple scalar (lower semicontinuous) step function, i.e., we consider for $(0,\infty)\subset \R$ the corresponding \emph{characteristic function} 
\begin{equation}
    \chi_{(0,\infty)}(t) = 
    \begin{cases}
        0 & \text{if } t \leq 0,\\
        1 & \text{if } t >0.
    \end{cases}
\end{equation}
To write this non-convex function as the generalized preconjugate of a convex function, let $\rho:\R\to\R$ satisfy $\rho(0)=0$, $\sup_{t \le 0} \rho(t)=0$, and $\sup_{t > 0} \rho(t)=1$. Then a simple case distinction shows that
\begin{equation}
    \label{eq:step-phi-preconjugation}
    \chi_{(0,\infty)}(t) = \sup_{s \ge 0}~ \rho(st) = \sup_{s \in \R}~ \rho(st) - \delta_{[0,\infty)}(s).
\end{equation}
Setting $\componentk(s,t) := \rho(st)$, we thus obtain that $\chi_{(0,\infty)}$ is the $\componentk$-preconjugate of the convex indicator function $\delta_{[0,\infty)}$. One possible choice for $\rho$ is $\rho=\chi_{(0,\infty)}$; however, we require $\rho$ to be smooth in order to apply \cref{alg:gpdps}. A better choice is therefore
\begin{equation}\label{eq:potts-rho}
    \rho(t) = 2t-t^2,\qquad (t\in \R),
\end{equation}
see \cref{fig:rho},
which has the advantage that the supremum in \eqref{eq:step-phi-preconjugation} is always attained at a finite $s \ge 0$. We will use this choice from now on.

Noting that $|t|_0 = \chi_{\{0\}}(t)$, we can proceed similarly by case distinction to write
\begin{equation*}
    |t|_0 = \sup_{s\in\R}~ \rho(st) = \sup_{s\in \R}~ \rho(st) - 0,
\end{equation*}
i.e., for $\componentk(s,t)=\rho(st)$ as above, $|\cdot|_0$ is the $\componentk$-preconjugate of the zero function $f^*\equiv 0$. In practice, it may be useful to add Huber regularization, i.e., replace $f^*$ by $f_\gamma^* := f^* + \frac\gamma2 |\cdot|^2 = \frac\gamma2|\cdot|^2$ for some $\gamma>0$. Using the fact that $f_\gamma^*$ and our choice \eqref{eq:potts-rho} are differentiable, an elementary calculus argument shows that the corresponding preconjugate is
\begin{equation*}
    |t|_\gamma := \sup_{s\in \R}~ \rho(st) - \frac\gamma2|s|^2 = \frac{2t^2}{2t^2+\gamma},
\end{equation*}
which is a still non-convex approximation of $|t|_0$, see \cref{fig:ell-gamma}.
\begin{figure}
    \centering
    \begin{minipage}[t]{0.495\textwidth}
        \centering
        \begin{tikzpicture}
            \begin{axis}[%
                width=0.8\textwidth,
                scale only axis,
                xmin=-1,
                xmax=3,
                grid=major
                ]
                \addplot [domain=-1:3,color=Blues-M, line width=1.0pt]
                    {
                        2*x-x*x
                    };
            \end{axis}
        \end{tikzpicture}

        \caption{plot of $\rho$ from \eqref{eq:potts-rho}}%
        \label{fig:rho}
    \end{minipage}
    \hfill
    \begin{minipage}[t]{0.495\textwidth}
        \begin{tikzpicture}
            \def\numsamples{201}
            \begin{axis}[%
                width=0.8\textwidth,
                scale only axis,
                xmin=-1,
                xmax=1,
                grid=major,
                legend style={legend pos=south east,legend cell align=left,align=left,draw=none,font=\scriptsize}
                ]
                \addplot [samples=\numsamples,domain=-1:1,color=Blues-G, line width=1.0pt]
                    {
                        (2*x*x)/(2*x*x+0.1)
                    };
                \addlegendentry{$\gamma=10^{-1}$};
                \addplot [samples=\numsamples,domain=-1:1,color=Blues-J, line width=1.0pt]
                    {
                        (2*x*x)/(2*x*x+0.01)
                    };
                \addlegendentry{$\gamma=10^{-2}$};
                \addplot [samples=\numsamples,domain=-1:1,color=Blues-M, line width=1.0pt]
                    {
                        (2*x*x)/(2*x*x+0.001)
                    };
                \addlegendentry{$\gamma=10^{-3}$};
            \end{axis}
        \end{tikzpicture}

        \caption{plot of $|t|_\gamma$ for different values of $\gamma$}%
        \label{fig:ell-gamma}
    \end{minipage}
\end{figure}

\bigskip

We now turn to the vectorial $\ell^0$ seminorm, where we distinguish between $p\in[1,\infty]$.

\paragraph{The case $p=1$.}
With this choice, \eqref{eq:l0seminorm} reduces to
\begin{equation*}
    \|z\|_{1,0} = \sum_{i=1}^{N_1}\sum_{j=1}^{N_2}\sum_{k=1}^2 |z_{ijk}|_0,
\end{equation*}
which is the most common choice for the Potts model found in the literature. Here, the Potts functional $\|D_h x\|_{1,0}$ counts for each pixel $(i, j)$ the jumps across each edge of the pixel separately, i.e., the contribution of each pixel is either $0$ (no jump), $1$ (jump in either horizontal or vertical direction), or $2$ (jump in both directions). We thus refer (in a slight abuse of terminology) to this case as the \emph{anisotropic Potts model}.

Since this functional is completely separable, we can apply the above scalar approach componentwise by taking
\begin{equation}\label{eq:potts:K1}
    \kappa_1(z,y) = \sum_{i=1}^{N_1}\sum_{j=1}^{N_2}\sum_{k=1}^2 \rho(z_{ijk}y_{ijk})
\end{equation}
such that $F = \|\freevar\|_{1,0}$ is the $\kappa_1$-preconjugate of the zero function $F^*\equiv 0$.
Correspondingly, the Huber regularization of $F$ is given by
\begin{equation*}
    F_\gamma(z) = \sum_{i=1}^{N_1}\sum_{j=1}^{N_2}\sum_{k=1}^2 |z_{ijk}|_\gamma.
\end{equation*}

\paragraph{The case $p=\infty$.}
Now \eqref{eq:l0seminorm} reduces to
\begin{equation*}
    \|z\|_{\infty,0} = \sum_{i=1}^{N_1}\sum_{j=1}^{N_2}\max\left\{|z_{ij1}|_0,|z_{ij2}|_0\right\}.
\end{equation*}
Here, each pixel contributes to the Potts functional only once, even if there is a jump across both edges. Since a simple case distinction shows that $\max\{|a|_0,|b|_0\} = | |(a,b)|_p |_0$ for any $a,b\in \R$ and $p\in [1,\infty]$, this case is equivalent to 
\begin{equation*}
    {|\kern-1pt\|}z{|\kern-1pt\|}_{0,p} \defeq \sum_{i=1}^{N_1}\sum_{j=1}^{N_2}\left| |(z_{ij1},z_{ij2})|_p \right|_0
\end{equation*}
for any $p\in[1,\infty]$, which leads to an alternate definition of the Potts functional sometimes found in the literature. We refer to this case as the \emph{isotropic Potts model}.

This functional is only separable with respect to the pixel coordinates $(i,j)$ but not with respect to $k$. We thus extend our preconjugation approach to $\R^2$ by observing for $t\in \R^2$ that
\begin{equation*}
    | |t|_2 |_0 = \sup_{s\in \R} \rho( \langle s, t\rangle ) = \sup_{s\in \R} \rho(s_1t_1 + s_2t_2)
\end{equation*}
since for $t=0$, $\rho(\langle s, t\rangle) = 0$ for all $s\in \R^2$, while for $t_1\neq 0$ or $t_2\neq 0$, the supremum will be attained at $1$ by the choice of $\rho$. Setting
\begin{equation}\label{eq:potts:Kinf}
    \kappa_\infty(z,y) = \sum_{i=1}^{N_1}\sum_{j=1}^{N_2} \rho\left(z_{ij1}y_{ij1}+z_{ij2}y_{ij2}\right)
\end{equation}
makes $F = \|\freevar\|_{\infty,0}$ again the $\kappa_\infty$-preconjugate of the zero function $F^*\equiv 0$.
The corresponding Huber regularization can be once more computed by elementary calculus as 
\begin{equation*}
    F_\gamma(z) = \sum_{i=1}^{N_1}\sum_{j=1}^{N_2} \left||(z_{ij1},z_{ij2})|_2\right|_\gamma.
\end{equation*}

\paragraph{The case $p\in(1,\infty)$.} In principle, one could proceed as for $p=\infty$ by constructing a function $\rho_p:\R^2\times \R^2\to\R$ with 
\begin{equation*}
    \sup_{s\in\R^2} \rho_p(s,t) = 
    \begin{cases}
        0 & \text{if } t = 0,\\
        1 & \text{if } t \neq 0, t_1t_2 = 0,\\
        2^{1/p} & \text{if } t\neq 0, t_1t_2\neq 0,
    \end{cases}
\end{equation*}
and setting $\kappa_p(s,t) = \rho_p(s,t)$.
However, since the corresponding Potts functional only differs from the case $p=1$ by the relative contribution of pixels with jumps in both directions and $2^{1/p} \to 1$ for $p\to\infty$, we will only consider the extremal cases $p=1$ and $p=\infty$.

In all cases, we can apply \cref{alg:gpdps} to
\begin{equation*}
    K(x,y) = \kappa_p(D_hx,y), \qquad 
    G(x) = \frac1{2\alpha}\norm{x-f}^2, \qquad 
    F^*_\gamma(y) = \frac\gamma2 \norm{y}^2
\end{equation*}
for $p\in [1,\infty]$ and $\gamma\geq 0$.  
We illustrate the application of \cref{alg:gpdps} for $p\in \{1,\infty\}$ and $\gamma>0$ in \cref{sec:numerical:potts}.

\begin{remark}
    We can also apply this approach for $\abs{t}^q$ with $q \in (0, 1)$ using the same $\rho$ as above, writing
    \begin{equation*}
        \abs{t}^q=\sup_{s \in \R} \kappa(t, s) \quad\text{for}\quad\kappa(t,s ) \defeq \abs{t}^q \rho(st),
    \end{equation*}
    as $\rho(st)=0$ if $t=0$ and attains the maximal value $1$ otherwise. However, $\kappa(t, s)$ is not $C^2$; we can achieve that by instead writing
    \begin{equation*}
        \abs{t}^q=\sup_{s \in \R} \kappa(t, s)
        \quad\text{for}\quad \kappa(t, s) \defeq \abs{t}^q \rho(|st|^2).
    \end{equation*}
\end{remark}

\section{Notation and assumptions}
\label{sec:preliminaries}

We start the development of our proposed method by introducing the necessary notation and overall assumptions.
Throughout the rest of this paper, we write $\linear(X; Y)$ for the space of bounded linear operators between Hilbert spaces $X$ and $Y$. In what follows, we let $x$ and $y$ denote elements of $X$ and $Y$, respectively, and denote by $u$ a pair $(x,y)\in X\times Y$. For brevity, we will also use this notation for similar tuples, e.g., $u^i :=(x^i,y^i)$, without explicit introduction in each case. 

For any Hilbert space, $I$ is the identity operator, $\iprod{x}{x'}$ is the inner product in the corresponding space, and $\B(x,r)$ is the closed unit ball of the radius $r$ at $x$. 
If $H:X\setto X$ is a set-valued map, we will frequently use the concise notation
\begin{equation*}
    \iprod{H(x)}{\tilde x}\defeq \{\iprod{w}{\tilde x}:w\in H(x)\} 
\end{equation*}
as well as, e.g.,
\begin{equation*}
0 \leq \iprod{H(x)}{\tilde x}
\end{equation*}
if the corresponding relation holds for \emph{all} $w\in H(x)$.

For self-adjoint $T,S\in\linear(X; Y)$, the inequality $T\ge S$ means $T-S$ is positive semidefinite. 
If $T\in \linear (X; X)$ is self-adjoint, we further set $\iprod{x}{x'}_T\defeq\iprod{Tx}{x'}$, and $\norm{x}_T\defeq\sqrt{\iprod{x}{x}_T}$ (which define an inner product and a norm in $X$, respectively, if $T$ is in addition positive definite). In this case, $T\ge S$ implies that $\norm{x}_T \geq \norm{x}_S$ for all $x\in X$.

We also recall that $K_x$ and $K_y$ denote the partial Fréchet derivatives of a continuosly differentiable operator $K$ with respect to the given variable. 

\bigskip

Throughout this paper, we make the following fundamental assumptions on \eqref{eq:our-minmax-problem}. 
\begin{assumption}
    \label{ass:monotone-gf}		
    The functionals $G: X \to \extR$ and $F^*:Y \to \extR$ are convex, proper, and lower semicontinuous. Furthermore, 
    \begin{enumerate}[label=(\roman*)]
        \item\label{item:monotone-g} there exist a constant $\gamma_G \in \R$ and a neighborhood $\neighx_G$ of $\realoptx$ such that
            \begin{equation}\label{eq:monotone-g}
                \iprod{\partial G(x) + K_x(\realoptx,\realopty)}{x-\realoptx} \geq \gamma_G \norm{x-\realoptx}^2 \qquad(x\in \neighx_G);
    \end{equation}

        \item\label{item:monotone-f} there exist a constant $\gamma_{F^*} \in \R$ and a neighborhood $\neighy_{F^*}$ of $\realopty$ such that
            \begin{equation}\label{eq:monotone-f}
                \iprod{\partial F^*(y) - K_y(\realoptx,\realopty)}{y-\realopty} \geq \gamma_{F^*} \norm{y-\realopty}^2 \qquad(y\in \neighy_{F^*}).
            \end{equation}
    \end{enumerate}
\end{assumption}
Let us comment on this assumption. First, since the subgradients $\subdiff G$ and $\subdiff F^*$ of convex, proper, and lower semicontinuous functionals are maximally monotone operators \cite[Theorem 20.25]{bauschke2017convex}, \cref{ass:monotone-gf} always holds with $\gamma_G=\gamma_{F^*}=0$. This is already sufficient for showing weak convergence of \cref{alg:gpdps}; see \cref{thm:weak-convergence}.
For strong convergence with rates, however, we (as usual in nonlinear optimization) need a local superlinear growth condition near the solution that requires taking $\gamma_G$ and/or $\gamma_{F^*}$ strictly positive (unless we can compensate by better properties of $K$ through \cref{ass:general} below); see \cref{thm:acceleration-nlpdhgm,thm:linear-convergence}. In this case, \cref{ass:monotone-gf}\,\ref{item:monotone-g}, for example, coincides with \emph{strong metric subregularity} of $\partial G$; see \cite{aragon2008characterization,artacho2013metric}. This property holds (at \emph{any} $\hat x$ and $\hat w\in \partial G(\hat x)$) whenever $G$ is strongly convex; however, it is a strictly weaker property since we only require it to hold at a \emph{specific} $\hat x$ and $\hat w = -K_x(\realoptx,\realopty)$ arising from the first-order necessary optimality conditions \eqref{eq:h} below. (For example, $\partial g$ for $g(x) = |x|$ is strongly metrically subregular at $x=0$ for $w\in (-1,1)$ -- but not at $w\in \{-1,1\}$ -- although $g$ is not strongly convex.)

\begin{assumption}
    \label{ass:general}
    The functional $K(x,y)\in C^1(X\times Y)$ and there exist $\metricRhoX,\metricRhoY>0$ such that for all 
    \begin{equation}
        \label{eq:neighu-definition}
        u,u'\in \neighu (\metricRhoX,\metricRhoY) \defeq
        (\B(\realoptx, \metricRhoX) \isect \neighx_G)
        \times
        (\B(\realopty, \metricRhoY) \isect \neighy_{F^*}),
    \end{equation} 
    the following properties hold:
    \begin{enumerate}[label=(\roman*)]
        \item \label{item:partial-gradk} (second partial derivatives) The second partial derivatives $K_{xy}(u)$ and $K_{yx}(u)$ exist and satisfy $K_{xy}(u)=[K_{yx}(u)]^*$.
        \item \label{item:lipschitz-gradk} (locally Lipschitz gradients)
            For some functions $L_x(y), L_y(x) \ge 0$ and a constant $L_{yx} \ge 0$,
            \begin{align*}
                \norm{K_x(x',y)-K_x(x,y)}&\le L_x(y)\norm{x'-x},
                &
                \norm{K_{yx}(x',y)-K_{yx}(x,y)}&\le L_{yx}\norm{x'-x},
                \\
                \norm{K_y(x,y')-K_y(x,y)}&\le L_y(x)\norm{y'-y}.
            \end{align*}

        \item \label{item:bounded-gradk} (locally bounded gradient)
            There exists $R_K>0$ with $\sup_{u\in\neighu(\metricRhoX,\metricRhoY)}\norm{K_{xy}(x,y)}\le R_K$.

        \item \label{item:k-nonlinear} (three-point condition) 
            There exist $\theta_x,\theta_y > 0$, $\lambda_x,\lambda_y\ge0$, $\xi_x,\xi_y\in\R$ such that
            \begin{subequations}%
                \label{eq:k-nonlinear-threepoint}
                \begin{align}
                    \label{eq:k-nonlinear-kx}
                    &\begin{aligned}[t]
                        &\iprod{K_x(x',\realopty)-K_x(\realoptx,\realopty)}{x-\realoptx}
                        +\xi_x\norm{x-\realoptx}^2
                        \\
                        \MoveEqLeft[-1]\ge
                        \theta_x\norm{K_y(\realoptx,y)-K_y(x,y)-K_{yx}(x,y)(\realoptx-x)}
                        -\frac{\lambda_x}2\norm{x-x'}^2, 
                    \end{aligned}
                    \\
                    &\begin{aligned}[t]
                        \label{eq:k-nonlinear-ky}
                        &\iprod{K_y(x,y)-K_y(x,y')+K_y(\realoptx,\realopty)-K_y(\realoptx,y)}{y-\realopty}
                        +\xi_y\norm{y-\realopty}^2
                        \\
                        \MoveEqLeft[-1]\ge
                        \theta_y\norm{K_x(x',\realopty)-K_x(x',y')-K_{xy}(x',y')(\realopty-y')}
                        -\frac{\lambda_y}2\norm{y-y'}^2.
                    \end{aligned}
                \end{align}%
            \end{subequations}
    \end{enumerate}
\end{assumption}

We again elaborate on this assumption. \Cref{ass:general}\,\ref{item:partial-gradk}--\ref{item:bounded-gradk} are standard in nonlinear optimization of smooth functions.
Apart from the estimates in \cref{ass:general}\,\ref{item:lipschitz-gradk}, we make use of the following inequality that is an immediate consequence:
\begin{equation}
    \label{eq:lipschitz-bound}
    \norm{K_y(x',y)-K_y(x,y)-K_{yx}(x,y)(x'-x)}\le \frac{L_{yx}}2\norm{x-x'}^2.
\end{equation}
The constants $\xi_x$ and $\xi_y$ in \cref{ass:general}\,\ref{item:k-nonlinear} can typically be taken positive by exploiting the strong monotonicity factors  $\gamma_G$ and $\gamma_{F^*}$ of $\subdiff G$ and $\subdiff F^*$. Indeed, further on in \cref{thm:convergence-result-main-h}, we will require that $\gamma_G-\tilde{\gamma}_G\ge\xi_x$ and  $\gamma_{F^*}-\tilde{\gamma}_{F^*}\ge\xi_y$, where $\tilde{\gamma}_G$ and $\tilde{\gamma}_{F^*}$ will be acceleration factors employed to update the step length parameters $\tau_i$, $\omega_i$, and $\sigma_i$ in the algorithm.

In \cref{app:three-point-reduction} we demonstrate that \cref{ass:general}\,\ref{item:k-nonlinear} is closely related to standard second-order optimality conditions, i.e., a positive definite Hessian at the solution $\realoptu$. In particular, if the primal problem for the saddle-point functional is strongly convex and the dual problem is strongly concave, the constants that ensure \cref{ass:general}\,\ref{item:k-nonlinear} can be found explicitly. Nonetheless, \cref{ass:general}\,\ref{item:k-nonlinear} is more general than the simple strong convex-concavity. Indeed, in \cref{app:step-potts} we verify \cref{ass:general} for $K$ arising from combinations of a linear operator with a generalized conjugate representations of the step function and the $\ell^0$ function from \cref{sec:applications:potts}.

Since \eqref{eq:k-nonlinear-ky} holds for any $\xi_y, \lambda_y \ge 0$ when $K(x, y)=\iprod{A(x)}{y}$ for some $A \in C^1(X)$, the conditions \eqref{eq:k-nonlinear-threepoint} reduce to the three-point condition for $A$ from \cite{tuomov-nlpdhgm-redo} with the exponent $p=1$. In the present work, such an exponent would correspond to exponents $p_x,p_y \in [1,2]$ over the norms with the factors $\theta_y$ and $\theta_y$ that we consider in \cref{ass:general-p2}\,\ref{item:k-nonlinear-p2}. These can sometimes be useful: The exponent $p=2$ was needed in \cite[Appendix B]{tuomov-nlpdhgm} to show the three-point condition for $A$ for a phase and amplitude reconstruction problem. For the sake of readability, in the main part of the present work we focus on the case $p_x=p_y=1$, i.e., \cref{ass:general}\,\ref{item:k-nonlinear}, and discuss the changes needed for $p_x,p_y \in (1,2]$ in \cref{app:three-point-relaxation}. 

\section{An abstract convergence result}
\label{sec:testing}

We want to find a critical point $\realoptu=(\realoptx,\realopty) \in X \times Y$ of the saddle point functional $(x,y) \mapsto G(x)+K(x,y)-F^*(y)$, i.e., satisfying
\begin{equation}
    \label{eq:h}
    0\in H(\realoptu)
    \quad\text{for}\quad
    H(u)\defeq	\begin{pmatrix}
        \subdiff G(x) + K_x(x,y) \\
        \subdiff F^*(y) - K_y(x,y)
    \end{pmatrix}.
\end{equation}
Since $G$ and $F^*$ are proper, convex, and lower semicontinuous, and $K$ is continuously differentiable, using the definition of the saddle-point, the Fréchet derivative, and the convex subdifferential, an elementary limiting argument as in, e.g., \cite[Prop.~2.2]{clason2015multitopology} shows that the inclusion \eqref{eq:h} is a first-order necessary optimality condition for a saddle point. 
If $K(x,y) = \iprod{Ax}{y}$ for $A\in \linear(X; Y)$, \eqref{eq:h} reduces to $-A^*\realopty \in \subdiff G(\realoptx)$ and $A\realoptx\in \subdiff F^*(\realopty)$, which coincides with the well-known Fenchel--Rockafellar extremality conditions for \eqref{eq:linear-problem}; see \cite[Remark 4.2]{ekeland1999convex}.

To study \cref{alg:gpdps}, we reformulate it in the preconditioned proximal point and testing framework of \cite{tuomov-proxtest}. Specifically, we write \cref{alg:gpdps} in \emph{implicit proximal point} form as solving in each iteration for $\nextu=(\nextx,\nexty)\in X\times Y$ in
\begin{equation}
    \label{eq:ppext}
    \tag{IPP}
    0 \in \Step_{i+1}\tilde H_{i+1}(\nextu)+\Precond_{i+1}(\nextu-\thisu),
\end{equation}
where the linearization $\tilde H_{i+1}$ of $H$, the linear preconditioner $\Precond_{i+1}$, and the step length operator $\Step_{i+1}$ are defined as
\begin{align}
    \label{eq:tildeh}
    \tilde H_{i+1}(u) & \defeq 
    \begin{pmatrix}
        \subdiff G(x) + K_x(\thisx,\thisy)+K_{xy}(\thisx,\thisy)(y-\thisy) \\
        \subdiff F^*(y)-K_y((1+\omega_i)x-\omega_i \thisx,\thisy)-K_{yx}(\thisx,\thisy)(x-[(1+\omega_i)x-\omega_i\thisx])
    \end{pmatrix},
    \\
    \label{eq:precond}
    \Precond_{i+1} & \defeq
    \begin{pmatrix}
        I & -\tau_i K_{xy}(\thisx,\thisy) \\
        -\omega_i \sigma_{i+1} K_{yx}(\thisx,\thisy) & I
    \end{pmatrix},
    \\
    \label{eq:steplength}
    \Step_{i+1} & \defeq
    \begin{pmatrix}
        \tau_i I & 0 \\
        0 & \sigma_{i+1} I
    \end{pmatrix}.
\end{align}
Inserting these definitions into \eqref{eq:ppext} and rearranging, we can rewrite inclusion \eqref{eq:ppext} as
\begin{equation}\label{eq:ppext-explicit}
	0 \in 
	 \begin{pmatrix}
		\tau_i\subdiff G(\nextx) + \tau_iK_x(\thisx,\thisy)+\nextx-\thisx
		\\
		\sigma_{i+1}\subdiff F^*(\nexty)-\sigma_{i+1}K_y((1+\omega_i)\nextx -\omega_i\thisx,\thisy)+\nexty-\thisy
	\end{pmatrix}.
\end{equation}
Therefore, based on the definitions of the proximal point mapping $\prox_{\tau G}(v)  = (I+\tau\partial G)^{-1}(v)$ and of $\overnextx = (1+\omega_i)\nextx -\omega_i\thisx$, solving \eqref{eq:ppext} for $\nextu$ is equivalent to performing one step of \cref{alg:gpdps}.
Since proximal mappings of proper, convex and lower semicontinuous functionals are well-defined, single-valued, and Lipschitz continuous \cite[Proposition 12.15]{bauschke2017convex}, and $K$ is twice Fréchet differentiable on $X\times Y$, this also shows that \eqref{eq:ppext} always admits a unique solution $\nextu$.

The next step is to ``test'' the inclusion \eqref{eq:ppext} by application of $\iprod{\freevar}{\nextu-\realoptu}_{\Test_{i+1}}$ for the testing operator
\begin{equation*}
    \Test_{i+1} \defeq
    \begin{pmatrix}
        \tauTest_i I & 0 \\
        0 & \sigmaTest_{i+1} I
    \end{pmatrix}.
\end{equation*}
This testing operator and the respective primal and dual testing variables $\tauTest_i$ and $\sigmaTest_{i+1}$ will be seen to encode convergence rates after some rearrangements of the tested inclusions for $i=0,\ldots,N-1$.

We will base our convergence analysis on the following abstract estimate, where $\norm{\freevar}_{\Test_{N+1}\Precond_{N+1}}^2$ forms a local metric that measures the convergence of the iterates while $\Penalty_{i+1}$ can potentially be used to measure function value or gap converge. In particular, we therefore want $\norm{u}_{\Test_{N+1}\Precond_{N+1}}\to\infty$ as $N\to\infty$ with a certain rate such that boundedness of $\norm{u^N -\realoptu}_{\Test_{N+1}\Precond_{N+1}}$ implies the convergence of $u^N\to \realoptu$ at the reciprocal rate (see \cref{thm:acceleration-nlpdhgm,thm:linear-convergence}).
\begin{theorem}[\protect{\cite[Theorem 2.1]{tuomov-proxtest}}]
    \label{thm:convergence-result-main-h}
    Suppose \eqref{eq:ppext} is solvable, and denote the iterates by $\{\thisu\}_{i \in \N}$. If $\Test_{i+1}\Precond_{i+1}$ is self-adjoint and for some $\realoptu \in \Space$ and $\Penalty_{i+1}=\Penalty_{i+1}(\realoptu)\in \R$, for all $i\le N-1$, 
    \begin{gather}
        \label{eq:convergence-fundamental-condition-iter-h}
        \begin{aligned}[t]
            \iprod{\Test_{i+1}\Step_{i+1}\tilde H_{i+1}(\nextu)}{\nextu-\realoptu}
            + \Penalty_{i+1}
            &
            \ge
            \frac{1}{2}\norm{\nextu-\realoptu}_{\Test_{i+2}\Precond_{i+2}-\Test_{i+1}\Precond_{i+1}}^2
            -\frac{1}{2}\norm{\nextu-\thisu}_{\Test_{i+1} \Precond_{i+1}}^2,
        \end{aligned}
        \shortintertext{then}
        \label{eq:convergence-result-main-h}
        \frac{1}{2}\norm{u^N-\realoptu}^2_{\Test_{N+1}\Precond_{N+1}}
        \le
        \frac{1}{2}\norm{u^0-\realoptu}^2_{\Test_{1}\Precond_{1}}
        +
        \sum_{i=0}^{N-1} \Penalty_{i+1}.
    \end{gather}
\end{theorem}

The next theorem specializes \cref{thm:convergence-result-main-h} to our specific setup, converting the abstract condition \eqref{eq:convergence-fundamental-condition-iter-h} into several step length and testing parameter update rules and bounds.
Specifically, \eqref{eq:scalar-step-rules0} below couples the primal and dual step lengths $\tau_i$ and $\sigma_i$ and the over-relaxation parameter $\omega_i$ with the testing parameters.
Condition \eqref{eq:scalar-test-update} determines convergence rates by limiting how fast the testing parameters can grow. This rate is limited through the available strong monotonicity or second-order behavior ($\gamma_G-\xi_x$ and $\gamma_{F^
*}-\xi_y$) through \eqref{eq:scalar-gamma-rules-g} and \eqref{eq:scalar-gamma-rules-fstar} as well as additional step length bounds from \eqref{eq:scalar-tau-sigma-rule}. We point out that only the latter are specific to our non-convex setting; the remaining conditions are present in the convex setting as well, see \cite{tuomov-proxtest}.
We will further develop these rules and conditions in the next section to obtain specific convergence results; an explicit example for a set of parameters satisfying these rules and conditions will be provided for the $\ell^0$-TV denoising in \cref{sec:numerical:potts,app:step-potts}. Here and in the following, we use the notation $\overnextx :=  x^{i+1}+\omega_i(x^{i+1}-x^{i})$ from \cref{alg:gpdps} for brevity.
\begin{theorem}
    \label{thm:nonneg-penalty}
    Suppose \cref{ass:monotone-gf,ass:general} hold with the constants $\theta_x, \theta_y>0$; $\xi_x,\xi_y \in \R$; $\lambda_x,\lambda_y \ge 0$; $L_{yx} \ge 0$ and $R_K>0$.
    For all $i \in \N$, let $\overnextu \defeq (\overnextx, \thisy)$, and suppose $\thisu,\nextu,\realoptu,\overnextu \in \neighu (\metricRhoX,\metricRhoY)$ for some $\metricRhoX,\metricRhoY \ge 0$.
    Assume for all $i\in\N$ that $\overline{\omega}\ge \omega_i \ge \underline{\omega}>0$ and that for some $0 < \delta \le \muConst < 1$; $\eta_i > 0$; and $\tilde\gamma_G,\tilde\gamma_{F^*} \ge 0$,
    \begin{subequations}%
        \label{eq:scalar-rules}
        \begin{align}
            \label{eq:scalar-step-rules0}
            \omega_i & = \eta_i\eta_{i+1}^{-1}, & \eta_i & = \sigmaTest_{i}\sigma_{i} = \tauTest_i\tau_i,
            \\
            \label{eq:scalar-test-update}
            \tauTest_{i+1} & =\tauTest_i(1+2\tau_i\tilde\gamma_G),
            &
            \sigmaTest_{i+2} & =\sigmaTest_{i+1}(1+2\sigma_{i+1}\tilde\gamma_{F^*}),
            \\
            \label{eq:scalar-tau-sigma-rule}
            1 & \ge \sigma_{i}\biggl(\frac{R_K^2\tau_i}{1-\muConst}+\frac{\lambda_y}{\omega_i}\biggr), 
            &
            \tau_i & \le\frac{\delta}{\lambda_x+L_{yx}\left(\omega_i+2\right)\metricRhoY},
            \\
            \label{eq:scalar-gamma-rules-g}
            \gamma_G & \ge \tilde\gamma_G + \xi_x,
            &
            \theta_y & \ge \overline{\omega}\metricRhoX,        
            \\
            \label{eq:scalar-gamma-rules-fstar}
            \gamma_{F^*} & \ge \tilde\gamma_{F^*} + \xi_y,
            &
            \theta_x &\ge \metricRhoY\underline{\omega}^{-1}.
        \end{align}%
    \end{subequations}
    Then \eqref{eq:convergence-fundamental-condition-iter-h} is satisfied for any $\Penalty_{i+1}\le0$.
\end{theorem}

\begin{proof}
    We split the proof into several steps.

    \proofstep{Step 1 (estimation of $\Test_{i+1}\Precond_{i+1}$)}
    By \eqref{eq:scalar-step-rules0}, $\tauTest_i\tau_i=\eta_i$ and  $\sigmaTest_{i+1} \sigma_{i+1} \omega_i=\eta_i$, so \eqref{eq:precond} yields
    \begin{equation}
        \label{eq:zimi-estim-eq}
        \Test_{i+1}\Precond_{i+1}
        =\begin{pmatrix}
            \tauTest_i I& -\eta_i K_{xy}(\thisx,\thisy) \\
            -\eta_i K_{yx}(\thisx,\thisy) & \sigmaTest_{i+1} I
        \end{pmatrix},
    \end{equation}
    which is clearly self-adjoint. Applying Cauchy's and Young's inequalities, we further obtain for any $\delta>0$, $x\in X$, and $y\in Y$ that
    \begin{equation*}
        -2\iprod{x}{\eta_iK_{xy}(\thisx,\thisy)y}\ge - (1-\delta)\tauTest_i\norm{x}^2-(1-\delta)^{-1}\tauTest_i^{-1}\eta_i^2\norm{K_{xy}(\thisx,\thisy)y}^2,
    \end{equation*}
    implying that
    \begin{equation}
        \label{eq:test-precond-expansion-estimate}
        \Test_{i+1}\Precond_{i+1}
        \ge\hat{Q}_{i+1}\defeq
        \begin{pmatrix}
            \delta \tauTest_i I& 0 \\
            0 & \sigmaTest_{i+1}I - \frac{\eta^2_i\tauTest_i^{-1}}{1-\delta}K_{yx}(\thisx,\thisy) K_{xy}(\thisx,\thisy)
        \end{pmatrix}.
    \end{equation}

    \proofstep{Step 2 (estimation of $\Test_{i+1}\Precond_{i+1}-\Test_{i+2}\Precond_{i+2}$)}
    Expanding $\Test_{i+1}\Precond_{i+1}-\Test_{i+2}\Precond_{i+2}$ according to \eqref{eq:zimi-estim-eq} and then applying  \eqref{eq:scalar-test-update}, we obtain
    \begin{multline}
        \label{eq:local-metric-transfer}
        \frac{1}{2}\norm{\nextu-\realoptu}_{\Test_{i+1}\Precond_{i+1}-\Test_{i+2}\Precond_{i+2}}^2
        =-\eta_i\tilde\gamma_G\norm{\nextx-\realoptx}^2-\eta_{i+1}\tilde\gamma_{F^*}\norm{\nexty-\realopty}^2
        \\
        +\iprod{(\eta_{i+1} K_{xy}(\nextx,\nexty)-\eta_i K_{xy}(\thisx,\thisy))(\nexty-\realopty)}{\nextx-\realoptx}.
    \end{multline}

    \proofstep{Step 3 (estimation of $\tilde H_{i+1}(\nextu)$)}
    By \eqref{eq:tildeh} we have
    \begin{equation*}
        \tilde H_{i+1}(\nextu) =
        \begin{pmatrix}
            \subdiff G(\nextx) + K_x(\thisx,\thisy)+K_{xy}(\thisx,\thisy)(\nexty-\thisy) \\
            \subdiff F^*(\nexty)-K_y(\overnextx,\thisy)-K_{yx}(\thisx,\thisy)(\nextx-\overnextx)
        \end{pmatrix}.
    \end{equation*}
    Since $0 \in H(\realoptu)$, we have $- K_x(\realoptx,\realopty)\in\subdiff G(\realoptx)$ and $K_y(\realoptx,\realopty)\in\subdiff F^*(\realopty)$. Using \eqref{eq:ppext-explicit} multliplied by $\Test_{i+1}$, \cref{ass:monotone-gf}, and \eqref{eq:scalar-step-rules0}, we can thus estimate
    \begin{equation}
        \label{eq:h-extended-estimate}
        \begin{aligned}[t]
            \iprod{\tilde H_{i+1}(\nextu)}{\nextu-\realoptu}_{\Step_{i+1}\Test_{i+1}}
            &\ge\eta_i\gamma_G\norm{\nextx-\realoptx}^2+\eta_{i+1}\gamma_{F^*}\norm{\nexty-\realopty}^2
            \\\MoveEqLeft[-1]
            +\eta_i\iprod{K_x(\thisx,\thisy)-K_x(\realoptx,\realopty)+K_{xy}(\thisx,\thisy)(\nexty-\thisy)}{\nextx-\realoptx}
            \\\MoveEqLeft[-1]
            +\eta_{i+1}\iprod{K_y(\realoptx,\realopty)-K_y(\overnextx,\thisy)-K_{yx}(\thisx,\thisy)(\nextx-\overnextx)}
            {\nexty-\realopty}.
        \end{aligned}
    \end{equation}
    Combining \eqref{eq:h-extended-estimate}, \eqref{eq:local-metric-transfer}, and \eqref{eq:test-precond-expansion-estimate}, we arrive at
    \begin{equation}
        \label{eq:ci-d}
        \begin{aligned}[t]
            S_{i+1} & \defeq \frac{1}{2}\norm{\nextu-\thisu}_{\Test_{i+1} \Precond_{i+1}}^2
            +\frac{1}{2}\norm{\nextu-\realoptu}_{\Test_{i+1}\Precond_{i+1}-\Test_{i+2}\Precond_{i+2}}^2
            +\iprod{\tilde H_{i+1}(\nextu)}{\nextu-\realoptu}_{\Step_{i+1}\Test_{i+1}}
            \\
            &\ge\frac{1}{2}\norm{\nextu-\thisu}_{\hat{Q}_{i+1}}^2+D
        \end{aligned}
    \end{equation}
    for
    \begin{equation*}
        \begin{aligned}
            D&\defeq
            \eta_i(\gamma_G-\tilde\gamma_G)\norm{\nextx-\realoptx}^2+\eta_{i+1}(\gamma_{F^*}-\tilde\gamma_{F^*})\norm{\nexty-\realopty}^2
            \\\notag\MoveEqLeft[-1]
            +\iprod{(\eta_{i+1} K_{xy}(\nextx,\nexty)-\eta_i K_{xy}(\thisx,\thisy))(\nexty-\realopty)}{\nextx-\realoptx}
            \\\notag\MoveEqLeft[-1]
            +\eta_i\iprod{K_x(\thisx,\thisy)-K_x(\realoptx,\realopty)+K_{xy}(\thisx,\thisy)(\nexty-\thisy)}{\nextx-\realoptx}
            \\\notag\MoveEqLeft[-1]
            +\eta_{i+1}\iprod{K_y(\realoptx,\realopty)-K_y(\overnextx,\thisy)-K_{yx}(\thisx,\thisy)(\nextx-\overnextx)}
            {\nexty-\realopty}.
        \end{aligned}
    \end{equation*}
    The claim of the theorem is established if we prove that $S_{i+1} \ge 0$.

    \proofstep{Step 4 (estimation of $D$)}
    With
    \begin{align*}
        \tilde D_{x+y} 
        &\begin{aligned}[t]
            &\defeq
            \iprod{(\eta_{i+1} K_{xy}(\nextx,\nexty)-\eta_i K_{xy}(\thisx,\thisy))(\nexty-\realopty)}{\nextx-\realoptx}
            \\\MoveEqLeft[-1]
            +\eta_i\iprod{K_x(\thisx,\thisy)-K_x(\realoptx,\realopty)+K_{xy}(\thisx,\thisy)(\nexty-\thisy)}{\nextx-\realoptx}
            \\\MoveEqLeft[-1]
            +\eta_{i+1}\iprod{K_y(\realoptx,\realopty)-K_y(\nextx,\thisy)}
            {\nexty-\realopty},\\
        \end{aligned}
        \shortintertext{and}
        D_\omega
        &\begin{aligned}[t]
            &\defeq \iprod{K_y(\nextx,\thisy)-K_y(\overnextx,\thisy)+K_{yx}(\nextx,\thisy)(\overnextx-\nextx)}{\nexty-\realopty}
            \\\MoveEqLeft[-1]
            +\iprod{[K_{yx}(\thisx,\thisy)-K_{yx}(\nextx,\thisy)](\overnextx-\nextx)}{\nexty-\realopty},
        \end{aligned}
    \end{align*}
    we can rewrite
    \begin{equation*}
        \begin{aligned}
            D&=
            \eta_i(\gamma_G-\tilde\gamma_G)\norm{\nextx-\realoptx}^2+\eta_{i+1}(\gamma_{F^*}-\tilde\gamma_{F^*})\norm{\nexty-\realopty}^2
            +\tilde D_{x+y}
            +\eta_{i+1} D_\omega.
        \end{aligned}
    \end{equation*}
    We rearrange
    \begin{equation*}
        \begin{aligned}[t]
            \tilde D_{x+y}&=\eta_i\iprod{K_x(\thisx,\realopty)-K_x(\realoptx,\realopty)}{\nextx-\realoptx}
            \\\MoveEqLeft[-1]
            +\eta_{i+1}\iprod{K_y(\realoptx,\nexty)-K_y(\nextx,\nexty)+K_{yx}(\nextx,\nexty)(\nextx-\realoptx)}{\nexty-\realopty}
            \\\MoveEqLeft[-1]
            +\eta_{i+1}\iprod{K_y(\realoptx,\realopty)-K_y(\realoptx,\nexty)}{\nexty-\realopty}
            \\\MoveEqLeft[-1]
            +\eta_{i+1}\iprod{K_y(\nextx,\nexty)-K_y(\nextx,\thisy)}{\nexty-\realopty}
            \\\MoveEqLeft[-1]
            -\eta_i\iprod{K_x(\thisx,\realopty)-K_x(\thisx,\thisy)-K_{xy}(\thisx,\thisy)(\realopty-\thisy)}{\nextx-\realoptx}.
        \end{aligned}
    \end{equation*}

    Since $\eta_{i+1}=\eta_i\omega_i^{-1}$, setting
    \begin{equation*}
        \begin{aligned}D_x&\defeq \xi_x\norm{\nextx-\realoptx}^2+\iprod{K_x(\thisx,\realopty)-K_x(\realoptx,\realopty)}{\nextx-\realoptx}
            \\\MoveEqLeft[-1]
            +\iprod{K_y(\realoptx,\nexty)-K_y(\nextx,\nexty)-K_{yx}(\nextx,\nexty)(\realoptx-\nextx)}{\nexty-\realopty}\omega_i^{-1},
            \quad\text{and}
            \\[1ex]
            D_y&\defeq \xi_y\norm{\nexty-\realopty}^2+\iprod{K_y(\realoptx,\realopty)-K_y(\realoptx,\nexty)}{\nexty-\realopty}
            \\\MoveEqLeft[-1]
            +\iprod{K_y(\nextx,\nexty)-K_y(\nextx,\thisy)}{\nexty-\realopty}
            \\\MoveEqLeft[-1]
        -\omega_i\iprod{K_x(\thisx,\realopty)-K_x(\thisx,\thisy)-K_{xy}(\thisx,\thisy)(\realopty-\thisy)}{\nextx-\realoptx},\end{aligned}
    \end{equation*}
    we can write
    \begin{equation*}
        D=
        \eta_i(\gamma_G-\tilde\gamma_G-\xi_y)\norm{\nextx-\realoptx}^2+\eta_{i+1}(\gamma_{F^*}-\tilde\gamma_{F^*}+\xi_x)\norm{\nexty-\realopty}^2
        +\eta_i D_x + \eta_{i+1} D_y
        +\eta_{i+1} D_\omega.
    \end{equation*}

    As for the estimate for $D_\omega$, using \cref{ass:general}\,\ref{item:lipschitz-gradk} and \eqref{eq:lipschitz-bound} we obtain
    \begin{equation}
        \label{eq:dw-estimate}
        \begin{aligned}[t]
            D_\omega&\ge -\frac{L_{yx}}{2}\norm{\overnextx-\nextx}^2\norm{\nexty-\realopty}
            -L_{yx}\norm{\nextx-\thisx}\norm{\overnextx-\nextx}\norm{\nexty-\realopty}
            \\
            &\ge-\frac{L_{yx}\omega_i(\omega_i+2)\metricRhoY}{2}\norm{\nextx-\thisx}^2
        \end{aligned}
    \end{equation}
    using in the last inequality the expansion $\overnextx\defeq \nextx +\omega_i(\nextx-\thisx)$ and the bound $\norm{\nexty-\realopty}\le \metricRhoY$ that follows from the assumed inclusion $\nextu\in\neighu(\metricRhoX,\metricRhoY)$.

    We now use \cref{ass:general}\,\ref{item:k-nonlinear} to further bound $D_x$ and $D_y$. From \eqref{eq:k-nonlinear-kx}, we obtain
    \begin{equation}
        \label{eq:dx-estimate}
        \begin{aligned}[t]
            D_x&\ge
            \theta_x\norm{K_y(\realoptx,\nexty)-K_y(\nextx,\nexty)-K_{yx}(\nextx,\nexty)(\realoptx-\nextx)}
            -\frac{\lambda_x}2\norm{\nextx-\thisx}^2
            \\\MoveEqLeft[-1]
            -\norm{\nexty-\realopty}\norm{K_y(\realoptx,\nexty)-K_y(\nextx,\nexty)-K_{yx}(\nextx,\nexty)(\realoptx-\nextx)}\omega_i^{-1}
            \\
            &\ge
            (\theta_x-\metricRhoY\underline{\omega}^{-1})
            \norm{K_y(\realoptx,\nexty)-K_y(\nextx,\nexty)-K_{yx}(\nextx,\nexty)(\realoptx-\nextx)}
            -\frac{\lambda_x}2\norm{\nextx-\thisx}^2
            \\
            &\ge 
            -\frac{\lambda_x}2\norm{\nextx-\thisx}^2,
        \end{aligned}
    \end{equation}
    using in the last two inequalities that $\nextu \in \neighu (\metricRhoX,\metricRhoY)$ for some $\metricRhoX,\metricRhoY \ge 0$, $\omega_i^{-1} \le \underline{\omega}^{-1}$ and $\theta_x \ge \metricRhoY\underline{\omega}^{-1}$ from \eqref{eq:scalar-gamma-rules-fstar}.
    Analogously, from \eqref{eq:k-nonlinear-ky} and Cauchy's inequality, 
    \begin{equation}
        \label{eq:dy-estimate}
        \begin{aligned}[t]
            D_y&\ge
            \theta_y\norm{K_x(\thisx,\realopty)-K_x(\thisx,\thisy)-K_{xy}(\thisx,\thisy)(\realopty-\thisy)}
            -\frac{\lambda_y}2\norm{\nexty-\thisy}^2
            \\\MoveEqLeft[-1]
            -\omega_i\norm{\nextx-\realoptx}\norm{K_x(\thisx,\realopty)-K_x(\thisx,\thisy)-K_{xy}(\thisx,\thisy)(\realopty-\thisy)}
            \\
            &\ge
            (\theta_y-\metricRhoX\overline{\omega})
            \norm{K_x(\thisx,\realopty)-K_x(\thisx,\thisy)-K_{xy}(\thisx,\thisy)(\realopty-\thisy)}
            -\frac{\lambda_y}2\norm{\nexty-\thisy}^2
            \\
            &\ge 
            -\frac{\lambda_y}2\norm{\nexty-\thisy}^2,
        \end{aligned}
    \end{equation}
    where in the last two inequalities we again used $\nextu \in \neighu (\metricRhoX,\metricRhoY)$, $\omega_i \le \overline{\omega}$, and $\theta_y \ge \overline{\omega}\metricRhoX$ from \eqref{eq:scalar-gamma-rules-g}.  
    Therefore, combining \eqref{eq:dw-estimate}, \eqref{eq:dx-estimate}, and \eqref{eq:dy-estimate}, we obtain
    \begin{equation}
        \label{eq:final-d-estimate}
        \begin{aligned}[t]
            D&=
            \eta_i D_x+\eta_{i+1}D_y+\eta_{i+1}D_\omega
            +\eta_i(\gamma_G-\tilde\gamma_G-\xi_x)\norm{\nextx-\realoptx}^2
            +\eta_{i+1}(\gamma_{F^*}-\tilde\gamma_{F^*}-\xi_y)\norm{\nexty-\realopty}^2
            \\
            &\ge
            \eta_{i+1}(\gamma_{F^*}-\tilde\gamma_{F^*}-\xi_y) \norm{\nexty-\realopty}^2
            -\eta_i\frac{\lambda_x}2\norm{\nextx-\thisx}^2
            \\\MoveEqLeft[-1]
            +\eta_{i}(\gamma_G-\tilde\gamma_G-\xi_x)\norm{\nextx-\realoptx}^2
            -\eta_{i+1}\frac{\lambda_y}2\norm{\nexty-\thisy}^2
            -\eta_i \frac{L_{yx}}2(\omega_i+2)\metricRhoY \norm{\nextx-\thisx}^2
            \\
            & \ge
            -\eta_i\frac{\lambda_x+L_{yx}(\omega_i+2)\metricRhoY}2\norm{\nextx-\thisx}^2
            -\eta_{i+1}\frac{\lambda_y}2\norm{\nexty-\thisy}^2,
        \end{aligned}
    \end{equation}
    where we have also used the first bounds of \eqref{eq:scalar-gamma-rules-g} and \eqref{eq:scalar-gamma-rules-fstar} in the final step.
    Further using \eqref{eq:scalar-tau-sigma-rule} and $\eta_{i+1}=\eta_i\omega_i^{-1}$, we deduce that $ D \ge - \frac{1}{2}\norm{\nextu-\thisu}_{\hat{Q}_{i+1}}^2$. Recalling \eqref{eq:ci-d}, we obtain $S_{i+1} \ge 0$, i.e., \eqref{eq:convergence-fundamental-condition-iter-h} holds with $\Penalty_{i+1} \le 0$ as claimed.
\end{proof}

In the subsequent sections, we will also need the following corollary.
\begin{corollary}
    \label{crl:zimi}
    Suppose that \cref{ass:general}\,\ref{item:bounded-gradk} and the conditions \eqref{eq:scalar-rules} hold. Then
    \begin{equation*}
        (1-\muConst)\sigmaTest_{i+1} \ge \eta_i^2\tauTest_i^{-1}  R_K^2
    \end{equation*}
    and
    \begin{equation}
        \label{eq:zimi-estim}
        \Test_{i+1}\Precond_{i+1}
        \ge\begin{pmatrix}
            \delta \tauTest_i I& 0 \\
            0 & (\muConst-\delta)(1-\delta)^{-1}\sigmaTest_{i+1} I
        \end{pmatrix}.
    \end{equation}
\end{corollary}

\begin{proof}
    Observe that due to \eqref{eq:scalar-rules},
    \begin{equation*}
        (1-\muConst)\sigmaTest_{i+1} \ge (1-\muConst)\sigmaTest_{i}=\frac{(1-\muConst)\eta_i^2}{\sigma_i \tau_i \tauTest_i}\ge \eta_i^2 \tauTest_i^{-1} R_K^2.
    \end{equation*}
    This is our first claim. 
    As for the second term, from \cref{ass:general}\,\ref{item:bounded-gradk} we have 
    \begin{equation*}
        \frac{\eta^2_i\tauTest_i^{-1}}{1-\delta}K_{yx}(\thisx,\thisy) K_{xy}(\thisx,\thisy)
        \le 
        \frac{\eta^2_i\tauTest_i^{-1}}{1-\delta}R_K^2 I
        \le
        \frac{1-\muConst}{1-\delta}\sigmaTest_{i+1}I.
    \end{equation*}  
    Inserting this bound into \eqref{eq:test-precond-expansion-estimate} in the proof of \cref{thm:nonneg-penalty} establishes \eqref{eq:zimi-estim}.
\end{proof}

\section{Local step length bounds}
\label{sec:locality}

In the previous section, we derived step length conditions that we will further develop in \cref{sec:convergence} to prove convergence and convergence rates.
However, we implicitly required that all the iterations $\{\thisu\}_{i\in\N}$ belong to $\neighu(\metricRhoX,\metricRhoY)$.
In this section, we derive additional step lengths restrictions to ensure that this holds.

We start with a lemma that bounds the next iterate $\nextu$ given bounds on the current iterate $\thisu$ and the step lengths for the current iteration. Afterwards, we chain these estimates to only require bounds on the initial iterates and the step lengths.
\begin{lemma}
    \label{lemma:step-length-bounds-r}
    Fix $i \in \N$.
    Suppose \cref{ass:monotone-gf}, \cref{ass:general}\,\ref{item:lipschitz-gradk}, and \ref{item:bounded-gradk} hold in $\neighu(\metricRhoX,\metricRhoY)$, and that $\nextu$ solves \eqref{eq:ppext}.
    For simplicity, assume $\omega_i \le 1$.
    Suppose $\localRhoX[i], \localRhoY[i], \delta_x,\delta_y>0$ and $\realoptu \in \inv H(0)$ are such that $\B(\realoptx, \localRhoX[i]+\delta_x)\times\B(\realopty, \localRhoY[i]+\delta_y) \subseteq \neighu(\metricRhoX,\metricRhoY)$ and
    $\thisu\in\B(\realoptx, \localRhoX[i])\times\B(\realopty, \localRhoY[i])$.
    If
    \begin{equation}
        \label{eq:step-length-bounds-r}
        \tau_i
        \le\frac{\delta_x}{2R_K\localRhoY[i]+2L_x(\realopty)\localRhoX[i]}
        \quad\text{and}\quad
        \sigma_{i+1}
        \le\frac{\delta_y}{L_y(\realoptx)\localRhoY[i]+R_K(\localRhoX[i]+\delta_x)},
    \end{equation}
    then $\nextu\in\B(\realoptx, \localRhoX[i]+\delta_x)\times\B(\realopty, \localRhoY[i]+\delta_y)$ and $\norm{\overnextx-\realoptx}\le \localRhoX[i]+\delta_x$.
\end{lemma}

\begin{proof}
    We want to show that the step length conditions \eqref{eq:step-length-bounds-r} are sufficient for
    \begin{equation*}
        \norm{\nextx-\realoptx}\le \localRhoX[i]+\delta_x,
        \quad
        \norm{\overnextx-\realoptx}\le \localRhoX[i]+\delta_x,
        \quad\text{and}\quad
        \norm{\nexty-\realopty}\le \localRhoY[i]+\delta_y.
    \end{equation*}
    We do this by applying the testing argument on the primal and dual variables separately. 
    Multiplying \eqref{eq:ppext} by $\Test_{i+1}^*(\nextu-\realoptu)$ with $\tauTest_i=1$ and $\sigmaTest_{i+1}=0$, we obtain
    \begin{equation*}
        0 \in \tau_i\iprod{\subdiff G(\nextx)+K_x(\thisx,\thisy)}{\nextx-\realoptx}
        +
        \iprod{\nextx-\thisx}{\nextx-\realoptx}.
    \end{equation*}
    Using the three-point identity
    \begin{equation}
        \label{eq:standard-identity}
        \iprod{\nextx-\thisx}{\nextx-\realoptx}
        = \frac{1}{2}\norm{\nextx-\thisx}^2
        - \frac{1}{2}\norm{\thisx-\realoptx}^2
        + \frac{1}{2}\norm{\nextx-\realoptx}^2,
    \end{equation}
    we obtain
    \begin{equation*}
        \norm{\thisx-\realoptx}^2 \in 2\tau_i\iprod{\subdiff G(\nextx)+K_x(\thisx,\thisy)}{\nextx-\realoptx}
        +
        \norm{\nextx-\thisx}^2
        +\norm{\nextx-\realoptx}^2.
    \end{equation*}
    Using further $0 \in \subdiff G(\realoptx)+K_x(\realoptx,\realopty)$ and the monotonicity of $\subdiff G$, we arrive at
    \begin{equation*}
        \norm{\nextx-\thisx}^2 + \norm{\nextx-\realoptx}^2 +
        2\tau_i\iprod{K_x(\thisx,\thisy)-K_x(\realoptx,\realopty)}{\nextx-\realoptx}\le
        \norm{\thisx-\realoptx}^2.
    \end{equation*}
    With $C_x\defeq\tau_i\norm{K_x(\thisx,\thisy)-K_x(\realoptx,\realopty)}$, this implies that
    \begin{equation}
        \label{eq:step-length-estimate-raw}
        \norm{\nextx-\thisx}^2 + \norm{\nextx-\realoptx}^2\le
        2C_x\norm{\nextx-\realoptx}+\norm{\thisx-\realoptx}^2.
    \end{equation}
    After rearranging the terms and using $\norm{\nextx-\realoptx}\le\norm{\nextx-\thisx}+\norm{\thisx-\realoptx}$, we thus have
    \begin{equation*}
        (\norm{\nextx-\thisx}-C_x)^2  + \norm{\nextx-\realoptx}^2\le (\norm{\thisx-\realoptx}+C_x)^2,
    \end{equation*}
    which leads to
    \begin{equation}
        \label{eq:nextx-realoptx-from-thisx-realoptx}
        \norm{\nextx-\realoptx}\le\norm{\thisx-\realoptx}+C_x.
    \end{equation}

    To estimate the dual variable, we multiply \eqref{eq:ppext} by $\Test_{i+1}^*(\nextu-\realoptu)$ with $\tauTest_{i}=0$ and $\sigmaTest_{i+1}=1$. This gives
    \begin{equation*}
        0 \in \sigma_{i+1}\iprod{\subdiff F^*(\nexty)-K_y(\overnextx,\thisy)}{\nexty-\realopty}+\iprod{\nexty-\thisy}{\nexty-\realopty}.
    \end{equation*}
    Using $0 \in \subdiff F^*(\realopty) - K_y(\realoptx,\realopty)$ and following the steps leading to \eqref{eq:nextx-realoptx-from-thisx-realoptx}, we deduce
    \begin{equation}
        \label{eq:nexty-realoptx-from-thisx-realoptx}
        \norm{\nexty-\realopty} \le \norm{\thisy-\realopty}+C_y
    \end{equation}
    with $C_y\defeq\sigma_{i+1}\norm{K_y(\realoptx,\realopty)-K_y(\overnextx,\thisy)}$.

    We now proceed to derive bounds on $C_x$ and $C_y$ with the goal of bounding both \eqref{eq:nextx-realoptx-from-thisx-realoptx} and \eqref{eq:nexty-realoptx-from-thisx-realoptx} from above.
    Using \cref{ass:general}\,\ref{item:lipschitz-gradk},\,\ref{item:bounded-gradk}, and the mean value theorem applied to $K_x(\thisx,\cdot)$ and $K_y(\cdot,\thisy)$,
    \begin{equation*}
        \begin{aligned}
            C_x&\le\tau_i(\norm{K_x(\thisx,\thisy)-K_x(\thisx,\realopty)}+\norm{K_x(\thisx,\realopty)-K_x(\realoptx,\realopty)})
            \\
            &\le\tau_i(R_K\localRhoY[i]+L_x(\realopty)\localRhoX[i])=: R_x,
            \quad\text{and}\\
            C_y&\le\sigma_{i+1}(\norm{K_y(\realoptx,\realopty)-K_y(\realoptx,\thisy)}+\norm{K_y(\realoptx,\thisy)-K_y(\overnextx,\thisy)}
            \\
            &\le\sigma_{i+1}(L_y(\realoptx)\localRhoY[i]+R_K(\localRhoX[i]+\delta_x)) =: R_y,
        \end{aligned}
    \end{equation*}
    the latter under the assumption that $\norm{\overnextx-\realoptx}\le \localRhoX[i]+\delta_x$, which we now verify. First, by definition,
    \begin{equation*}
        \begin{aligned}
            \norm{\overnextx-\realoptx}^2&=\norm{\nextx-\realoptx+\omega_i(\nextx-\thisx)}^2\\
            &=\norm{\nextx-\realoptx}^2+\omega_i^2\norm{\nextx-\thisx}^2+2\omega_i\iprod{\nextx-\realoptx}{\nextx-\thisx}\\
            &=(1+\omega_i)\norm{\nextx-\realoptx}^2+\omega_i(1+\omega_i)\norm{\nextx-\thisx}^2-\omega_i\norm{\thisx-\realoptx}^2\\
            &\le(1+\omega_i)(\norm{\nextx-\realoptx}^2+\norm{\nextx-\thisx}^2)-\omega_i\norm{\thisx-\realoptx}^2.
        \end{aligned}
    \end{equation*}
    Applying \eqref{eq:step-length-estimate-raw} and \eqref{eq:nextx-realoptx-from-thisx-realoptx}, we obtain
    \begin{equation*}
        \begin{aligned}
            \norm{\overnextx-\realoptx}^2&\le (1+\omega_i)(2C_x\norm{\nextx-\realoptx}+\norm{\thisx-\realoptx}^2)-\omega_i\norm{\thisx-\realoptx}^2\\
            &\le 4C_x\norm{\nextx-\realoptx}+\norm{\thisx-\realoptx}^2\le 4C_x(\norm{\thisx-\realoptx}+C_x)+\norm{\thisx-\realoptx}^2\le(2C_x+\localRhoX[i])^2.
        \end{aligned}
    \end{equation*}
    The bound \eqref{eq:step-length-bounds-r} on $\tau_i$ implies that $C_x\le R_x\le \delta_x/2$ and hence that $\norm{\overnextx-\realoptx}\le \localRhoX[i]+\delta_x$.
    From \eqref{eq:nextx-realoptx-from-thisx-realoptx} we thus obtain $\norm{\nextx-\realoptx}\le \localRhoX[i]+\delta_x$. The bound \eqref{eq:step-length-bounds-r} on $\sigma_i$ then implies that $C_y\le R_y \le \delta_y$, which together with \eqref{eq:nexty-realoptx-from-thisx-realoptx} completes the proof. 
\end{proof}

To chain the applications of \cref{lemma:step-length-bounds-r} on each iteration $i \in \N$, we introduce the following assumption, for which we recall the notations in \cref{ass:general} as well as the definition of $\neighu (\metricRhoX,\metricRhoY)$ from \eqref{eq:neighu-definition}.
\begin{assumption}
    \label{ass:neighbourhood-compatibility}
    Suppose \cref{ass:general} holds near a solution $\realoptu \in \inv H(0)$. Given an initial iterate $u^0 \in X \times Y$, and initial step length parameters $\tau_0, \sigma_1, \omega_0>0$ as well as $0 < \delta \le \muConst < 1$ (to satisfy \eqref{eq:scalar-rules}), define the weighted distance
    \begin{equation}\label{eq:scalar-step-length-rmax}
        r_{\max}\defeq \sqrt{2\inv\delta(\norm{x^0-\realoptx}^2+\inv\nu\norm{y^0-\realopty}^2)}
        \quad\text{with}\quad
        \nu\defeq\sigma_1\omega_0\tau_0^{-1}.
    \end{equation}
    We then assume that there exist $\delta_x,\delta_y>0$ and $\localRhoY\ge r_{\max}\sqrt{\nu(1-\delta)\delta(\muConst-\delta)^{-1}}$ such that
    \begin{equation*}
        \B(\realoptx,r_{\max}+\delta_x)\times\B(\realopty,\localRhoY+\delta_y)\subseteq \neighu (\metricRhoX,\metricRhoY)
    \end{equation*}
    and that for all $i \in \N$ the step lengths $\tau_i, \sigma_i>0$ satisfy
    \begin{equation}
        \label{eq:scalar-step-length-bounds-r}
        \tau_i
        \le\frac{\delta_x}{2R_K\localRhoY+2L_x(\realopty)r_{\max}}
        \quad\text{and}\quad
        \sigma_{i+1}
        \le\frac{\delta_y}{L_y(\realoptx)\localRhoY+R_K(r_{\max}+\delta_x)}.
    \end{equation}
\end{assumption}

\begin{lemma}
    \label{lemma:neighborhood-compatible-iterations}
    For all $i \in \N$, suppose $\nextu$ solves \eqref{eq:ppext} and that all the conditions of \cref{thm:nonneg-penalty} 
    are satisfied for some $\metricRhoX, \metricRhoY>0$ and $\tilde\gamma_{G},\tilde\gamma_{F^*}\ge0$ except for the requirement $\thisu,\nextu,\overnextu \in \neighu (\metricRhoX,\metricRhoY)$. Then if
    \cref{ass:neighbourhood-compatibility} holds, $\{\thisu\}_{i \in \N}, \{\overnextu\}_{i \in \N} \subset \neighu (\metricRhoX,\metricRhoY)$.
\end{lemma}
\begin{proof}
    We define $\localRhoX[i] \defeq \frac1{\sqrt{\delta\tauTest_i}}\norm{u^0-\realoptu}_{\Test_{1}\Precond_{1}}$ and
    \begin{equation*}
        \neighu_i \defeq\bigl\{(x,y) \in X \times Y \,\bigm|\, \norm{x-\realoptx}^2+{\textstyle \frac{\sigmaTest_{i+1}}{\tauTest_{i}}		\frac{\muConst-\delta}{(1-\delta)\delta}}\norm{y-\realopty}^2\le\localRhoX[i]^2 \bigr\}.
    \end{equation*}
    Since the conditions \eqref{eq:scalar-rules} hold, we can apply \cref{crl:zimi} and the estimate \eqref{eq:zimi-estim} on $\Test_{i+1}\Precond_{i+1}$ to deduce that
    \begin{equation}
        \label{eq:neighui-zimi-rxi-inclusion}
        \{u \in X \times Y \mid \norm{u-\realoptu}_{\Test_{i+1}\Precond_{i+1}} \le \norm{u^0-\realoptu}_{\Test_{1}\Precond_{1}} \} \subset \neighu_i.
    \end{equation}
    From \eqref{eq:scalar-test-update}, we also deduce that $\tauTest_{i+1}\ge\tauTest_{i}$ and hence that $\localRhoX[i+1]\le\localRhoX[i]$.
    Consequently, if $\localRhoX[0] \le r_{\max}$, then
    \begin{equation}
        \label{eq:ball-neigh-neghui-inclusion}
        \B(\realoptx,\localRhoX[i]+\delta_{x})\times \B(\realopty, \localRhoY+\delta_y)\subseteq \B(\realoptx, r_{\max}+\delta_x) \times \B(\realopty, \localRhoY+\delta_y)\subseteq\neighu(\metricRhoX, \metricRhoY),
    \end{equation}
    so it will suffice to show that $u^i \in \B(\realoptx,\localRhoX[i]+\delta_{x})\times \B(\realopty, \localRhoY+\delta_y)$ for each $i \in \N$ to prove the claim.
    We do this in two steps. In the first step, we show that $\localRhoX[i] \le r_{\max}$ and
    \begin{equation}
        \label{eq:neighui-product-inclusion}
        \neighu_i\subseteq\B(\realoptx,\localRhoX[i])\times \B(\realopty, \localRhoY)
        \quad (i \in \N).
    \end{equation}
    In the second step, we show by induction that $u^i \in \neighu_i$ as well as $\overnextu \in \neighu(\metricRhoX, \metricRhoY)$ for $i \in \N$.

    \proofstep{Step 1}
    We first prove \eqref{eq:neighui-product-inclusion}.
    Since $\neighu_i\subseteq\B(\realoptx,\localRhoX[i])\times Y$, we only have to show that $\neighu_i\subseteq X\times \B(\realopty, \localRhoY)$. First, note that \eqref{eq:scalar-rules} and $\tilde\gamma_{G},\tilde\gamma_{F^*}\ge0$ imply  $\sigmaTest_{i+1}\ge \sigmaTest_{i}\ge \sigmaTest_{1}$ as well as $\tauTest_{i+1}\geq\tauTest_i \geq \tauTest_0=\eta_1\omega_0\tau_0^{-1}=\nu\sigmaTest_{1}$ for $\nu$ defined in \eqref{eq:scalar-step-length-rmax}. We then obtain from the definition of $\localRhoX[i]$ substituting $\Test_{1}\Precond_{1}$ from \eqref{eq:zimi-estim-eq} that
    \begin{equation*}
        \localRhoX[i]^2\delta\tauTest_i
        =\norm{u^0-\realoptu}^2_{\Test_{1}\Precond_{1}}
        =\nu\sigmaTest_{1}\norm{x^0-\realoptx}^2-2\eta_0\iprod{x^0-\realoptx}{K_{xy}(x^0,y^0)(y^0-\realopty)}+\sigmaTest_{1}\norm{y^0-\realopty}^2.
    \end{equation*}
    Using Cauchy's and Young's inequalities, the fact that $\tauTest_i\geq \nu\sigmaTest_{1}$, and the assumption that $\norm{K_{xy}(x^0,y^0)}\le R_K$, we arrive at
    \begin{equation*}
        \localRhoX[i]^2
        \le(2\nu\sigmaTest_{1}\norm{x^0-\realoptx}^2+(\sigmaTest_{1}+\eta_0^2\tauTest_0^{-1}R_K^2)\norm{y^0-\realopty}^2)(\delta\nu\sigmaTest_{1})^{-1}.
    \end{equation*}
    We obtain from \cref{crl:zimi} that $\eta_0^2\tauTest_0^{-1}R_K^2\le(1-\muConst)\sigmaTest_{1}\le\sigmaTest_{1}$ and hence that $\localRhoX[i]^2\le r_{\max}^2$. The assumption on $\localRhoY$ then yields for all $i \in \N$ that
    \begin{equation}
        \label{eq:scalar-step-length-bounds-u0-dual}
        \localRhoY^2 
        \ge r_{\max}^2\frac{\tauTest_{0}}{\sigmaTest_{1}}\frac{(1-\delta)\delta}{\muConst-\delta}
        \ge\frac{\localRhoX[0]^2\tauTest_{0}}{\sigmaTest_{i+1}}\frac{(1-\delta)\delta}{\muConst-\delta}
        =\frac{\localRhoX[i]^2\tauTest_{i}}{\sigmaTest_{i+1}}\frac{(1-\delta)\delta}{\muConst-\delta}.
    \end{equation}
    Thus \eqref{eq:neighui-product-inclusion} follows from the definition of $\neighu_i$.

    \proofstep{Step 2}
    We next show by induction that $\thisu \in \neighu_i$ and $\overnextu\in \neighu(\metricRhoX, \metricRhoY)$ for all $i \in \N$.
    Since \eqref{eq:neighui-zimi-rxi-inclusion} holds for $i=0$, we have that $u^0\in\neighu_0$. 
    Moreover, since in Step 1 we have $\localRhoX[0]\le r_{\max}$, the bound \eqref{eq:step-length-bounds-r} for $i=0$ follows from \eqref{eq:scalar-step-length-bounds-r}. This gives the induction basis.

    Suppose now that $u^N \in \neighu_N$.
    By \eqref{eq:neighui-product-inclusion}, we have that $u^N\in \B(\realoptx,\localRhoX[N])\times \B(\realopty, \localRhoY)$.
    Since again the bound \eqref{eq:step-length-bounds-r} for $i=N$ follows from \eqref{eq:scalar-step-length-bounds-r} and the bound $\localRhoX[N]\le r_{\max}$ follows from Step 1, we can apply \cref{lemma:step-length-bounds-r} to obtain 
    \begin{equation*}
        u^{N+1}\in \B(\realoptx,\localRhoX[N]+\delta_x)\times \B(\realopty, \localRhoY+\delta_y)\quad\text{and}\quad \bar{x}^{N+1}\in \B(\realoptx,\localRhoX[N]+\delta_x). 
    \end{equation*}
    By  \eqref{eq:ball-neigh-neghui-inclusion}, we have
    $
    \B(\realoptx,\localRhoX[N]+\delta_{x})\times \B(\realopty, \localRhoY+\delta_y)
    \subseteq \neighu(\metricRhoX, \metricRhoY)
    $
    and thus $u^{N+1}, \bar u^{N+1} \in \neighu(\metricRhoX,\metricRhoY)$.
    \Cref{thm:nonneg-penalty} now implies that \eqref{eq:convergence-result-main-h} is satisfied for $i\le N$ with $\Penalty_{N+1}\le0$, which together with \eqref{eq:convergence-result-main-h} and \eqref{eq:neighui-zimi-rxi-inclusion} yields that $u^{N+1}\in\neighu_{N+1}$. 
    This completes the induction step and hence the proof.
\end{proof}

\section{Convergence estimates}
\label{sec:convergence}

We are now ready to formulate the main convergence results of this paper based on the estimates derived above. 
First, based on \eqref{eq:scalar-gamma-rules-g} and \eqref{eq:scalar-gamma-rules-fstar}, strong convexity may be required if $\xi_x$ and $\xi_y$ have to be positive for \cref{ass:general} to be satisfied. Moreover, the neighborhood $\neighu (\metricRhoX,\metricRhoY)$ has to be small enough, as determined by the assumptions $\theta_x \ge \metricRhoY{\underline{\omega}}^{-1}$ and $\theta_y \ge \overline{\omega}\metricRhoX$ in the next results. 
This affects the admissible step lengths and how close we have to initialize $u^0$ via \cref{ass:neighbourhood-compatibility}.
After the next three main convergence results, we show that \cref{ass:neighbourhood-compatibility} is satisfied if we initialize close enough to a root $\realoptu \in \inv H(0)$. Hence, to apply the theorems in practice, we have to find constants for which \cref{ass:monotone-gf,ass:general} are satisfied, use these constants to bound and compute the step lengths as described in the theorems, and initialize close enough to $\realoptu$. 
In \cref{app:three-point-relaxation} we consider some relaxation of \cref{ass:general}\,\ref{item:k-nonlinear}, which in turn requires larger $\gamma_{G}$ and $\gamma_{F^*}$ instead of $\theta_x \ge \metricRhoY\underline{\omega}^{-1}$ and $\theta_y \ge \overline{\omega}\metricRhoX$.

The following theorem provides conditions sufficient for weak convergence of the sequence $\{\thisu\}_{i\in\N}$ generated by \cref{alg:gpdps}. Apart from technical requirements of \cref{thm:nonneg-penalty}, we require additional weak-to-strong continuity of the mapping $u\mapsto K_{yx}(u)x$. While its verification depends on the particular choice of $K$, it is trivially satisfied in two cases: (i) $X$ and $Y$ are finite-dimensional and $K_{yx}$ is continuous; or (ii) the mapping $u\mapsto K_{yx}(u)x$ is linear and compact.

\begin{theorem}[weak convergence: $\omega_i=1$]
    \label{thm:weak-convergence}
    Suppose \cref{ass:monotone-gf,ass:general,ass:neighbourhood-compatibility} hold for some $R_K>0$; $L_{yx} \ge 0$; 
    $\lambda_x,\lambda_y,\theta_x,\theta_y\geq 0$; and $\xi_x, \xi_y\in\R$ such that
    \begin{subequations}
        \label{eq:gamma-weak}
        \begin{align}
            \label{eq:gamma-weak-g}
            \xi_x &= \gamma_G,
            &
            \theta_y &\ge 2\metricRhoX,
            \\
            \label{eq:gamma-weak-fstar}
            \xi_y &= \gamma_{F^*},
            &
            \theta_x &\ge 2\metricRhoY.
        \end{align}
    \end{subequations} 
    For some $0<\delta<\muConst<1$, choose
    \begin{align}
        \label{eq:step-weak-convergence}
        \tau_i \equiv \tau < \frac{\delta}{\lambda_x+3L_{yx}\metricRhoY},
        \qquad
        \sigma_i\equiv\sigma \le \biggl(\frac{R_K^2\tau}{1-\muConst}+\lambda_y\biggr)^{-1},
        \quad\text{and}\quad
        \omega_i\equiv 1.
    \end{align}
    Furthermore, suppose that
    \begin{enumerate}[label=(\roman*)]
        \item\label{item:weak-convergence-cont} 
            $\thisu \weakto \bar{u}$ implies that $K_{yx}(\thisu)x\rightarrow K_{yx}(\bar{u})x$ for all $x\in X$,
    \end{enumerate}
    and either
    \begin{enumerate}[label=(ii\alph*)]
        \item\label{item:weak-convergence-weakstrong}
            the mapping $u\mapsto(K_{x}(u),K_y(u))$ is weak-to-strong continuous in $\neighu (\metricRhoX,\metricRhoY)$; or
        \item\label{item:weak-convergence-weakweak} 
            the mapping $u\mapsto(K_{x}(u),K_y(u))$ is weak-to-weak continuous, but  \cref{ass:monotone-gf} (monotone $\subdiff G$ and $\subdiff F^*$) and \cref{ass:general}\,\ref{item:k-nonlinear} (three-point condition on $K$) hold at \emph{any} weak limit $\bar{u}=(\bar{x},\bar{y})$ of $\{\thisu\}_{i\in\N}$ for the same choices of $\theta_x$ and $\theta_y$.
    \end{enumerate}
    Then the sequence $\{\thisu\}_{i\in\N}$ generated by \cref{alg:gpdps} converges weakly to some $\bar{u} \in \inv H(0)$ (possibly different from $\realoptu$).
\end{theorem}
Since it is assumed that $\theta_x \ge 2\metricRhoY$, we can replace $\metricRhoY$ by $\theta_x/2$ in the bound on $\tau$ in \eqref{eq:step-weak-convergence} if the latter is more readily available.

For constant $\tau$, $\sigma$, and $\omega=1$, we have to set $\sigmaTest_{i}\equiv \sigmaTest$ and $\tauTest_i\equiv\tauTest$ to satisfy  \eqref{eq:scalar-step-rules0}.
Consequently, applying \cref{crl:zimi} to bound $\Test_{i+1}\Precond_{i+1}$ from below will not help to prove \cref{thm:weak-convergence}.
We instead will make use of the following enhanced version of Opial's lemma.
\begin{lemma}[{\cite[Lemma A.2]{tuomov-nlpdhgm-redo}}]
    \label{lemma:opial-improved}
    Let $U$ be a Hilbert space, $\hat U \subset U$ (not necessarily closed or convex), and $\{\thisu\}_{i \in \N} \subset U$. Also let $A_i \in \linear(U; U)$ be self-adjoint and $A_i \ge \hat{\epsilon}^2 I$ for some $\hat{\epsilon}\ne0$ for all $i \in \N$. If the following conditions hold, then $\thisu \weakto \bar{u}$ in $U$ for some $\bar{u} \in \hat U$:
    \begin{enumerate}[label=(\roman*)]
        \item\label{item:opial-improved-non-increasing} The sequence $\{\norm{\thisu-\hat u}_{A_i}\}_{i \in \N}$ is nonincreasing for \emph{some} $\hat u \in \hat U$.
        \item\label{item:opial-improved-limit} All weak limit points of $\{\thisu\}_{i \in \N}$ belong to $\hat U$.
        \item\label{item:opial-improved-a-limit} There exists $C>0$ such that $\norm{A_{i}}\le C^2$ for all $i$, and for any weakly convergent subsequence $\{u_{i_k}\}_{k\in\N}$ there exists $A_\infty  \in \linear(U; U)$ such that $A_{i_k}u \to A_\infty u$ strongly in $U$ for all $u \in U$.
    \end{enumerate}
\end{lemma}
\begin{proof}[Proof of \cref{thm:weak-convergence}]
    We first verify \eqref{eq:scalar-rules} so that we can apply \cref{thm:nonneg-penalty} and \cref{lemma:neighborhood-compatible-iterations}. 
    We set $\sigmaTest_N\equiv1$, $\tauTest_N\equiv\sigma\tau^{-1}$, $\tilde{\gamma}_{G}=\tilde{\gamma}_{F^*}=0$ to satisfy \eqref{eq:scalar-step-rules0}, \eqref{eq:scalar-test-update}, \eqref{eq:scalar-gamma-rules-g} and \eqref{eq:scalar-gamma-rules-fstar} for $\omega=\underline{\omega}=\overline{\omega}=1$ and $\xi_x$, $\xi_y$, $\theta_x$, $\theta_y$ satisfying \eqref{eq:gamma-weak}.
    With the choice $\omega=1$, the bounds \eqref{eq:step-weak-convergence} thus ensure \eqref{eq:scalar-tau-sigma-rule}. 

    Hence \eqref{eq:scalar-rules} holds, which together with \cref{ass:neighbourhood-compatibility} and $\sigmaTest_1=1$ enables us to use \cref{lemma:neighborhood-compatible-iterations} to obtain $\{\thisu\}_{i \in \N}\in \neighu (\metricRhoX,\metricRhoY)$ and $\{\overnextx\}_{i \in \N}\in\B(\realoptx,\metricRhoX)$. Therefore there exists at least one weak limit point of $\{\thisu\}_{i \in \N}$.
    Moreover, \eqref{eq:zimi-estim-eq} yields self-adjointness of $\Test_{i+1}\Precond_{i+1}$ and since the bounds \eqref{eq:step-weak-convergence} are strict, \cref{thm:nonneg-penalty} holds with
    $\Penalty_{i+1}\le-\hat{\delta}\sum_{i=0}^{N}\norm{\nextu-\thisu}^2$ for some $\hat{\delta}>0$.	 

    We now verify the conditions of \cref{lemma:opial-improved} with $\hat U=\inv H(0)$ and $A_i=\Test_{i+1}\Precond_{i+1}$. 
    Estimate \eqref{eq:convergence-result-main-h} is valid for any starting iterate; thus setting $N=1$ and taking $\thisu$ instead of $u^0$, we obtain 
    $\norm{\nextu-\realoptu}^2_{\Test_{i+2}\Precond_{i+2}}
    \le \norm{\thisu-\realoptu}^2_{\Test_{i+1}\Precond_{i+1}}+\Penalty_{i+1}$ for any $\Penalty_{i+1}\le 0$ due to \cref{thm:nonneg-penalty}. This verifies \ref{item:opial-improved-non-increasing}. 
    Moreover, \ref{item:opial-improved-a-limit} follows from the assumed constant step lengths, \cref{ass:general}\,\ref{item:bounded-gradk}, and the assumption that $K_{yx}(\thisu)x\rightarrow K_{yx}(\bar{u})x$  for all $x\in  X$ if $\thisu \weakto \bar{u}$.

    Hence we only need to verify \ref{item:opial-improved-limit}, i.e., if a subsequence of $\{\thisu\}_{i \in \N}$ converges weakly to some $\bar u$, then $\bar u \in \inv H(0)$. We note that $\Step_{i+1}\equiv \Step$, and \eqref{eq:ppext} implies that $v_{i+1}\in \Step A(\nextu)$ for 
    \begin{align}
        \label{eq:weak-converging-adef}
        A(\nextu)&\defeq
        \begin{pmatrix}
            \subdiff G(\nextx)-\gamma_G(\nextx-\bar{x})\\
            \subdiff F^*(\nexty) -\gamma_{F^*}(\nexty-\bar{y})
        \end{pmatrix}
        \quad\text{and}
        \\
        \label{eq:weak-converging-subdiff}
        v_{i+1} &\defeq
        \begin{multlined}[t]
            \Step
            \begin{pmatrix}
                - K_x(\nextx,\nexty)-\gamma_G(\nextx-\bar{x})\\
                K_y(\nextx,\nexty)-\gamma_{F^*}(\nexty-\bar{y})
            \end{pmatrix}
            -\Precond_{i+1}(\nextu-\thisu)
            \\
            -
            \Step
            \begin{pmatrix}
                K_x(\thisx,\thisy)-K_x(\nextx,\nexty)+K_{xy}(\thisx,\thisy)(\nexty-\thisy)\\
                K_y(\nextx,\nexty)-K_y(\overnextx,\thisy)-K_{yx}(\thisx,\thisy)(\nextx-\overnextx)
            \end{pmatrix}.
        \end{multlined}
    \end{align}
    Therefore it suffices to show that if $u^{i_k}\weakto \bar{u} =(\bar{x}, \bar{y})$ for a subsequence, then 
    \begin{equation*}
        v_{i_k}\weakto \bar{v}\defeq
        \Step
        \begin{pmatrix}
            - K_x(\bar{x},\bar{y})\\
            K_y(\bar{x},\bar{y})
        \end{pmatrix}
        \qquad\text{and}\qquad
        \bar{v}\in WA(\bar{u}),
    \end{equation*}
    which by construction is equivalent to $\bar{u} \in \inv H(0)$. Note that $A$ is maximally monotone since it only involves subgradient mappings of proper convex lower semicontinuous functions due to \cref{ass:monotone-gf}. 
    Moreover, further use of \eqref{eq:convergence-result-main-h} shows that
    $\sum_{i=0}^{\infty} \frac{\hat{\delta}}{2}\norm{\nextu-\thisu}^2 < \infty$ and hence that $\norm{\nextu-\thisu} \to 0$. The last two terms in \eqref{eq:weak-converging-subdiff} thus converge strongly to zero.
    We therefore only have to consider the first term, for which we make a case distinction.
    \begin{enumerate}[label=(\alph*)]
        \item 
            If assumption \ref{item:weak-convergence-weakstrong} holds, we obtain that $v_{i_k}\to \bar{v}$, and the required inclusion $\bar{v}\in A(\bar{u})$ follows from the fact that the graph of the maximally monotone operator $A$ is sequentially weakly–strongly closed; see \cite[Proposition 16.36]{bauschke2017convex}.

        \item
            If assumption \ref{item:weak-convergence-weakweak} holds, then only $v_{i_k}\weakto \bar{v}$. In this case, we can apply the Brezis--Crandall--Pazy Lemma \cite[Corollary 20.59 (iii)]{bauschke2017convex} to obtain the required inclusion under the additional condition that $\limsup_{k\to \infty}~\iprod{u^{i_k}-\bar{u}}{v_{i_k}-\bar{v}}\le 0$.
            In our case, recalling that the last two terms of \eqref{eq:weak-converging-subdiff} converge strongly to zero, we have that 
            \begin{equation*}
                \limsup_{k\to \infty}~\iprod{u^{i_k}-\bar{u}}{v_{i_k}-\bar{v}}
                \le
                \limsup_{i\rightarrow\infty}~\iprod{u_i-\bar{u}}{v_i-\bar{v}} = \limsup_{i\rightarrow\infty}~q_i
            \end{equation*}
            for
            \begin{multline*}
                q_i  \defeq
                \iprod{K_x(\bar{x},\bar{y})-K_x(\nextx,\nexty)}{\nextx-\bar{x}}
                +\iprod{K_y(\nextx,\nexty)-K_y(\bar{x},\bar{y})}{\nexty-\bar{y}}
                \\
                -\gamma_{F^*}\norm{\nexty-\bar{y}}^2-\gamma_G\norm{\nextx-\bar{x}}^2.
            \end{multline*}
            Defining
            \begin{equation*}
                \begin{aligned}
                    d_i^x & \defeq
                    \iprod{K_y(\nextx,\nexty)-K_y(\bar{x},\nexty)+K_{yx}(\nextx,\nexty)(\bar{x}-\nextx)}{\nexty-\bar{y}}
                    \\
                    \MoveEqLeft[-1]
                    -\iprod{K_x(\thisx,\bar{y})-K_x(\bar{x},\bar{y})}{\nextx-\bar{x}}
                    -\gamma_G\norm{\nextx-\bar{x}}^2 
                    \quad\text{and}
                    \\[1ex]
                    d_i^y & \defeq
                    \iprod{K_x(\thisx,\bar{y})-K_x(\thisx,\thisy)-K_{xy}(\thisx,\thisy)(\bar{y}-\thisy)}{\nextx-\bar{x}}
                    \\
                    \MoveEqLeft[-1]
                    -\iprod{K_y(\nextx,\nexty)-K_y(\nextx,\thisy)+K_y(\bar{x},\bar{y})-K_y(\bar{x},\nexty)}{\nexty-\bar{y}}
                    -\gamma_{F^*}\norm{\nexty-\bar{y}}^2,
                \end{aligned}
            \end{equation*}
            we rearrange and estimate
            \begin{equation}
                \label{eq:weak-qi-estimate}
                \begin{aligned}[t]
                    q_i
                    &=d_i^x + d_i^y
                    +\iprod{K_y(\nextx,\nexty)-K_y(\nextx,\thisy)}{\nexty-\bar{y}}
                    \\\MoveEqLeft[-1]
                    +\iprod{(K_{xy}(\nextx,\nexty)-K_{xy}(\thisx,\thisy))(\thisy-\bar{y})}{\nextx-\bar{x}}
                    \\
                    \MoveEqLeft[-1]
                    +\iprod{K_x(\thisx,\thisy)-K_x(\nextx,\nexty)+K_{xy}(\nextx,\nexty)(\nexty-\thisy)}{\nextx-\bar{x}}
                    \\
                    &
                    \le d_i^x + d_i^y + O(\norm{\nextu-\thisu}).
                \end{aligned}
            \end{equation}
            Using $\xi_x=\gamma_G$, $\xi_y=\gamma_{F^*}$, \eqref{eq:lipschitz-bound}, and both \cref{ass:monotone-gf} and \cref{ass:general}\,\ref{item:k-nonlinear} at $\bar{u}$, we estimate $q_i \le O(\norm{\nextu-\thisu})$ as
            \begin{equation*}
                \begin{aligned}
                    d_i^x & \le
                    (\norm{\nexty-\bar{y}}-\theta_x)\norm{K_y(\nextx,\nexty)-K_y(\bar{x},\nexty)+K_{yx}(\nextx,\nexty)(\bar{x}-\nextx)}
                    \le 0,
                    \\
                    d_i^y &\le(\norm{\nextx-\bar{x}}-\theta_y)
                    \norm{K_x(\thisx,\bar{y})-K_x(\thisx,\thisy)-K_{xy}(\thisx,\thisy)(\bar{y}-\thisy)}
                    \le 0.
                \end{aligned}
            \end{equation*}
            In the last bounds we used $\theta_x \ge 2\metricRhoY$, $\theta_y \ge 2\metricRhoX$, and $\norm{\nexty-\bar{y}} \le 2\metricRhoY$ because both $\norm{\nexty-\realopty}\leq \metricRhoY$ and $\norm{\realopty-\bar{y}} \le \metricRhoY$; likewise, $\norm{\nextx-\bar{x}} \le 2\metricRhoX$.
            Since $\norm{\nextu-\thisu} \to 0$, we obtain that $\limsup_{i\rightarrow\infty}~q_i \le 0$. The Brezis--Crandall--Pazy Lemma thus yields the desired inclusion $\bar{v}\in A(\bar{u})$.
    \end{enumerate}
    Hence in both cases, $\bar{u} \in \inv H(0)$ and the condition \ref{item:opial-improved-limit} of \cref{lemma:opial-improved} is satisfied. Applying \cref{lemma:opial-improved}, we obtain the claim.
\end{proof}

We now provide convergence rates under additional assumptions of strong convexity of $G$ and/or $F^*$, although we still allow non-convexity of the overall problem through $K$. To be specific, we require that we can take the acceleration or step length update factors $\tilde\gamma_G>0$ and/or $\tilde\gamma_{F^*}>0$ in \eqref{eq:scalar-gamma-rules-g} and \eqref{eq:scalar-gamma-rules-fstar}, respectively. 
Let us start with $\tilde\gamma_G>0$, which is the case, for instance, when $G$ is strongly convex and \eqref{eq:k-nonlinear-kx} holds with $\xi_x=0$.
Since we obtain \emph{a fortiori} strong convergence from the rates, we do not require the additional assumptions on $K$ introduced in \cref{thm:weak-convergence}; on the other hand, we only obtain convergence of the primal iterates. 
Similar to the linear case of \cite{tuomov-nlpdhgm-redo}, the step length choice follows directly from having to satisfy \eqref{eq:scalar-test-update} and the desire to keep the right-hand side of the $\sigma$-rule \eqref{eq:scalar-tau-sigma-rule} constant.
\begin{theorem}[convergence rates under acceleration: $\omega_i=1$]
    \label{thm:acceleration-nlpdhgm}
    Suppose \cref{ass:monotone-gf,ass:general,ass:neighbourhood-compatibility} hold for some $R_K>0$; $L_{yx} \ge 0$; $\lambda_x,\lambda_y,\theta_x,\theta_y\geq 0$; and $\xi_x, \xi_y\in\R$ such that for some $\tilde\gamma_G > 0$,
    \begin{subequations}
        \label{eq:gamma-acc}
        \begin{align}
            \label{eq:gamma-acc-g}
            \xi_x &= \gamma_G-\tilde\gamma_G,
            &
            \theta_y &\ge \metricRhoX,
            \\
            \label{eq:gamma-acc-fstar}
            \xi_y &= \gamma_{F^*},
            &
            \theta_x &\ge \metricRhoY.
        \end{align}
    \end{subequations} 
    Choose
    \begin{equation}
        \label{eq:acceleration-updates}
        \tau_{i+1}=\frac{\tau_{i}}{1+2\tilde\gamma_G\tau_i},
        \quad
        \sigma_{i+1}\equiv\sigma,
        \quad\text{and}\quad
        \omega_i\equiv 1,
    \end{equation}
    satisfying for some $0<\delta\le\muConst<1$ the bounds
    \begin{equation}
        \label{eq:acceleration-init}
        0<\tau_0\le \frac{\delta}{\lambda_x+3L_{yx}\metricRhoY}
        \quad\text{and}\quad
        0<\sigma \tau_0 \le \frac{1-\muConst}{R_K^2}.
    \end{equation}
    Then $\norm{x^N-\realoptx}^2$ converges to zero at the rate $O(1/N)$.
\end{theorem}

\begin{proof}
    We again first verify \eqref{eq:scalar-rules} so that we can apply \cref{thm:nonneg-penalty} and \cref{lemma:neighborhood-compatible-iterations}.
    Setting $\sigmaTest_{i}\equiv 1$, $\eta_i\equiv\sigma$, $\tauTest_i\defeq\sigma\tau_i^{-1}$, and $\tilde\gamma_{F^*}=0$, \eqref{eq:scalar-step-rules0} follows from the $\sigma$-rule of \eqref{eq:acceleration-updates} and the choice of $\sigmaTest_i$, $\eta_i$, and $\tauTest_i$.
    Using \eqref{eq:acceleration-updates} and $\tau_i\defeq\sigma\tauTest_i^{-1}$, we obtain $\tauTest_{i+1}=(1+2\tilde\gamma_G\tau_i)\tauTest_i$, and hence \eqref{eq:scalar-test-update} follows.
    Since $\tau_i\le \tau_0$ and $\lambda_y\ge 0$, \eqref{eq:scalar-tau-sigma-rule} follows from \eqref{eq:acceleration-init} and $\omega_i=1$.
    Furthermore, \eqref{eq:scalar-gamma-rules-g} and \eqref{eq:scalar-gamma-rules-fstar} are satisfied due to the assumed bounds \eqref{eq:gamma-acc} on $\xi_x$, $\xi_y$, $\theta_x$, and $\theta_y$ taking $\overline{\omega}=\underline{\omega}=1$.

    We can thus apply \cref{thm:nonneg-penalty} and \cref{lemma:neighborhood-compatible-iterations} to arrive at \eqref{eq:convergence-result-main-h} for $\Penalty_{i+1} = 0$. 
    We now estimate the convergence rate from \eqref{eq:convergence-result-main-h} by bounding $\Test_{N+1}\Precond_{N+1}$ from below. Using \cref{crl:zimi}, we obtain $\delta \tauTest_N \norm{x^N-\realoptx}^2 \le\norm{u^0-\realoptu}^2_{\Test_{1}\Precond_{1}}$.
    Moreover, 
    \begin{equation*}
        \tauTest_{N+1}=(1+2\tilde\gamma_G\tau_N)\tauTest_N=\tauTest_N+2\tilde\gamma_G\sigma = \ldots =	 \tauTest_1 +2N\tilde\gamma_G\sigma,
    \end{equation*}
    which yields the claim.
\end{proof}

\begin{theorem}[linear convergence: $\omega_i<1$]
    \label{thm:linear-convergence}
    Suppose \cref{ass:monotone-gf,ass:general,ass:neighbourhood-compatibility} hold for some $R_K>0$; $L_{yx} \ge 0$;
    $\lambda_x,\lambda_y\geq 0$; and $\tilde\gamma_G,\tilde\gamma_{F^*}>0$ as well as
    \begin{subequations}
        \label{eq:gamma-lin}
        \begin{align}
            \label{eq:gamma-lin-g}
            \xi_x &= \gamma_G-\tilde\gamma_G,
            &
            \theta_y &\ge \omega\metricRhoX,
            \\
            \label{eq:gamma-lin-fstar}
            \xi_y &= \gamma_{F^*}-\tilde\gamma_{F^*},
            &
            \theta_x &\ge \metricRhoY\omega^{-1}
        \end{align}
    \end{subequations} 
    with
    \begin{equation}
        \tau_i \equiv \tau,\quad \sigma_i \equiv \sigma\defeq\tau\tilde{\gamma}_G\tilde{\gamma}_{F^*}^{-1},\quad \text{and} \quad \omega_i\equiv\omega\defeq(1+2\tilde{\gamma}_{G}\tau)^{-1}.
    \end{equation}
    Assume for some $0<\delta\le\muConst<1$ the bound
    \begin{equation}
        \label{eq:linear-convergence-step-bounds}
        \tau \le \min\Biggl\{
            \frac{\delta}{\lambda_x+3L_{yx}\metricRhoY},\,
            \frac{2\tilde\gamma_{F^*}\tilde\gamma_{G}^{-1}}
            {\lambda_y+\sqrt{\lambda_y^2+4\tilde\gamma_{F^*}\tilde\gamma_{G}^{-1}(R_K^2(1-\muConst)^{-1}+2\tilde\gamma_G\lambda_y)}}
        \Biggr\}.
    \end{equation}
    Then $\norm{u^N-\realoptu}^2$ converges to zero with the linear rate $O(\omega^N)$.
\end{theorem}

\begin{proof}
    We will use \cref{thm:nonneg-penalty} and \cref{lemma:neighborhood-compatible-iterations}, for both of which we need to verify \eqref{eq:scalar-rules} first.
    We set $\overline{\omega}\defeq\underline{\omega}\defeq\omega$,
    \begin{align*}
        \sigmaTest_N&\defeq\omega(1+2\sigma\tilde{\gamma}_{F^*})^N=\omega(1+2\tilde\gamma_G\tau)^N=\omega^{1-N},\quad\text{and}\\ \tauTest_N&\defeq\omega\sigma\tau^{-1}(1+2\tau\tilde{\gamma}_{G})^N=\omega^{1-N}\sigma\tau^{-1}.
    \end{align*}
    Then $\sigmaTest_1=1$ and $\sigmaTest_N\sigma=\tauTest_N\tau$, verifying \eqref{eq:scalar-step-rules0} and \eqref{eq:scalar-test-update}.
    We next observe that substituting $\sigma_i = \tau\tilde{\gamma}_G\tilde{\gamma}_{F^*}^{-1}$, the first bound of \eqref{eq:scalar-tau-sigma-rule} is tantamount to requiring
    \begin{equation*}
        \tau\bigl(\tau R_K^2\inv{(1-\muConst)}+\lambda_y\omega^{-1}\bigr) \le \tilde{\gamma}_{F^*}\tilde{\gamma}_{G}^{-1}.
    \end{equation*}
    Substituting $\omega=(1+2\tilde{\gamma}_{G}\tau)^{-1}$, this in turn is equivalent to 
    \begin{equation*}
        \left(R_K^2(1-\muConst)^{-1}+2\tilde\gamma_G\lambda_y\right)\tau^2+\lambda_y\tau-\tilde\gamma_{F^*}\tilde\gamma_{G}^{-1}\le 0,
    \end{equation*} 
    which after solving a quadratic inequality for $\tau$ yields the second bound of \eqref{eq:linear-convergence-step-bounds}.
    Since $\omega\le 1$, the first bound of \eqref{eq:linear-convergence-step-bounds} gives the second bound of \eqref{eq:scalar-tau-sigma-rule}.
    Finally, \eqref{eq:scalar-gamma-rules-g} and \eqref{eq:scalar-gamma-rules-fstar} follow directly from \eqref{eq:gamma-lin} with $\underline{\omega}=\overline{\omega}=\omega$.

    Since \cref{ass:neighbourhood-compatibility} and \eqref{eq:scalar-rules} hold, we can apply \cref{lemma:neighborhood-compatible-iterations} to obtain $\{\thisu\}_{i \in \N}\in \neighu (\metricRhoX,\metricRhoY)$ and $\{\overnextx\}_{i \in \N}\in\B(\realoptx,\metricRhoX)$.
    Moreover, \eqref{eq:zimi-estim-eq} yields self-adjointness of $\Test_{i+1}\Precond_{i+1}$.
    Consequently, we can apply \cref{thm:nonneg-penalty} and \cref{lemma:neighborhood-compatible-iterations} to arrive at \eqref{eq:convergence-result-main-h} for any $\Penalty_{i+1}\le0$. 

    We now estimate the convergence rate from \eqref{eq:convergence-result-main-h} by bounding $\Test_{N+1}\Precond_{N+1}$ from below. Using \cref{crl:zimi}, we obtain that
    \begin{equation}
        \label{eq:linearconv-result}
        \frac{1}{\omega^{N}}\left(\delta\sigma\omega\tau^{-1}\norm{x^{N}-\realoptx}^2
        +\frac{\muConst-\delta}{1-\delta}\norm{y^{N}-\realopty}^2\right)
        \le
        \norm{u^0-\realoptu}^2_{\Test_{1}\Precond_{1}}.
    \end{equation}
    Since $\omega \in (0, 1)$, this gives the claimed linear convergence rate through the exponential growth of $1/\omega^N$.
\end{proof}

\begin{remark}
    If $K(x,y)=\iprod{A(x)}{y}$ for some $A \in C^1(X)$, then $K_x(x,y)=[\grad A(x)]^*y$ and $K_y(x,y)=A(x)$ with $L_y(x)=0$ and $L_{yx}=L$ for $L$ a local Lipschitz factor of $\grad A$.
    Furthermore, \cref{ass:general}, the step length bounds, and the update rules required in \cref{thm:weak-convergence} or \ref{thm:linear-convergence} reduce to the corresponding ones introduced in \cite{tuomov-nlpdhgm-redo} for this case.
    As for acceleration, \cref{thm:acceleration-nlpdhgm} now gives a weaker convergence rate of $O(1/N)$ compared to $O(1/N^2)$ in \cite[Theorem 4.3]{tuomov-nlpdhgm-redo}. This is due to \eqref{eq:scalar-tau-sigma-rule} requiring $\sigma_i$ to be bounded whenever $\lambda_y>0$, even when $\tau_i$ goes to zero.
\end{remark}

Before we conclude this section, we refine \cref{ass:neighbourhood-compatibility} by showing that its implicit requirements do not add any additional step length bounds provided the starting point is sufficiently close to $\realoptu$. 
\begin{proposition}
    \label{pr:locality-infinite}
    Under the assumptions of \cref{thm:weak-convergence}, \ref{thm:acceleration-nlpdhgm}, or \ref{thm:linear-convergence}, suppose that $\metricRhoX,\metricRhoY>0$.
    Then there exists $\varepsilon>0$ such that \cref{ass:neighbourhood-compatibility} holds whenever the initial iterate $u^0=(x^0,y^0)$ satisfies
    \begin{equation}
        \label{eq:neighborhood-start-close}
        r_{\max} \defeq \sqrt{2\inv\delta(\norm{x^0-\realoptx}^2+\inv\nu\norm{y^0-\realopty}^2)}
        \le\varepsilon
        \quad \text{with}\quad
        \nu\defeq\sigma_1\omega_0\tau_0^{-1}.
    \end{equation}
\end{proposition}

\begin{proof}
    We take $\mu$, $\delta$, $\sigma_i$, $\tau_i$, and $\omega_i$ as they are defined in the corresponding \cref{thm:weak-convergence}, \ref{thm:acceleration-nlpdhgm}, or \ref{thm:linear-convergence}, and $L_x(\realopt{y})$, $L_y(\realopt{x}), R_K$ from \cref{ass:general}.
    We need to show that there exist $\delta_x,\delta_y>0$ and $\localRhoY\ge r_{\max}\sqrt{\nu(1-\delta)\delta(\muConst-\delta)^{-1}}$ such that \eqref{eq:scalar-step-length-bounds-r} holds and 
    \begin{equation}
        \label{eq:locality-inclusion}
        \B(\realoptx,r_{\max}+\delta_x)\times\B(\realopty,\localRhoY+\delta_y)\subseteq \neighu (\metricRhoX,\metricRhoY).
    \end{equation}
    Let $\varepsilon>0$ and set $\localRhoY \defeq \varepsilon\sqrt{\nu(1-\delta)\delta(\muConst-\delta)^{-1}}$ as well as $\delta_x \defeq \sqrt{\varepsilon}$ and $\delta_y \defeq \metricRhoY-\localRhoY$.
    Observing \eqref{eq:neighborhood-start-close}, we then see both that $\delta_y>0$ and that \eqref{eq:locality-inclusion} holds for $\varepsilon>0$ sufficiently small. Furthermore, \eqref{eq:neighborhood-start-close} yields that $r_{\max}\le \varepsilon$ in \cref{lemma:neighborhood-compatible-iterations}. Let
    \begin{equation*}
        c_\varepsilon\defeq
        \min\left\{
            \frac{\delta_x}{2R_K r_y+2L_x(\realopty)r_{\max}},
            \frac{\delta_y}
        {L_y(\realoptx) r_y+R_K(r_{\max}+\delta_x)}\right\}.
    \end{equation*}
    Since $r_y, r_{\max} = O(\varepsilon)$, $\delta_x=\sqrt{\varepsilon}$, and $\delta_y>\metricRhoY/2>0$ for $\varepsilon>0$ small enough, we see that $c_\varepsilon \to \infty$ as $\varepsilon\to 0$. Comparing the definition of $c_\varepsilon$ to \eqref{eq:scalar-step-length-bounds-r}, we therefore see that the latter holds for any given $\tau_0>0$ and $\sigma_i\equiv\sigma>0$ by taking $\varepsilon>0$ sufficiently small. Since in \cref{thm:weak-convergence,thm:acceleration-nlpdhgm,thm:linear-convergence} we have $\tau_i\le\tau_0$, the inequalities \eqref{eq:scalar-step-length-bounds-r} hold.
\end{proof}

\section{Numerical examples}
\label{sec:numerical}

Finally, we illustrate the applicability of the proposed approach for the example applications described in \cref{sec:applications}. The Julia implementation used to generate the following results is on Zenodo \cite{nlpdhgm-general-code}.

\subsection{An elliptic Nash equilibrium problem}\label{sec:numerical:nash}

Our first example illustrates the reformulation from \cref{sec:applications:nash} for the two-player elliptic Nash equilibrium problem from \cite{borzikanzow2013}. Here the action space of each player is $L^2(\Omega)$ for a bounded domain $\Omega\subset \R^d$ with boundary $\partial\Omega$. To avoid confusion with the spatial variable, we will in this subsection denote the primal variable with $u$ and the dual variable with $v$.
The set of admissible strategies is 
\begin{equation*}
    X_k = \left\{w\in L^2(\Omega): w(x)\in [a,b] \text{ a.e. } x\in\Omega\right\} \qquad (k=1,2).
\end{equation*}
For a set of strategies $u:=(u_1,u_2)\in X=X_1\times X_2$, the payout function for each player is
\begin{equation*}
    \phi_k(u_1,u_2) = \frac12\norm{S(u_1,u_2)-z_k}_{L^2(\Omega)}^2 + \frac{\alpha_k}2\norm{B_ku_k}_{L^2(\Omega)}^2\qquad(k=1,2),
\end{equation*}
where $\alpha_k>0$, $z_k\in L^2(\Omega)$ are given target states, $S:L^2(\Omega)^2 \to L^2(\Omega)$ maps $u=(u_1,u_2)$ to the solution $y$ to the elliptic boundary value problem
\begin{equation}\label{eq:elliptic}
    \left\{
        \begin{aligned}
            -\Delta y &= B_1u_1 + B_2u_2 +f\quad&&\text{on }\Omega,\\
            y &= 0 \quad&&\text{on }\partial\Omega,
        \end{aligned}
    \right.
\end{equation}
$B_k:L^2(\Omega)\to L^2(\Omega)$ are control operators which are here chosen as
\begin{equation*}
    B_kw := 
    \begin{cases} 
        w(x) & \text{if }x\in \omega_k,\\ 
        0 & \text{if }x\notin \omega_k,
    \end{cases}
\end{equation*}
for some control domains $\omega_k\subset\Omega$,
and $f$ is a common source term. Following \cref{sec:applications:nash}, the corresponding Nash equilibrium problem \eqref{eq:nash} can then be solved by applying \cref{alg:gpdps} to 
\begin{align*}
    G&:L^2(\Omega)^2\to \overline\R, & G(u_1,u_2) &= \delta_X(u_1,u_2),\\
    F^*&:L^2(\Omega)^2\to \overline\R, & F^*(v_1,v_2) &= \delta_X(v_1,v_2),\\
    K&:L^2(\Omega)^2 \times L^2(\Omega)^2 \to \R, & K((u_1,u_2),(v_1,v_2)) &= 
    [\phi_1(u_1,u_2) - \phi_1(v_1,u_2)] \\
    && \MoveEqLeft[-1] + [\phi_2(u_1,u_2) - \phi_2(u_1,v_2)]
\end{align*}

To implement the algorithm, we need explicit forms of the proximal mappings for $G$ and $F^*$ and of the partial derivatives of $K$.
Since $G=F^*=\delta_X$ for $X=X_1\times X_2$, we have 
\begin{equation*}
    \prox_{\tau G}(w) = \prox_{\sigma F^*}(w) = \proj_X(w) = \left(\proj_{X_1}(w_1),\proj_{X_2}(w_2)\right)
\end{equation*} 
for the metric projections onto the convex sets $X_k$ given pointwise almost everywhere by
\begin{equation*}
    [\proj_{X_k}(w_k)](x) = 
    \begin{cases} 
        b & \text{if } w_k(x)>b,\\
        w_k(x) & \text{if } w_k(x) \in [a,b],\\
        a & \text{if } w_k(x)<a.
    \end{cases}
\end{equation*}
It remains to address the computation of $K_u(u,v)$ and $K_v(u,v)$. 
Using adjoint calculus and the linearity of the adjoint equation, we have that
\begin{equation*}
    K_u(u,v) = \begin{pmatrix}
        p_1(u,v) + \alpha_1u_1\\
        p_2(u,v) + \alpha_2u_2
    \end{pmatrix},
    \qquad
    K_v(u,v) = \begin{pmatrix}
        q_1(u,v) - \alpha_1v_1\\
        q_2(u,v) - \alpha_2v_2
    \end{pmatrix},
\end{equation*}
where $p_1(u,v)=:p_1$ and $p_2(u,v)=:p_2$ are the solutions to the equations
\begin{align*}
    -\Delta p_1 &= 2S(u_1,u_2)-S(u_1,v_2) - z_1,\\
    -\Delta p_2 &= 2S(u_1,u_2)-S(v_1,u_2) - z_2,
    \intertext{and $q_1(u,v)=:q_1$ and $q_2(u,v)=:q_2$ are the solutions to the equations}
    -\Delta q_1 &= -S(v_1,u_2) + z_1,\\
    -\Delta q_2 &= -S(u_1,v_2) + z_2,
\end{align*}
all with homogeneous Dirichlet conditions.
Hence, every iteration of \cref{alg:gpdps} requires nine solutions of a partial differential equation (recall that $K_v$ is evaluated at $(\overnextu,\this v)$, while $K_u$ is evaluated at $(u^i,v^i)$). Since $S$ and hence $K_u$ and $K_v$ are affine in $u$ and $v$, the assumptions of \cref{thm:weak-convergence} are satisfied for sufficiently small step lengths. Since neither $F^*$ nor $G$ are strongly convex, no acceleration is possible.

\bigskip

For our numerical tests we follow \cite{borzikanzow2013} and consider a finite-difference discretization of \eqref{eq:elliptic} on $\Omega=(0,1)^2$ with $N$ nodes in each direction, 
\begin{equation*}
    \omega_1 = (0,1)\times (0,1/2),\qquad
    \omega_2 = (0,1)\times (1/2,1),
\end{equation*}
as well as $a=-0.5$, $b=0.5$, and $\alpha_i = 1$. Using the method of manufactured solutions, $z_1$, $z_2$, and $f$ are chosen such that the solution $u^*=(u_1^*,u_2^*)$ of the Nash equilibrium problem is known a priori; see \cref{fig:enep:ue}. By construction, the saddle point then satisfies $v^* = u^*$ and hence $\Psi(u^*,v^*)=0$.
\begin{figure}[t]
    \centering
    \begin{subfigure}[t]{0.495\textwidth}
        \centering
        \includegraphics[width=\textwidth]{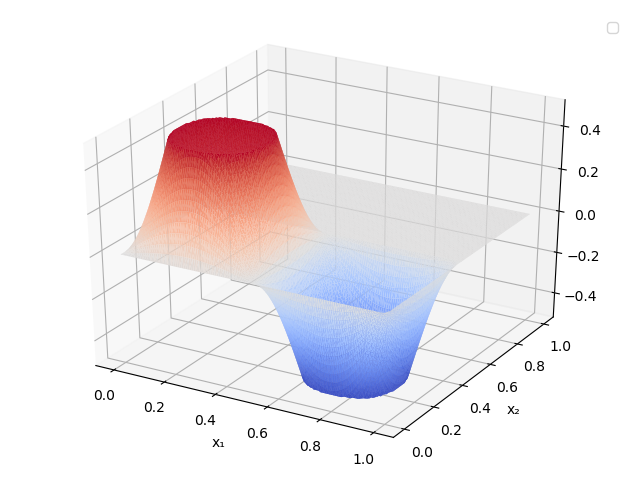}
        \caption{$u_1^*$}
        \label{fig:enep:ue1}
    \end{subfigure}\hfill
    \begin{subfigure}[t]{0.495\textwidth}
        \centering
        \includegraphics[width=\textwidth]{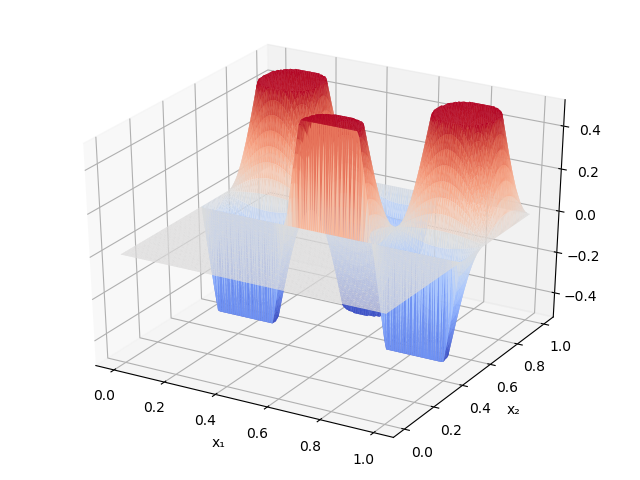}
        \caption{$u_2^*$}
        \label{fig:enep:ue2}
    \end{subfigure}
    \caption{Constructed solution for elliptic NEP example ($N=128$)}
    \label{fig:enep:ue}
\end{figure}
\begin{table}[t]
    \centering
    \caption{Results for elliptic NEP example for different $N$}
    \label{tab:enep}
    \begin{tabular}{c
            S[table-format=1.3e-2,round-precision=3]
            S[table-format=1.3e-2,round-precision=3]
            S[table-format=1.3e-2,round-precision=3]
            S[table-format=1.3e-2,round-precision=3]
            S[table-format=1.3e-2,round-precision=3]
        }
        \toprule
        $i$ & {$N=64$} & {$N=128$} & {$N=256$} & {$N=512$} & {$N=1024$} \\
        \midrule
        $1$ & 1.298e-01 & 1.319e-01 & 1.330e-01 & 1.335e-01 & 1.338e-01 \\
        $2$ & 3.889e-06 & 4.048e-06 & 4.074e-06 & 4.088e-06 & 4.097e-06 \\
        $3$ & 3.835e-10 & 3.977e-10 & 4.010e-10 & 4.026e-10 & 4.032e-10 \\
        $4$ & 3.811e-14 & 3.952e-14 & 3.986e-14 & 4.001e-14 & 4.008e-14 \\
        $5$ & 3.787e-18 & 3.928e-18 & 3.963e-18 & 3.977e-18 & 3.985e-18 \\
        \bottomrule
    \end{tabular}
\end{table}
Since the Lipschitz constants for $K$ and its derivatives are not available, we simply take the parameters in \cref{alg:gpdps} as $\sigma_{i+1}\equiv\sigma=1.0$, $\tau_i\equiv \tau=0.99$, and $\omega=1.0$. The results of the algorithm for different values of $N\in\{64,128,256,512,1024\}$ are shown in \cref{tab:enep}, which reports the distance of the primal-dual iterates $(u^i,v^i)$ to the exact solution. As can be seen, the iteration converges in each case to machine precision within $5$ iterations, and the convergence behavior is virtually identical. This demonstrates the mesh independence expected from an algorithm for which convergence can be shown in function spaces.

\subsection{\texorpdfstring{$\ell^0$-TV denoising}{ℓ⁰-TV denoising}}
\label{sec:numerical:potts}

Our next example concerns the $\ell^0$-TV denoising or segmentation problem from \cref{sec:applications:potts}.
Recall that we can solve the (Huber-regularized) $\ell^0$-TV problem \eqref{eq:potts} by applying \cref{alg:gpdps} to
\begin{align*}
    G&:\R^{N_1\times N_2}\to \R, & G(x)&=\frac{1}{2\alpha}\norm{x-f}_2^2,\\
    F_\gamma^*&:\R^{N_1\times N_2\times 2}\to \R, & F_\gamma^*(y) &= \frac{\gamma}{2}\norm{y}_2^2,\\
    K_p&:\R^{N_1\times N_2}\times \R^{N_1\times N_2\times 2}\to\R, & K_p(x,y) &= \kappa_p(D_h x, y),
\end{align*}
for $p\in \{1,\infty\}$ and $\gamma\geq0$, where $D_h:\R^{N_1\times N_2}\to \R^{N_1\times N_2\times 2}$ is the discrete gradient.
We write $H_\gamma$ for $H$ defined in \eqref{eq:h} corresponding to $F^*=F_\gamma^*$.
Since $G$ and $F_\gamma^*$ are quadratic, a simple computation shows that
\begin{equation*}
    \prox_{\tau G}(x) = \frac{1}{1+\frac{\tau}{\alpha}}\left(x+\frac\tau\alpha f\right),
    \quad\text{and}\quad
    \prox_{\sigma F_\gamma^*}(y) = \frac{1}{1+\gamma\sigma} y,
\end{equation*}
where all operations are to be understood componentwise. For the derivatives of $K_p$, we have by the chain rule
\begin{equation}
    \label{eq:potts:kderiv}
    K_x(x,y) = D_h^T \kappa_{p,z}(D_hx,y),\qquad K_y(x,y) = \kappa_{p,y}(D_hx,y),
\end{equation}
where $D_h^T$ is the discrete (negative) divergence.
For the partial derivatives of $\kappa_{p,z}(z,y)$ and $\kappa_{p,y}(z,y)$, we again distinguish the cases $p=1$ and $p=\infty$:
\begin{align*}
    \intertext{For $p=1$, we have componentwise}
    [\kappa_{1,z}(z,y)]_{ijk} &= 2(1-z_{ijk}y_{ijk})y_{ijk},\\
    [\kappa_{1,y}(z,y)]_{ijk} &= 2(1-z_{ijk}y_{ijk})z_{ijk}.
    \intertext{For $p=\infty$, we have componentwise}
    [\kappa_{\infty,z}(z,y)]_{ijk} &= 2(1-z_{ij1}y_{ij1}-z_{ij2}y_{ij2})y_{ijk},\\
    [\kappa_{\infty,y}(z,y)]_{ijk} &= 2(1-z_{ij1}y_{ij1}-z_{ij2}y_{ij2})z_{ijk}.
\end{align*}

It remains to choose valid step sizes for \cref{alg:gpdps}, for which the next result gives useful estimates.
We recall from \cite{chambolle2004algorithm} that a forward differences discretization of the gradient operator satisfies $\norm{D_h}_2 \le \sqrt{8}/h$.
Recalling \eqref{eq:potts:kderiv} and the definitions of $G$ and $F_\gamma^*$, a critical point $(\realoptx, \realopty) \in \inv H_\gamma(0)$ satisfies
\begin{equation}
    \label{eq:potts:oc}
    0 = \inv\alpha(\realoptx - f) + D_h^T \kappa_{p,z}(D_h\realoptx,\realopty)
    \quad\text{and}\quad
    \gamma \realopty = \kappa_{p,y}(D_h\realoptx,\realopty).
\end{equation}
For brevity, we set
\begin{align*}
    \realopt m_x & \defeq \max_{i j} \abs{[D_h \realoptx]_{\freevar i j}}_2
    \quad\text{and}
    &
    \realopt m_y & \defeq \max_{ij} \abs{\realopty_{\freevar i j}}_2
    && (p=\infty),
    \\
    \realopt m_x & \defeq \max_{kij} \abs{[D_h \realoptx]_{k i j}}
    \quad\text{and}
    &
    \realopt m_y & \defeq \max_{kij} \abs{\realopty_{k i j}}
    && (p=1).
\end{align*}
Using the results of \cref{app:step-potts} we verify the fundamental \cref{ass:general}.
\begin{corollary}
    \label{cor:potts-k-condition}
    Let $K=K_p$ for either $p=1$ or $p=\infty$.
    Choose $L \ge \norm{D_h}_2$ and $R_K > 2 L$.
    Then \cref{ass:general} holds for some $\theta_x,\theta_y>0$ and $\metricRhoX,\metricRhoY>0$ with
    \begin{align*}
        L_x(y)&=2L^2\norm{y}_2^2,
        &
        L_y(x)&=2L^2\norm{x}_2^2,
        \qquad
        &
        L_{yx}&=4L\sup\nolimits_{y \in \B(\realopty, \metricRhoY)} \norm{y}_2,
    \end{align*}
    and the constants $\xi_x,\xi_y >0$, $\lambda_x,\lambda_y \ge 0$ satisfying
    \begin{equation}
        \label{eq:potts:k-condition:xi}
        \xi_x\lambda_x > 2L^2 (\inv L \lambda_x + \realopt m_y^2)\realopt m_y^2
        \quad\text{and}\quad
        \lambda_y > \realopt m_x^2.
    \end{equation}
\end{corollary}
\begin{proof}
    We consider only $p=\infty$ as the proof for $p=1$ is similar.
    Taking $\tilde R_K > 2$, \cref{lemma:rho-kcondition} applied componentwise shows that the operator $\kappa_p$ satisfies \cref{ass:general} for some $\tilde\theta_z, \tilde\theta_y>0$ and $\tilde\rho_x, \tilde\rho_y>0$ (depending on $\tilde R_K$) when we take 
    \[
        \tilde L_z(y) = 2\norm{y}_2^2,\qquad \tilde L_y(z) = 2\norm{z}_2^2,\quad\text{and}\quad \tilde L_{yz} = 4\max_{y \in \B(\realopty, \tilde\metricRhoY)} \norm{y}_2.
    \]
    Moreover, the constants $\tilde \xi_z,\tilde \xi_y \in \R$ and $\tilde \lambda_z,\tilde \lambda_y \ge 0$ need to satisfy $\tilde\xi_z\tilde\lambda_z > \max_{ij} 2(\lambda_z +  \norm{\realopty_{\freevar i j}}^2)\norm{\realopty_{\freevar i j}}^2$ as well as $\tilde\xi_y > 0$ and $\tilde\lambda_y > \max_{i j} \norm{\realoptz_{\freevar i j}}^2$ for $\realoptz=D_h\realoptx$.

    By \cref{lemma:oper-combo-kcondition} on compositions with a linear operator, we can now take 
    \[
        \begin{aligned}
            &R_K = \tilde R_K L, && \metricRhoX=L^{-1}\tilde\rho_x, && \metricRhoY=\tilde\rho_y, &&
            \xi_x=L\tilde\xi_z, && \xi_y=\tilde\xi_y,\\
            &\lambda_x=L\tilde\lambda_z, &&
            \lambda_y=\tilde\lambda_y, && \theta_x=\tilde\theta_z, && \theta_y=\tilde\theta_yL^{-1},\\
            &L_x(y)=L^2\tilde L_z(y), && L_y(x)=\tilde L_y(D_hx), && L_{yx}=L^2\tilde L_{yz}.
        \end{aligned}
    \]
    These give the claim.
\end{proof}

We now obtain from  \cref{thm:linear-convergence} the following estimate.

\begin{corollary}
    Suppose \cref{ass:monotone-gf} holds. Choose $L \ge \norm{D_h}_2$. For some $\tilde\gamma_G \in (0, \inv\alpha)$ and $\tilde\gamma_{F^*} \in (0, \gamma)$, take $\xi_x=\inv\alpha-\tilde\gamma_G$ and $\xi_y=\gamma-\tilde\gamma_{F^*}$ as well as $\lambda_x, \lambda_y \ge 0$ such that \eqref{eq:potts:k-condition:xi} holds.
    For some $0<\delta\le\muConst<1$, take $\sigma=\tau\tilde{\gamma}_G\tilde{\gamma}_{F^*}^{-1}$ and $\omega\defeq(1+2\tilde{\gamma}_{G}\tau)^{-1}$ as well as 
\begin{equation}
    \label{eq:potts:step-linear}
        \tau < \min\left\{
            \frac{\delta}{\lambda_x}
        ,
        \frac{2\tilde\gamma_{F^*}\tilde\gamma_{G}^{-1}}
        {\lambda_y+\sqrt{\lambda_y^2+4\tilde\gamma_{F^*}\tilde\gamma_{G}^{-1}(4 L^2(1-\muConst)^{-1}+2\tilde\gamma_G\lambda_y)}}
        \right\}.
\end{equation}
    Then $\norm{u^N-\realoptu}^2$ converges to zero with the linear rate $O(\omega^N)$ provided $u^0$ is close enough to $\realoptu$.
\end{corollary}

\begin{proof}
    The assumptions $\tilde\gamma_G \in (0, \inv\alpha)$ and $\tilde\gamma_{F^*} \in (0, \gamma)$ ensure $\xi_x,\xi_y>0$.
    Since we have assumed \eqref{eq:potts:k-condition:xi}, \cref{cor:potts-k-condition} yields \cref{ass:general} for any $R_K > 2 L$ and some $\theta_x,\theta_y>0$.
    We next use \cref{thm:linear-convergence}, whose conditions we need to verify. 
    First, taking $\metricRhoX,\metricRhoY>0$ ensures that $\theta_x \ge \metricRhoY\omega^{-1}$ and $\theta_y \ge \omega \metricRhoX$. Furthermore, the strict inequality in \eqref{eq:potts:step-linear} implies \eqref{eq:linear-convergence-step-bounds} for sufficiently small $\metricRhoY>0$.
    Finally, \cref{pr:locality-infinite} ensures that we can satisfy \cref{ass:neighbourhood-compatibility} by taking $u^0$ sufficiently close to $\realoptu$.
    The rest of the conditions we have assumed explicitly, so we can apply \cref{thm:linear-convergence} to finish the proof.
\end{proof}

Recall that \cref{ass:monotone-gf} is a second-order growth condition at the critical point $(\realoptx,\realopty)$, which is a common assumption needed to show convergence of algorithms for non-convex optimization problems.
To calculate the upper bounds on $\tau$ in \eqref{eq:potts:step-linear}, we need to find $\lambda_x, \lambda_y \ge 0$ satisfying \eqref{eq:potts:k-condition:xi}.
For this, in turn, we need to estimate $\realopt m_x$ and $\realopt m_y$.
To do this, note that the critical point conditions \eqref{eq:potts:oc} imply
\begin{equation}
    \label{eq:potts:realopty}
    \realopty_{\freevar ij}=\frac{2[D_h \realoptx]_{\freevar i j}}{2\abs{[D_h \realoptx]_{\freevar i j}}_2^2+\gamma}
    \quad (p=\infty)
    \quad\text{and}\quad
    \realopty_{kij}=\frac{2[D_h \realoptx]_{ki j}}{2\abs{[D_h \realoptx]_{ki j}}^2+\gamma}
    \quad (p=1).
\end{equation}
Since $t \mapsto t/(t+\gamma)$ is increasing, we can estimate $\realopt m_y$ based on $\realopt m_x$.
Since any solution of the Potts problem should be piecewise constant with very few intensity quantization levels, we can estimate $\realopt m_x$ as the expected maximal jump between neighboring pixels. We take this as 100\% of the dynamic range for safety.
In practice, as a practical choice of $\gamma>0$ will likely not satisfy $\xi_x > 2L\realopt m_y^2$, we use an over-approximation $\bar\gamma \defeq 10 \ge \gamma$ in \eqref{eq:potts:realopty}. We remark that we thus cannot guarantee convergence of \cref{alg:gpdps} for small $\gamma>0$; however, we demonstrate below that these estimates can still lead to useful step sizes for such cases.
Similarly, we do not have an estimate for the unknown local neighborhood of convergence; we compensate for this by taking small $\delta=0.1$ in \eqref{eq:potts:step-linear}. As the results below demonstrate, with these parameters we nevertheless observe convergence for the reasonable starting point $u^0=(x^0,y^0)$ with $x^0 = f$ and $y^0 \equiv 0$.

\bigskip

We illustrate the performance of the algorithm and the effects of the choice of $p$. As a test image, we choose ``blobs'' from the ImageJ framework \cite{ImageJ} with size $N_1\times N_2 = 256\times 254$, see \cref{fig:blobs:orig}. 
We set $\alpha = 1$ and $\gamma = 10^{-3}$ (cf. \cref{fig:ell-gamma}) and use the accelerated step size rule from \cref{thm:linear-convergence}.
To do this, we need to satisfy \eqref{eq:potts:step-linear} for the primal step length $\tau$.
We discretize the problem such that $h=1$ and hence $L=\sqrt{8}$. Furthermore, we set $\tilde\gamma_{F^*} = \gamma/100$ and $\tilde \gamma_G=\inv{\tilde\alpha}$ for $\tilde \alpha = 10\alpha$.
The above estimates then lead to the step length parameters
\begin{description}
    \item[$p=1$:] $\tau=1.04085\cdot10^{-3}$, $\sigma = 1.04085$, $\omega = 0.99480$;
    \item[$p=\infty$:] $\tau=5.51922\cdot10^{-4}$, $\sigma = 0.551922$, $\omega = 0.99724$.
\end{description}
Since the exact solution $(\realoptx,\realopty)$ is not available here, we instead use $x^{\max} \defeq x^{N_{\max}}$ for $N_{\max} = 10^6$ and similarly $y^{\max}$ as references for computing errors.
The corresponding reference images $x^{\max}$ obtained from \cref{alg:gpdps} after $N_{\max}=10^6$ iterations are shown in \cref{fig:blobs:aniso,fig:blobs:iso} for $p=1$ and $p=\infty$, respectively. 
While the evaluation of the formulation and the algorithm in the context of image processing is outside of the scope of this work, we briefly comment on the difference between $p=1$ and $p=\infty$. As can be seen by comparing the two images, the results are very similar. However, since diagonal jumps are penalized less for $p=\infty$, the isotropic Huber--Potts model is better able to preserve small light blobs such as the one indicated by the red circles. The edges of the blobs are also noticeably smoother.
\begin{figure}[t]
    \centering
    \begin{subfigure}[t]{0.33\textwidth}
        \centering
        \begin{tikzpicture}
            \node[anchor=south west, inner sep=0] (image)  at (0,0) {\includegraphics[width=\textwidth]{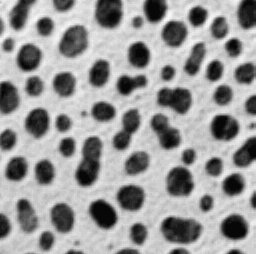}};
            \begin{scope}[x={(image.south east)},y={(image.north west)}]
                \draw[red,very thick] (0.535,0.91) circle(0.05);
            \end{scope}
        \end{tikzpicture}
        \caption{original image $f$}
        \label{fig:blobs:orig}
    \end{subfigure}\hfill
    \begin{subfigure}[t]{0.33\textwidth}
        \centering
        \begin{tikzpicture}
            \node[anchor=south west, inner sep=0] (image)  at (0,0) {\includegraphics[width=\textwidth]{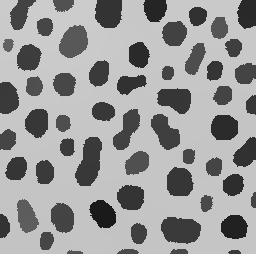}};
            \begin{scope}[x={(image.south east)},y={(image.north west)}]
                \draw[red,very thick] (0.535,0.91) circle(0.05);
            \end{scope}
        \end{tikzpicture}
        \caption{$x^{\max}$ for $p=1$}
        \label{fig:blobs:aniso}
    \end{subfigure}\hfill
    \begin{subfigure}[t]{0.33\textwidth}
        \centering
        \begin{tikzpicture}
            \node[anchor=south west, inner sep=0] (image)  at (0,0) {\includegraphics[width=\textwidth]{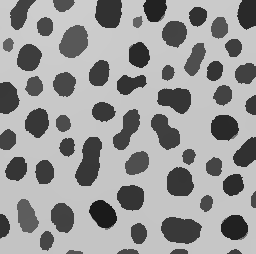}};
            \begin{scope}[x={(image.south east)},y={(image.north west)}]
                \draw[red,very thick] (0.535,0.91) circle(0.05);
            \end{scope}
        \end{tikzpicture}
        \caption{$x^{\max}$ for $p=\infty$}
        \label{fig:blobs:iso}
    \end{subfigure}
    \caption{$\ell^0$-TV denoising: original image $f$ and reference iterates $x^{\max}$ for anisotropic ($p=1$) and isotropic ($p=\infty$) Huber--Potts model}
    \label{fig:blobs:img}
\end{figure}
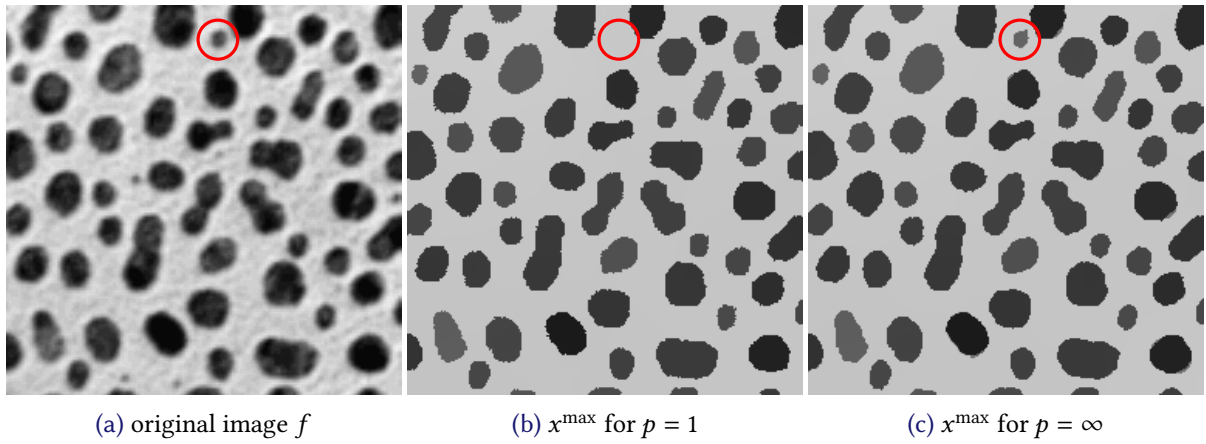
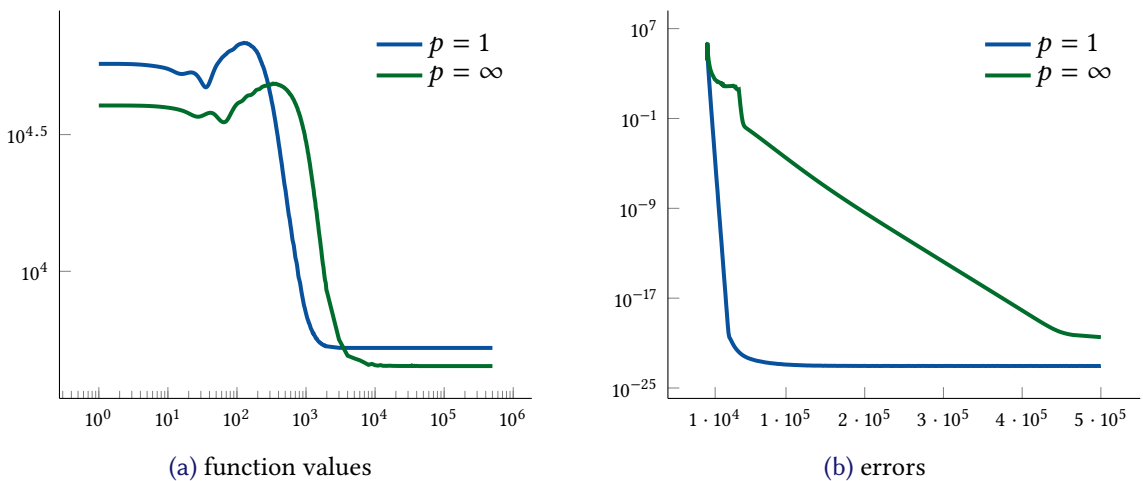
\begin{figure}[t]
    \centering
    \begin{subfigure}[t]{0.495\textwidth}
        \centering
        \input{blobs_fun.tikz}
        \caption{function values}
        \label{fig:blobs:fun}
    \end{subfigure}
    \hfill
    \begin{subfigure}[t]{0.495\textwidth}
        \centering
        \input{blobs_err.tikz}
        \caption{errors}
        \label{fig:blobs:err}
    \end{subfigure}
    \caption{$\ell^0$-TV denoising: convergence of function values $F_\gamma(x^N)+G(x^N)$ and errors $\norm{x^N-x^{\max}}^2+\norm{y^N-y^{\max}}^2$}
    \label{fig:blobs:conv}
\end{figure}

The convergence behavior of the method for both choices of $p$ over $N_{\max}/2 = 5\cdot10^5$ iterations is given in \cref{fig:blobs:conv}. For the function values, we observe in \cref{fig:blobs:fun} the usual fast decrease in the beginning of the iteration, after which the values stagnate. Nevertheless, the errors continue to decrease down to machine precision at the predicted linear rate.
The convergence behavior for $p=1$ and $p=\infty$ is similar, although the linear convergence for $p=\infty$ is with a significantly smaller constant.
We remark that visually, the iterates in both cases are indistinguishable from the reference images already after $N=10^4$ iterations. This is consistent with \cref{fig:blobs:err} since the total error is dominated by the dual component, which acts as an edge indicator; small changes of the boundaries of the blobs during the iteration will, even for small gray value changes, lead to large differences in the dual variable.

\section{Conclusion}

Using generalized conjugation, some non-smooth non-convex optimization problems can be transformed into saddle-point problems involving non-smooth convex functionals and a smooth non-convex-concave coupling term. For such problems, a generalized primal--dual proximal splitting method can be applied that converges weakly under step length conditions if a local quadratic growth condition is satisfied near a saddle-point. Under additional strong convexity assumptions on the functionals (but not the coupling term and hence the problem), convergence rates for accelerated algorithms can be shown. This approach can be applied to elliptic Nash equilibrium problems and for the anisotropic and isotropic Huber-regularized Potts models, as the numerical examples illustrate. 
Future work is concerned with further evaluating and comparing the performance of the proposed algorithm for these examples.

\section*{Acknowledgments}

In the first stages of the research T.~Valkonen and S.~Mazurenko were supported by the EPSRC First Grant EP/P021298/1, ``PARTIAL Analysis of Relations in Tasks of Inversion for Algorithmic Leverage''.
Later T.~Valkonen was supported by the Academy of Finland grants 314701 and 320022.
C.~Clason was supported by the German Science Foundation (DFG) under grant Cl~487/2-1.
We thank the anonymous reviewers for insightful comments.

\section*{\texorpdfstring{\normalsize}{}A data statement for the EPSRC}

The source codes for the numerical experiments are on Zenodo at \cite{nlpdhgm-general-code}.

\appendix

\section{Reductions of the three-point condition}
\label{app:three-point-reduction}

The following two propositions demonstrate that \cref{ass:general}\,\ref{item:k-nonlinear} is closely related to standard second-order optimality conditions, i.e., that the Hessian is positive definite at the solution $\realoptu$.
\begin{proposition}
    \label{pr:strongly-convex-primal}
    Suppose \cref{ass:general}\,\ref{item:lipschitz-gradk} (locally Lipschitz gradients of $K$) holds in some neighborhood $\neighu$ of $\realoptu$, and for some $\xi_x\in\R$, $\gamma_x>0$,
    \begin{equation}
        \label{eq:strongly-convex-primal}
        \xi_x\norm{x-\realoptx}^2
        +\iprod{K_x(x,\realopty)-K_x(\realoptx,\realopty)}{x-\realoptx}
        \ge\gamma_x\norm{x-\realoptx}^2
        \quad
        ((x,y)\in\neighu).
    \end{equation}
    Then \eqref{eq:k-nonlinear-kx} holds in $\neighu$ with $\theta_x=2(\gamma_x-\alpha)L_{yx}^{-1}$, and $\lambda_x=L_x(\realopty)^2(2\alpha)^{-1}$ for any $\alpha\in(0,\gamma_x]$.
\end{proposition}

\begin{proof}
    An application of Cauchy's and Young's inequalities with any factor $\alpha>0$, \cref{ass:general}\,\ref{item:lipschitz-gradk}, and \eqref{eq:strongly-convex-primal} yields the estimate
    \begin{equation*}\begin{aligned}
        \iprod{K_x(x',\realopty)-K_x(\realoptx,\realopty)}{x-\realoptx}+\xi_x\norm{x-\realoptx}^2
        &=
        \iprod{K_x(x,\realopty)-K_x(\realoptx,\realopty)}{x-\realoptx}+\xi_x\norm{x-\realoptx}^2
        \\
        &\quad
        +\iprod{K_x(x',\realopty)-K_x(x,\realopty)}{x-\realoptx}
        \\
        &
        \ge(\gamma_x-\alpha)\norm{x-\realoptx}^2-L_x(\realopty)^2(4\alpha)^{-1} \norm{x'-x}^2.
    \end{aligned}\end{equation*}
    At the same time, using \eqref{eq:lipschitz-bound},
    \begin{equation*}
        \norm{K_y(\realoptx,y)-K_y(x,y)-K_{yx}(x,y)(\realoptx-x)}\le \frac{L_{yx}}2\norm{x-\realoptx}^2.
    \end{equation*}
    Therefore \eqref{eq:k-nonlinear-kx} holds if we take $\theta_x\le 2(\gamma_x-\alpha)L_{yx}^{-1}$ and $\lambda_x=L_x(\realopty)^2(2\alpha)^{-1}$.
\end{proof}

\begin{proposition}
    \label{pr:strongly-concave-dual}
    Suppose \cref{ass:general}\,\ref{item:lipschitz-gradk} (locally Lipschitz gradients of $K$) holds in some neighborhood $\neighu$ of $\realoptu$ with $L_y(x)\le \bar{L}_y$, and that
    \begin{equation*}
        \norm{K_{xy}(x,y')-K_{xy}(x,y)} \le L_{xy}\norm{y'-y}
        \quad
        (u,u'\in\neighu)
    \end{equation*}
    for some constant $L_{xy} \ge 0$. Assume, moreover, for some $\xi_y\in\R$, $\gamma_y>0$ that
    \begin{equation}
        \label{eq:strongly-concave-dual}
        \xi_y\norm{y-\realopty}^2
        +\iprod{K_y(\realoptx,\realopty)-K_y(\realoptx,y)}{y-\realopty}
        \ge\gamma_y\norm{y-\realopty}^2
        \quad
        ((x,y)\in\neighu).
    \end{equation}
    Then \eqref{eq:k-nonlinear-ky} holds in $\neighu$ with $\theta_y=2(\gamma_y-\alpha_1)(1+\alpha_2)^{-1} L_{xy}^{-1}$, and $\lambda_y=(\bar{L}_y^2(2\alpha_1)^{-1}+(1+\alpha_2^{-1})L_{xy}\theta_y)$ for any $\alpha_1\in(0,\gamma_y]$, $\alpha_2>0$.
\end{proposition}

\begin{proof}
    An application of Cauchy's and Young's inequalities with any factor $\alpha>0$, \cref{ass:general}\,\ref{item:lipschitz-gradk}, and \eqref{eq:strongly-concave-dual} yields the estimate
    \begin{multline*}
        \iprod{K_y(x,y)-K_y(x,y')+K_y(\realoptx,\realopty)-K_y(\realoptx,y)}{y-\realopty}
        +\xi_y\norm{y-\realopty}^2
        \\
        \begin{aligned}
            &\ge
            \iprod{K_y(x,y)-K_y(x,y')}{y-\realopty}+\gamma_y\norm{y-\realopty}^2
            \\
            &
            \ge(\gamma_y-\alpha_1)\norm{y-\realopty}^2-\frac{L_y(x)^2}{4\alpha_1} \norm{y'-y}^2.
        \end{aligned}
    \end{multline*}
    At the same time, using \eqref{eq:lipschitz-bound} and Young's inequality for any $\alpha_2>0$,
    \begin{equation*}
        \begin{aligned}
            \norm{K_x(x',\realopty)-K_x(x',y')-K_{xy}(x',y')(\realopty-y')}
            &\le \frac{L_{xy}}2\norm{y'-\realopty}^2
            \\
            &\le \frac{L_{xy}}2(1+\alpha_2)\norm{y-\realopty}^2
            +\frac{L_{xy}}2(1+\alpha_2^{-1})\norm{y'-y}^2.
        \end{aligned}
    \end{equation*}
    Therefore \eqref{eq:k-nonlinear-ky} holds if we take 
    $\theta_y\le 2\frac{\gamma_y-\alpha_1}{(1+\alpha_2)L_{xy}}$
    and
    $\lambda_y=\frac{\bar{L}_y^2}{2\alpha_1}+(1+\alpha_2^{-1})L_{xy}\theta_y$.
\end{proof}

\section{Relaxations of the three-point condition}
\label{app:three-point-relaxation}

In all the results of this paper, \cref{ass:general}\,\ref{item:k-nonlinear} can be generalized to the following three-point condition similar to the one used in \cite{tuomov-nlpdhgm-redo}.
\begin{assumption}
    \label{ass:general-p2}
    The functional $K(x,y)\in C^1(X\times Y)$ and there exists a neighborhood
    \begin{equation}
        \label{eq:neighu-definition-p2}
        \neighu (\metricRhoX,\metricRhoY) \defeq
        (\B(\realoptx, \metricRhoX) \isect \neighx_G)
        \times
        (\B(\realopty, \metricRhoY) \isect \neighy_{F^*}),
    \end{equation} 
    for some $\metricRhoX,\metricRhoY>0$ such that for all $u',u \in \neighu(\metricRhoX,\metricRhoY)$, the following property holds:
    \begin{enumerate}[label={(iv*)}]
        \item \label{item:k-nonlinear-p2} (three-point condition) 
            There exist $\theta_x,\theta_y > 0$, $\lambda_x,\lambda_y\ge0$, $\xi_x,\xi_y\in\R$, and $p_x,p_y\in[1,2]$ such that
            \begin{subequations}%
                \label{eq:k-nonlinear-threepoint-p2}
                \begin{align}
                    \label{eq:k-nonlinear-kx-p2}
                    &\begin{aligned}[t]
                        &\iprod{K_x(x',\realopty)-K_x(\realoptx,\realopty)}{x-\realoptx}
                        +\xi_x\norm{x-\realoptx}^2
                        \\
                        \MoveEqLeft[-1]\ge
                        \theta_x\norm{K_y(\realoptx,y)-K_y(x,y)-K_{yx}(x,y)(\realoptx-x)}^{p_x}
                        -\frac{\lambda_x}2\norm{x-x'}^2, \quad\text{and}
                    \end{aligned}
                    \\[1ex]
                    &\begin{aligned}[t]
                        \label{eq:k-nonlinear-ky-p2}
                        &\iprod{K_y(x,y)-K_y(x,y')+K_y(\realoptx,\realopty)-K_y(\realoptx,y)}{y-\realopty}
                        +\xi_y\norm{y-\realopty}^2
                        \\
                        \MoveEqLeft[-1]\ge
                        \theta_y\norm{K_x(x',\realopty)-K_x(x',y')-K_{xy}(x',y')(\realopty-y')}^{p_y}
                        -\frac{\lambda_y}2\norm{y-y'}^2.
                    \end{aligned}
                \end{align}%
            \end{subequations}
    \end{enumerate}
\end{assumption}

This assumption introduces $p_x$ and $p_y$ in $[1,2]$, while in \cref{ass:general}\,\ref{item:k-nonlinear} we had $p_x=p_y=1$. For instance, in \cite[Appendix B]{tuomov-nlpdhgm-redo} we verified \cref{ass:general-p2} with $p_x=2$ for the case $K(x,y)=\iprod{A(x)}{y}$ for the reconstruction of the phase and amplitude of a complex number.
This relaxation mainly affects the proof of Step 4 in \cref{thm:nonneg-penalty}, which now requires a few intermediate derivations.
\begin{corollary}
    \label{crl:nonneg-penalty-p2}
    The results of \cref{thm:nonneg-penalty} continue to hold if \cref{ass:general}\,\ref{item:k-nonlinear} is replaced with \cref{ass:general-p2}\,\ref{item:k-nonlinear-p2} for some $p_x,p_y\in [1,2]$, where in case $p_y \in (1, 2]$, \eqref{eq:scalar-gamma-rules-g} is replaced by
    \begin{subequations}%
        \label{eq:scalar-gamma-rules-p2}
        \begin{align}
            \label{eq:scalar-gamma-rules-g-p2}
            \gamma_G  & \ge \tilde\gamma_G + \xi_x +
            \frac{p_y-1}{(\theta_yp_y^{p_y}\rho_x^{p_y-2}\overline{\omega}^{-1})^{\frac{1}{p_y-1}}},
            \\
            \intertext{and in case $p_x \in (1, 2]$, \eqref{eq:scalar-gamma-rules-fstar} is replaced by}
            \label{eq:scalar-gamma-rules-fstar-p2}
            \gamma_{F^*} & \ge \tilde\gamma_{F^*}+ \xi_y +
            \frac{p_x-1}{(\underline{\omega}\theta_xp_x^{p_x}\rho_y^{p_x-2})^{\frac{1}{p_x-1}}}.
        \end{align}%
    \end{subequations}
\end{corollary}

\begin{proof}
    The beginning of the proof follows the exact same steps as in the proof of \cref{thm:nonneg-penalty} up until \eqref{eq:dw-estimate}.
    We now use \cref{ass:general-p2}\,\ref{item:k-nonlinear-p2} to further bound $D_x$ and $D_y$ similarly to \eqref{eq:dx-estimate} and \eqref{eq:dy-estimate}. From \eqref{eq:k-nonlinear-kx-p2},
    \begin{equation}
        \label{eq:dx-pre-estimate-p2}
        \begin{aligned}[t]
            D_x&\ge
            \theta_x\norm{K_y(\realoptx,\nexty)-K_y(\nextx,\nexty)-K_{yx}(\nextx,\nexty)(\realoptx-\nextx)}^{p_x}
            -\frac{\lambda_x}2\norm{\nextx-\thisx}^2
            \\\MoveEqLeft[-1]
            -\norm{\nexty-\realopty}\norm{K_y(\realoptx,\nexty)-K_y(\nextx,\nexty)-K_{yx}(\nextx,\nexty)(\realoptx-\nextx)}\omega_i^{-1}.
        \end{aligned}
    \end{equation}
    The following generalized Young's inequality for any positive $a,b,p$ and $q$ such that $q^{-1}+p^{-1}=1$ allows for our choice of varying $p_x\in[1,2]$:
    \begin{equation}
        \label{eq:generalized-young}
        ab=\left(ab^{\frac{2-p}{p}}\right)b^{2\frac{p-1}{p}}
        \le \frac{1}{p}\left(ab^{\frac{2-p}{p}}\right)^p+\frac{1}{q}b^{2\frac{p-1}{p}q}
        =\frac{1}{p}a^pb^{2-p}+\biggl(1-\frac{1}{p}\biggr)b^2.
    \end{equation}
    Applying this inequality  with $p=p_x$,
    \begin{align*}
        a & \defeq (\zeta_x p_x)^{-1/2} \norm{K_y(\realoptx,\nexty)-K_y(\nextx,\nexty)-K_{yx}(\nextx,\nexty)(\realoptx-\nextx)},\quad\text{and}
        \\
        b & \defeq (\zeta_x p_x)^{1/2}\norm{\nexty-\realopty},
    \end{align*}
    for any $\zeta_x>0$ to the last term of \eqref{eq:dx-pre-estimate-p2}, we arrive at the estimate
    \begin{equation*}
        \begin{aligned}[t]
            D_x&\ge 
            \theta_x\norm{K_y(\realoptx,\nexty)-K_y(\nextx,\nexty)-K_{yx}(\nextx,\nexty)(\realoptx-\nextx)}^{p_x}
            -\frac{\lambda_x}{2}\norm{\nextx-\thisx}^2
            \\
            \MoveEqLeft[-1]
            -\frac{\norm{\nexty-\realopty}^{2-p_x}}{p_x^{p_x}\omega_i\zeta_x^{p_x-1}}
            \norm{K_y(\realoptx,\nexty)-K_y(\nextx,\nexty)-K_{yx}(\nextx,\nexty)(\realoptx-\nextx)}^{p_x}
            \\
            \MoveEqLeft[-1]
            -\frac{p_x-1}{\omega_i}\zeta_x \norm{\nexty-\realopty}^2.
        \end{aligned}
    \end{equation*}
    We now use $\nextu \in \neighu (\metricRhoX,\metricRhoY)$ for some $\metricRhoX,\metricRhoY \ge 0$, and $\omega_i^{-1} \le \underline{\omega}^{-1}$ to obtain
    \begin{equation}
        \label{eq:dx-thetax-p2}
        \theta_x-\norm{\nexty-\realopty}^{2-p_x}(p_x^{p_x}\omega_i\zeta_x^{p_x-1})^{-1}
        \ge\theta_x-\metricRhoY^{2-p_x}(p_x^{p_x}\underline{\omega}\zeta_x^{p_x-1})^{-1}.
    \end{equation}
    If $p_x=1$, we use the assumed inequality $\theta_x \ge \metricRhoY\underline{\omega}^{-1}$ from \eqref{eq:scalar-gamma-rules-fstar} to show that the right-hand side of \eqref{eq:dx-thetax-p2} is non-negative for any $\zeta_x>0$. 
    Otherwise we take $\zeta_x \defeq (\underline{\omega}\theta_xp_x^{p_x}\rho_y^{p_x-2})^{1/(1-p_x)}$  to ensure the right-hand side of \eqref{eq:dx-thetax-p2} is zero.
    In either case, $\theta_x-\metricRhoY^{2-p_x}(p_x^{p_x}\underline{\omega}\zeta_x^{p_x-1})^{-1}\ge0$ and hence
    \begin{equation}
        \label{eq:dx-estimate-p2}
        D_x \ge
        -\frac{\lambda_x}2\norm{\nextx-\thisx}^2
        -(p_x-1)\inv\omega_i\zeta_x \norm{\nexty-\realopty}^2.
    \end{equation}

    Analogously, from \eqref{eq:k-nonlinear-ky-p2} and Cauchy's inequality, 
    \begin{equation*}
        \begin{aligned}[t]
            D_y&\ge
            \theta_y\norm{K_x(\thisx,\realopty)-K_x(\thisx,\thisy)-K_{xy}(\thisx,\thisy)(\realopty-\thisy)}^{p_y}
            -\frac{\lambda_y}2\norm{\nexty-\thisy}^2
            \\\MoveEqLeft[-1]
            -\omega_i\norm{\nextx-\realoptx}\norm{K_x(\thisx,\realopty)-K_x(\thisx,\thisy)-K_{xy}(\thisx,\thisy)(\realopty-\thisy)}.
        \end{aligned}
    \end{equation*}
    This has a structure similar to \eqref{eq:dx-pre-estimate-p2} with $\omega_i$ now as a multiplier. Hence, we apply a similar generalized Young's inequality to the last term with any $\zeta_y>0$. Noting that $\omega_i\le \overline{\omega}$, we use the following bound similar to \eqref{eq:dx-thetax-p2}:
    \begin{equation*}
        \theta_y-\norm{\nextx-\realoptx}^{2-p_y}\omega_i(p_y^{p_y}\zeta_y^{p_y-1})^{-1}
        \ge\theta_y-\metricRhoX^{2-p_y}\overline{\omega}(p_y^{p_y}\zeta_y^{p_y-1})^{-1}
        \ge 0.
    \end{equation*}
    The last inequality holds for any $\zeta_y>0$ if $p_y=1$ due to the assumed $\theta_y \ge \overline{\omega}\metricRhoX$ from \eqref{eq:scalar-gamma-rules-g}; otherwise, we set $\zeta_y \defeq (\theta_yp_y^{p_y}\rho_x^{p_y-2}\overline{\omega}^{-1})^{1/(1-p_y)}$.
    We then obtain that
    \begin{equation}
        \label{eq:dy-estimate-p2}
        D_y \ge
        -\frac{\lambda_y}2\norm{\nexty-\thisy}^2
        -(p_y-1)\omega_i\zeta_y \norm{\nextx-\realoptx}^2.
    \end{equation}    

    Combining \eqref{eq:dw-estimate}, \eqref{eq:dx-estimate-p2}, and \eqref{eq:dy-estimate-p2}, we can thus bound
    \begin{equation}
        \begin{aligned}[t]
            D&=
            \eta_i D_x+\eta_{i+1}D_y+\eta_{i+1}D_\omega
            +\eta_i(\gamma_G-\tilde\gamma_G-\xi_x)\norm{\nextx-\realoptx}^2
            \\
            \MoveEqLeft[-1]
            +\eta_{i+1}(\gamma_{F^*}-\tilde\gamma_{F^*}-\xi_y)\norm{\nexty-\realopty}^2
            \\
            &\ge
            \eta_{i+1}(\gamma_{F^*}-\tilde\gamma_{F^*}-\xi_y-(p_x-1)\zeta_x) \norm{\nexty-\realopty}^2
            -\eta_i\frac{\lambda_x}2\norm{\nextx-\thisx}^2
            \\\MoveEqLeft[-1]
            +\eta_{i}(\gamma_G-\tilde\gamma_G-\xi_x-(p_y-1)\zeta_y)\norm{\nextx-\realoptx}^2
            -\eta_{i+1}\frac{\lambda_y}2\norm{\nexty-\thisy}^2
            \\\MoveEqLeft[-1]
            -\eta_i \frac{L_{yx}}2(\omega_i+2)\metricRhoY \norm{\nextx-\thisx}^2
            \\
            & \ge
            -\eta_i\frac{\lambda_x+L_{yx}(\omega_i+2)\metricRhoY}2\norm{\nextx-\thisx}^2
            -\eta_{i+1}\frac{\lambda_y}2\norm{\nexty-\thisy}^2,
        \end{aligned}
    \end{equation}
    where in the final step, we have also used \eqref{eq:scalar-gamma-rules-p2} and the selected $\zeta_x$ and $\zeta_y$ if $p_x>1$ or $p_y>1$ or both. Thus, we obtained exactly the same lower bound as in \eqref{eq:final-d-estimate}. We then continue along the rest of the proof of \cref{thm:nonneg-penalty} to obtain the claim.
\end{proof}

It is worth observing that when $p_x \in (1, 2]$ or $p_y\in(1,2]$, the inequalities \eqref{eq:scalar-gamma-rules-p2} do not directly bound the respective $\metricRhoY$ or $\metricRhoX$. Hence, we do not need to initalize the corresponding variable locally, unlike when $p_x=1$ or $p_y=1$. On the other hand, sufficient strong convexity is required from the corresponding $G$ and $F^*$.

We start with the lemma ensuring that the iterates stay in the initial neighborhood of the saddle point.
\begin{corollary}
    The results of \cref{lemma:neighborhood-compatible-iterations} continue to hold if the corresponding conditions of \cref{thm:nonneg-penalty} are replaced with those in \cref{crl:nonneg-penalty-p2}.
\end{corollary}
\begin{proof}
    The proof repeats that of \cref{lemma:neighborhood-compatible-iterations}, applying \cref{crl:nonneg-penalty-p2} instead of \cref{thm:nonneg-penalty} in Step~2.
\end{proof}

We next extend the results of \cref{sec:convergence} to arbitrary choices of both $p_x \in [1,2]$ and $p_y \in [1,2]$. This mainly consists of verifying \eqref{eq:scalar-gamma-rules-g-p2} when $p_y \ne 1$ and \eqref{eq:scalar-gamma-rules-fstar-p2} when $p_x \ne 1$. Note that it is possible to take $p_x=1$ and $p_y \ne 1$, or vice versa, as long as the corresponding conditions are satisfied.
\begin{corollary}
    \label{thm:weak-convergence-p2}
    The results of \cref{thm:weak-convergence} continue to hold if \cref{ass:general}\,\ref{item:k-nonlinear} is replaced with \cref{ass:general-p2}\,\ref{item:k-nonlinear-p2} for some $p_x,p_y\in [1,2]$, where in case $p_y \in (1, 2]$, \eqref{eq:gamma-weak-g} is replaced with 
    \begin{subequations}%
        \label{eq:scalar-rules-p2-weak}
        \begin{align}
            \label{eq:scalar-gamma-rules-g-p2-weak}
            \xi_x &= \gamma_G  -
            \frac{p_y-1}{(\theta_yp_y^{p_y}(2\rho_x)^{p_y-2})^{\frac{1}{p_y-1}}},        
            \\
            \intertext{and in case $p_x \in (1, 2]$, \eqref{eq:gamma-weak-fstar} is replaced with}
            \label{eq:scalar-gamma-rules-fstar-p2-weak}
            \xi_y &= \gamma_{F^*} -
            \frac{p_x-1}{(\theta_xp_x^{p_x}(2\rho_y)^{p_x-2})^{\frac{1}{p_x-1}}}.
        \end{align}
    \end{subequations}
\end{corollary}
\begin{proof}
    Since conditions \eqref{eq:scalar-rules-p2-weak} are sufficient for \eqref{eq:scalar-gamma-rules-p2} with $\overline{\omega}=\underline{\omega}=1$ to hold, we can repeat the proof of \cref{thm:weak-convergence} replacing the references to \cref{thm:nonneg-penalty} by references to \cref{crl:nonneg-penalty-p2} up until \eqref{eq:weak-qi-estimate}.
    If $p_x>1$, we now obtain a lower bound on $d_i^x$ by arguing as in \eqref{eq:dx-pre-estimate-p2}--\eqref{eq:dx-thetax-p2} with $\realoptu$ replaced by $\bar{u}$. Specifically, using \eqref{eq:lipschitz-bound}, \cref{ass:general-p2}\,\ref{item:k-nonlinear-p2} at $\bar{u}$, and the generalized Young's inequality \eqref{eq:generalized-young}, we obtain for any $\zeta_x>0$ that
    \begin{equation*}
        \begin{aligned}[t]
            d_i^x&\le
            -\theta_x\norm{K_y(\bar{x},\nexty)-K_y(\nextx,\nexty)-K_{yx}(\nextx,\nexty)(\bar{x}-\nextx)}^{p_x}
            \\\MoveEqLeft[-1]
            +\norm{\nexty-\bar{y}}\norm{K_y(\bar{x},\nexty)-K_y(\nextx,\nexty)-K_{yx}(\nextx,\nexty)(\bar{x}-\nextx)}
            \\\MoveEqLeft[-1]
            +\frac{\lambda_x}2\norm{\nextx-\thisx}^2
            -\frac{p_y-1}{(\theta_yp_y^{p_y}(2\rho_x)^{p_y-2})^{\frac{1}{p_y-1}}}\norm{\nextx-\bar{x}}^2
            \\
            &\le 
            \left(\frac{\norm{\nexty-\bar{y}}^{2-p_x}}{p_x^{p_x}\zeta_x^{p_x-1}}-\theta_x\right)
            \norm{K_y(\bar{x},\nexty)-K_y(\nextx,\nexty)-K_{yx}(\nextx,\nexty)(\bar{x}-\nextx)}^{p_x}
            \\
            \MoveEqLeft[-1]
            +(p_x-1)\zeta_x \norm{\nexty-\bar{y}}^2
            +\frac{\lambda_x}2\norm{\nextx-\thisx}^2
            -\frac{p_y-1}{(\theta_yp_y^{p_y}(2\rho_x)^{p_y-2})^{\frac{1}{p_y-1}}}\norm{\nextx-\bar{x}}^2.
        \end{aligned}
    \end{equation*}
    Inserting $\zeta_x=(\theta_xp_x^{p_x}(2\rho_y)^{p_x-2})^{1/(1-p_x)}$ and $\norm{\nexty-\bar{y}} \le 2\metricRhoY$, we eliminate the first term on the right-hand side. 
    Likewise, if $p_y>1$, similar steps applied to $d_i^y$ result in 
    \begin{equation*}
        d_i^y\le
        (p_y-1)\zeta_y \norm{\nextx-\bar{x}}^2
        +\frac{\lambda_y}{2}\norm{\nexty-\thisy}^2
        -\frac{p_x-1}{(\theta_xp_x^{p_x}(2\rho_y)^{p_x-2})^{\frac{1}{p_x-1}}}\norm{\nexty-\bar{y}}^2
    \end{equation*}
    for $\zeta_y=(\theta_yp_y^{p_y}(2\rho_x)^{p_y-2})^{1/(p_y-1)}$. 
    Using $\norm{\nextu-\thisu}\to0$ and the selection of $\zeta_x$ and $\zeta_y$, we then obtain the desired
    estimate $\limsup_{i\rightarrow\infty}~q_i\defeq \limsup_{i\rightarrow\infty}~(d_i^x + d_i^y + O(\norm{\nextu-\thisu}))\le 0$.
\end{proof}

\begin{corollary}
    \label{thm:acceleration-nlpdhgm-p2}
    The results of \cref{thm:acceleration-nlpdhgm} continue to hold if \cref{ass:general}\,\ref{item:k-nonlinear} is replaced with \cref{ass:general-p2}\,\ref{item:k-nonlinear-p2} for some $p_x,p_y\in [1,2]$, where in case $p_y \in (1, 2]$, \eqref{eq:gamma-acc-g} is replaced for some $\tilde\gamma_G > 0$ with 
    \begin{subequations}%
        \label{eq:scalar-rules-p2-acc}
        \begin{align}
            \label{eq:scalar-gamma-rules-g-p2-acc}
            \xi_x &= \gamma_G - \tilde\gamma_G
            - \frac{p_y-1}{(\theta_yp_y^{p_y}(\rho_x)^{p_y-2})^{\frac{1}{p_y-1}}},        
            \\
            \intertext{and in case $p_x \in (1, 2]$,  \eqref{eq:gamma-acc-fstar} is replaced with}
            \label{eq:scalar-gamma-rules-fstar-p2-acc}
            \xi_y &= \gamma_{F^*}
            - \frac{p_x-1}{(\theta_xp_x^{p_x}(\rho_y)^{p_x-2})^{\frac{1}{p_x-1}}}.
        \end{align}
    \end{subequations}
\end{corollary}
\begin{proof}
    Conditions \eqref{eq:scalar-rules-p2-acc} are sufficient for \eqref{eq:scalar-gamma-rules-p2} with $\overline{\omega}=\underline{\omega}=1$ to hold; therefore, we can repeat the proof of \cref{thm:acceleration-nlpdhgm} replacing the references to \cref{thm:nonneg-penalty} by references to \cref{crl:nonneg-penalty-p2}.
\end{proof}

\begin{corollary}
    \label{thm:linear-convergence-p2}
    The results of \cref{thm:linear-convergence} continue to hold if \cref{ass:general}\,\ref{item:k-nonlinear} is replaced with \cref{ass:general-p2}\,\ref{item:k-nonlinear-p2} for some $p_x,p_y\in [1,2]$, where in case $p_y \in (1, 2]$, \eqref{eq:gamma-lin-g} is replaced for some $\tilde\gamma_G>0$ with 
    \begin{subequations}%
        \label{eq:scalar-rules-p2-lin}
        \begin{align}
            \label{eq:scalar-gamma-rules-g-p2-lin}
            \xi_x &= \gamma_G - \tilde\gamma_G
            - \frac{p_y-1}{(\theta_yp_y^{p_y}(\rho_x)^{p_y-2}\omega^{-1})^{\frac{1}{p_y-1}}},        
            \\
            \intertext{and in case $p_x \in (1, 2]$, \eqref{eq:gamma-lin-fstar} is replaced for some $\tilde\gamma_{F^*}>0$ with}
            \label{eq:scalar-gamma-rules-fstar-p2-lin}
            \xi_y & = \gamma_{F^*} - \tilde\gamma_{F^*}
            - \frac{p_x-1}{(\omega\theta_xp_x^{p_x}(\rho_y)^{p_x-2})^{\frac{1}{p_x-1}}}.
        \end{align}
    \end{subequations}
\end{corollary}
\begin{proof}
    Conditions \eqref{eq:scalar-rules-p2-lin} are sufficient for \eqref{eq:scalar-gamma-rules-p2} with $\overline{\omega}=\underline{\omega}=\omega$ to hold; therefore, we can repeat the proof of \cref{thm:linear-convergence} replacing the references to \cref{thm:nonneg-penalty} by references to \cref{crl:nonneg-penalty-p2}.
\end{proof}

\begin{corollary}
    The results of \cref{pr:locality-infinite} continue to hold if the corresponding conditions of \cref{thm:weak-convergence}, \ref{thm:acceleration-nlpdhgm}, or \ref{thm:linear-convergence} are replaced with those in \cref{thm:weak-convergence-p2}, \ref{thm:acceleration-nlpdhgm-p2}, or \ref{thm:linear-convergence-p2}.
\end{corollary}

\begin{proof}
    The proof repeats that of \cref{pr:locality-infinite}.
\end{proof}

\section{Verification of conditions for step function presentation and Potts model}
\label{app:step-potts}

Throughout this section, we set $\rho(t) \defeq 2t-t^2$ and $\componentk(x, y) \defeq \rho(\iprod{x}{y})$ for $x, y \in \R^m$. Then $\rho'(t)=2(1-t)$ so that
\begin{subequations}
    \label{eq:step-phi-derivatives}
    \begin{align}
        \componentk_x(x,y)&=2y(1-\iprod{y}{x})&&\text{and}
        &
        \componentk_{xy}(x,y)&=2(I-\iprod{y}{x}I-y \otimes x),
        \\
        \componentk_y(x,y)&=2x(1-\iprod{x}{y})&&\text{and}
        &
        \componentk_{yx}(x,y)&=2(I-\iprod{x}{y}I-x \otimes y),
    \end{align}
\end{subequations}
where $a\otimes b\in\R^{n\times n}$ is the tensor product between two vectors $a$ and $b$, producing a matrix of all the combinations of products between the entries.

The following lemma verifies \cref{ass:general} for $K=\kappa$.
\begin{lemma}
    \label{lemma:rho-kcondition}
    Let $R_K>2$, and suppose $\realoptx,\realopty \in \R^m$ for $m \ge 1$ with
    \begin{equation}
        \label{eq:rho-kcondition-opt-assumption}
        0 \le \iprod{\realoptx}{\realopty} I + \realoptx \otimes \realopty \le 2I.
    \end{equation}
    Then the function $K=\componentk$ defined above satisfies \cref{ass:general} for some $\theta_x, \theta_y >0$ and some $\metricRhoX,\metricRhoY>0$ dependent on $R_K$ with
    \begin{align*}
        L_x(y) & = 2\abs{y}_2^2,
        &
        L_y(x) & = 2\abs{x}_2^2,
        &
        L_{yx} & = 4(\abs{\realopty}_2+\metricRhoY),
    \end{align*}
    as well as the constants $\xi_x,\xi_y \in \R$, $\lambda_x,\lambda_y \ge 0$ satisfying $\lambda_x\xi_x > 2(\lambda_x+\abs{\realopty}_2^2)\abs{\realopty}_2^2$, $\xi_y > 0$, and $\lambda_y > \abs{\realoptx}_2^2$.
\end{lemma}

\begin{proof}
    First, \cref{ass:general}\,\ref{item:partial-gradk} holds everywhere since $K\in C^\infty(\R^m)$.
    To verify \cref{ass:general}\,\ref{item:lipschitz-gradk}, we observe using \eqref{eq:step-phi-derivatives} that
    \begin{subequations}%
        \label{eq:step-phi-derivatives-diff}
        \begin{align}
            \componentk_x(x', y)-\componentk_x(x, y) & = 2(y \otimes y)(x-x'),
            \\
            \componentk_{xy}(x, y')-\componentk_{xy}(x, y) & = 2\iprod{y-y'}{x}I+2(y-y') \otimes x,
            \\
            \componentk_y(x, y')-\componentk_y(x, y) & = 2(x \otimes x)(y-y'), 
            \\
            \componentk_{yx}(x', y)-\componentk_{yx}(x, y) & = 2\iprod{x-x'}{y}I+2(x-x') \otimes y.
        \end{align}
    \end{subequations}%
    Hence $L_x$, $L_y$, and $L_{yx}$ are as claimed.

    To verify \cref{ass:general}\,\ref{item:bounded-gradk}, we first of all observe using \eqref{eq:rho-kcondition-opt-assumption} that
    \begin{equation*}
        \abs{\componentk_{xy}(\realoptx, \realopty)}_2
        =
        2\abs{I-\iprod{\realopty}{\realoptx}I-\realopty \otimes \realoptx}_2
        \le
        2.
    \end{equation*}
    Therefore $\sup_{(x, y) \in \B(\realoptx, \metricRhoX) \times \B(\realopty, \metricRhoY)} \abs{\componentk_{xy}(x, y)}_2 \le R_K$ for some $\metricRhoX, \metricRhoY>0$ dependent on $R_K>2$.

    Finally, to verify \cref{ass:general}\,\ref{item:k-nonlinear}, we start with \eqref{eq:k-nonlinear-kx}, i.e.,
    \begin{equation*}
        \iprod{\componentk_x(x',\realopty)-\componentk_x(\realoptx,\realopty)}{x-\realoptx}
        +\xi_x\abs{x-\realoptx}_2^2
        \ge
        \theta_x\abs{\componentk_y(\realoptx,y)-\componentk_y(x,y)-\componentk_{yx}(x,y)(\realoptx-x)}_2
        -\frac{\lambda_x}2\abs{x-x'}_2^2.
    \end{equation*}
    Expanding the equation using \eqref{eq:step-phi-derivatives}, \eqref{eq:step-phi-derivatives-diff}, and
    \begin{equation*}
        \begin{aligned}
            \componentk_y(\realoptx,y)&-\componentk_y(x,y)-\componentk_{yx}(x,y)(\realoptx-x)
            \\&
            =2\realoptx(1-\iprod{\realoptx}{y})-2x(1-\iprod{x}{y})
            -2(I-\iprod{x}{y}I-x \otimes y)(\realoptx-x)
            \\&
            =2[\iprod{x}{y}x-\iprod{\realoptx}{y}\realoptx+(\iprod{x}{y}I+x \otimes y)(\realoptx-x)]
            \\&
            =2[\iprod{x-\realoptx}{y}\realoptx+(x \otimes y)(\realoptx-x)]
            \\&
            =-2((\realoptx-x) \otimes y)(\realoptx-x),
        \end{aligned}
    \end{equation*}
    we require that
    \begin{equation}
        \label{eq:k-nonlinear-kx-rho}
        2\iprod{\realoptx-x'}{x-\realoptx}_{\realopty \otimes \realopty}
        +\xi_x\abs{x-\realoptx}_2^2
        \ge
        2\theta_x\abs{y}_2\abs{x-\realoptx}_2^2
        -\frac{\lambda_x}2\abs{x-x'}_2^2.
    \end{equation}
    Taking any $\alpha>0$, this will hold by Cauchy's and Young's inequalities if
    $\xi_x \ge (2+\alpha)\abs{\realopty}_2^2 + 2\theta_x\abs{y}_2$ and
    $\lambda_x/2 \ge \inv\alpha\abs{\realopty}_2^2$.
    If $\abs{\realopty}_2=0$, clearly these hold for some $\alpha,\theta_x>0$.
    Otherwise, solving $\alpha$ from the latter as an equality, i.e., taking $\alpha=2\inv\lambda_x\abs{\realopty}_2^2$, the former holds if $\xi_x \ge 2(1+\inv\lambda_x\abs{\realopty}_2^2)\abs{\realopty}_2^2 + 2\theta_x\abs{y}_2$.
    If $\lambda_x\xi_x > 2(\lambda_x+\abs{\realopty}_2^2)\abs{\realopty}_2^2$, this holds for some $\theta_x,\rho_x,\rho_y>0$ in a neighborhood $\B(\realoptx, \metricRhoX) \times \B(\realopty, \metricRhoY)$ of ($\realoptx, \realopty)$.

    It remains to verify \eqref{eq:k-nonlinear-ky}, i.e.,
    \begin{multline*}
        \iprod{\componentk_y(x,y)-\componentk_y(x,y')+\componentk_y(\realoptx,\realopty)-\componentk_y(\realoptx,y)}{y-\realopty}
        +\xi_y\abs{y-\realopty}_2^2
        \\
        \ge\theta_y\abs{\componentk_x(x',\realopty)-\componentk_x(x',y')-\componentk_{xy}(x',y')(\realopty-y')}_2
        -\frac{\lambda_y}2\abs{y-y'}_2^2.
    \end{multline*}
    Again, using \eqref{eq:step-phi-derivatives} and \eqref{eq:step-phi-derivatives-diff} we expand this as
    \begin{equation*}
        2\iprod{y'-y}{y-\realopty}_{x \otimes x}+2\abs{y-\realopty}_{\realoptx \otimes \realoptx}^2
        +\xi_y\abs{y-\realopty}_2^2
        \ge
        2\theta_y\abs{x'}_2\abs{y'-\realopty}_2^2
        -\frac{\lambda_y}2\abs{y-y'}_2^2.
    \end{equation*}
    Rearranging the $\theta_y$-term, we see that this holds if
    \begin{equation*}
        2\iprod{y'-y}{y-\realopty}_{x \otimes x-2\theta_y\abs{x'}_2I}
        +2\abs{y-\realopty}^2_{\realoptx \otimes \realoptx}+(\xi_y-2\theta_y)\abs{x'}_2\abs{y-\realopty}_2^2
        \ge
        \left(2\theta_y\abs{x'}_2-\frac{\lambda_y}{2}\right)\abs{y'-y}_2^2.
    \end{equation*}
    Rearranging and estimating the first term as
    \begin{equation*}
        \begin{aligned}
            2\iprod{y'-y}{y-\realopty}_{x \otimes x-2\theta_y\abs{x'}_2I}
            &
            =
            2\iprod{y'-y}{x}\iprod{y-\realopty}{x}-4\theta_y\abs{x'}_2\iprod{y'-y}{y-\realopty}
            \\ &
            \ge
            -2\abs{y-\realopty}^2_{x \otimes x}-\frac{1}{2}\abs{y'-y}^2_{x \otimes x}
            -4\theta_y\abs{x'}_2\abs{y'-y}_2^2-\theta_y\abs{x'}_2\abs{y-\realopty}_2^2
        \end{aligned}
    \end{equation*}
    and then using Young's inequality on both parts, we obtain the condition
    \begin{equation*}
        2\left(\abs{y-\realopty}^2_{\realoptx \otimes \realoptx}-\abs{y-\realopty}^2_{x \otimes x}\right)+(\xi_y-3\theta_y)\abs{x'}_2\abs{y-\realopty}_2^2
        \ge
        \left(\frac{1}{2}\abs{x}_2^2 + 6\theta_y\abs{x'}_2-\frac{\lambda_y}{2}\right)\abs{y'-y}_2^2.
    \end{equation*}
    If $\xi_y > 0$ and $\lambda_y > \abs{\realoptx}_2^2$, this holds for some $\theta_y,\rho_y,\rho_x>0$ in  $\B(\realoptx, \metricRhoX) \times \B(\realopty, \metricRhoY)$.
\end{proof}

We comment on the condition \eqref{eq:rho-kcondition-opt-assumption} on the primal--dual solutions pair $\realoptx, \realopty \in \R$. First, for $m=1$, this condition reduces to $\realoptx\realopty \in [0, 1]$. This is necessarily satisfied in the case of the step function (where $f^*=\delta_{[0, \infty)}$) and in the case of the $\ell^0$ function (where $f^*=0$) as in both cases, $\realoptx\realopty \in \{0, 1\}$ by the dual optimality condition $\componentk_y(\realoptx, \realopty) \in \subdiff f^*(\realopty)$. Furthermore, if we take $f^*_\gamma=\frac{\gamma}{2}\abs{\freevar}_2^2$ for some $\gamma \ge 0$, then for any $m\geq 1$ the dual optimality condition reads $2\realoptx(1-\iprod{\realoptx}{\realopty})=\gamma\realopty$, i.e,  $\realopty=2\realoptx(\gamma+2\abs{\realoptx}_2^2)^{-1}$, for which \eqref{eq:rho-kcondition-opt-assumption} is easily verified.

The following lemma shows that \cref{ass:general} remains valid if we include a linear operator in the primal component.
\begin{lemma}
    \label{lemma:oper-combo-kcondition}
    Let $K(x, y)=\tilde K(Ax, y)$ for some $A \in \linear(X; Z)$ and $\tilde K \in C^1(Z \times Y)$ on Hilbert spaces $X, Y, Z$.
    Suppose $\tilde K$ satisfies \cref{ass:general} at $(\realoptz, \realopty) \defeq (A \realoptx, \realopty)$. Mark the corresponding constants with a tilde: $\tilde L_z$, $\tilde R_K$, and so on.
    Then $K$ satisfies \cref{ass:general} with $R_K \defeq \tilde R_K \norm{A}$; $\xi_x=\norm{A}\tilde\xi_z$, $\xi_y=\tilde\xi_y$; $\lambda_x=\norm{A}\tilde\lambda_z$, $\lambda_y=\tilde\lambda_y$; $\theta_x=\tilde\theta_z$, $\theta_y=\tilde\theta_y\norm{A}^{-1}$; $\metricRhoX=\norm{A}^{-1}\tilde\rho_x$, and $\metricRhoY=\tilde\rho_y$ as well as
    \begin{align}
        \label{eq:k-tildek-l-est}
        L_x(y)&=\norm{A}^2\tilde L_z(y),
        &
        L_y(x)&=\tilde L_y(Ax),
        &
        L_{yx}&=\norm{A}^2\tilde L_{yz}.
    \end{align}
\end{lemma}

\begin{proof}
    Observe first of all that by the chain rule,
    \begin{align*}
        K_x(x, y) &= A^* \tilde K_z(Ax, y),
        &
        K_y(x, y) &= \tilde K_y(Ax, y),
        &
        K_{xy}(x,y)&=A^*\tilde K_{zy}(Ax, y),
    \end{align*}
    and hence \cref{ass:general}\,\cref{item:partial-gradk} holds for $K$ if it holds for $\tilde K$.

    Let now \cref{ass:general}\,\ref{item:lipschitz-gradk} hold for $\tilde K$ with $\tilde L_x$, $\tilde L_y$, and $\tilde L_{yx}$.
    Observing that
    \begin{equation}
        \label{eq:oper-combo-kcondition-neigh}
        A\B(\realoptx, \metricRhoX) \times \B(\realopty, \metricRhoY) \subset \B(\realoptz, \tilde\rho_x) \times \B(\realopty,\tilde\rho_y),
    \end{equation}
    \cref{ass:general}\,\ref{item:lipschitz-gradk} thus also holds with the function of \eqref{eq:k-tildek-l-est}.
    Similarly in \cref{ass:general}\,\ref{item:bounded-gradk}, we can take $R_K \defeq \tilde R_K \norm{A}$.

    Finally, we expand \cref{ass:general}\,\ref{item:k-nonlinear} for $K$ as
    \begin{align*}
        &\begin{aligned}
            &\iprod{\tilde K_z(z',\realopty)-\tilde K_z(\realoptz,\realopty)}{z-\realoptz}
            +\xi_x\norm{x-\realoptx}^2
            \\
            \MoveEqLeft[-1]\ge
            \theta_x\norm{\tilde K_y(\realoptz,y)-\tilde K_y(z,y)-\tilde K_{yz}(z,y)(\realoptz-z)}
            -\frac{\lambda_x}2\norm{x-x'}^2
        \end{aligned}
        \shortintertext{and}
        &\begin{aligned}
            &\iprod{\tilde K_y(z,y)-\tilde K_y(z,y')+\tilde K_y(\realoptz,\realopty)-\tilde K_y(\realoptz,y)}{y-\realopty}
            +\xi_y\norm{y-\realopty}^2
            \\
            \MoveEqLeft[-1]\ge
            \theta_y\norm{A^*[\tilde K_z(z',\realopty)-\tilde K_z(z',y')-\tilde K_{zy}(z',y')(\realopty-y')]}
            -\frac{\lambda_y}2\norm{y-y'}^2,
        \end{aligned}
    \end{align*}%
    where $z=Ax$, $z'=Ax'$, and $\realoptz=A\realoptx$.
    Since $\norm{z-z'} \le \norm{A}\norm{x-x'}$, etc., this follows from \cref{ass:general}\,\ref{item:k-nonlinear} for $\tilde K$ with the constants as claimed.
\end{proof}

Applying this lemma to $\tilde K(z, y)=\sum_{k=1}^n \componentk(z_k, y_k)$, we can thus lift the scalar estimates for $K=\kappa$ as in \eqref{eq:step-phi-derivatives} to the corresponding estimates on $K(x, y) \defeq \sum_{k=1}^n \componentk([D_h x]_k, y_k)$ as used in the Potts model example.

\bibliographystyle{jnsao}
\bibliography{general}

\end{document}

%% file: symbols.tex
\newcommand{\defeq}{:=}
\newcommand{\set}[1]{\{#1\}}
\newcommand{\grad}{\nabla}
\newcommand{\inv}[1]{#1^{-1}}
\newcommand{\field}[1]{\mathbb{#1}}
\newcommand{\N}{\field{N}}
\newcommand{\R}{\field{R}}
\newcommand{\extR}{\overline \R}
\newcommand{\abs}[1]{|#1|}
\newcommand{\norm}[1]{\|#1\|}
\newcommand{\iprod}[2]{\langle #1,#2\rangle}
\newcommand{\subdiff}{\partial}
\newcommand{\weakto}{\mathrel{\rightharpoonup}}

\newcommand{\freevar}{\,\boldsymbol\cdot\,}
\newcommand{\isect}\cap

\newcommand{\B}{\vmathbb{B}}

\def\tilde{\widetilde}

\newcommand{\setto}{\rightrightarrows}

\def\Rinf{\overline \R}
\def\linear{\mathcal{L}}

\newcommand{\linearLArrow}[1][]{\linear_{\triangleleft\ifx&#1&\else,\,#1\fi}}
\newcommand{\linearLArrowSpecial}[1][]{\linear^{\star}_{\triangleleft\ifx&#1&\else,\,#1\fi}}

\newcommand{\prox}{\mathrm{prox}}
\newcommand{\proj}{\mathrm{proj}}

\def\realopt#1{\widehat #1}
\def\this#1{#1^i}
\def\nexxt#1{#1^{i+1}}
\def\overnext#1{\bar #1^{i+1}}

\def\realoptu{{\realopt{u}}}

\def\realoptx{{\realopt{x}}}
\def\realoptz{{\realopt{z}}}
\def\realopty{{\realopt{y}}}

\def\nextu{\nexxt{u}}
\def\nextx{\nexxt{x}}

\def\nexty{\nexxt{y}}

\def\thisu{\this{u}}
\def\thisx{\this{x}}

\def\thisy{\this{y}}

\def\overnextu{\overnext{u}}
\def\overnextx{\overnext{x}}

\def\tauTest{\phi}

\def\sigmaTest{\psi}

\def\Space{U}

\newcommand{\Test}{Z}
\newcommand{\Precond}{M}
\newcommand{\Step}{W}

\DeclareFontFamily{U}{mathx}{\hyphenchar\font45}
\DeclareFontShape{U}{mathx}{m}{n}{<-> mathx10}{}
\DeclareSymbolFont{mathx}{U}{mathx}{m}{n}
\DeclareMathAccent{\widebar}{0}{mathx}{"73}

\newcommand{\Penalty}{\Delta}

\def\neighu{\mathcal{U}}

\def\neighx{\mathcal{X}}
\def\neighy{\mathcal{Y}}

\def\bar{\widebar}

\def\metricRhoX{\rho_x}
\def\metricRhoY{\rho_y}
\newcommand{\localRhoX}[1][]{\if\relax\detokenize{#1}\relax r_x\else r_{x,#1}\fi}
\newcommand{\localRhoY}[1][]{\if\relax\detokenize{#1}\relax r_y\else r_{y,#1}\fi}

\def\componentk{\kappa}
\def\muConst{\mu}

%% file: blobs_fun.tikz
\begin{tikzpicture}
    \begin{axis}[%
        width=\linewidth,
        xmode=log,
        xminorticks=true,
        ymode=log,
        yminorticks=true,
        axis x line*=bottom,
        axis y line*=left,
        legend style={legend pos=north east,legend cell align=left,align=left,draw=none,fill=none}
        ]
        \addplot [color=Blues-K,line width=1.5pt]
            table[header=false,y index=1]{blobs_aniso.txt};
        \addlegendentry{$p=1$};
        \addplot [color=Greens-K,line width=1.5pt]
            table[header=false,y index=1]{blobs_iso.txt};
        \addlegendentry{$p=\infty$};
    \end{axis}
\end{tikzpicture}

%% file: blobs_err.tikz
\begin{tikzpicture}
    \begin{axis}[%
        width=\linewidth,
        scaled x ticks=false,
        xtick = {1e4,1e5,2e5,3e5,4e5,5e5},
        x tick label style={/pgf/number format/sci},
        xminorticks=true,
        ymode=log,
        yminorticks=true,
        axis x line*=bottom,
        axis y line*=left,
        legend style={legend pos=north east,legend cell align=left,align=left,draw=none,fill=none}
        ]
        \addplot [color=Blues-K,line width=1.5pt]
            table[header=false,y index=2]{blobs_aniso.txt};
        \addlegendentry{$p=1$};
        \addplot [color=Greens-K,line width=1.5pt]
            table[header=false,y index=2]{blobs_iso.txt};
        \addlegendentry{$p=\infty$};
    \end{axis}

\end{tikzpicture}